\documentclass[11pt]{article}
\usepackage{amssymb}
\usepackage{graphicx}
\usepackage{xcolor} 
\usepackage{tensor}
\usepackage{fullpage} 
\usepackage{amsmath}
\usepackage{amsthm}
\usepackage{verbatim}
\usepackage{inputenc}

\usepackage{enumitem}
\setlist[enumerate]{leftmargin=1.5em}
\setlist[itemize]{leftmargin=1.5em}

\usepackage{seqsplit}

\definecolor{green}{rgb}{0,0.8,0} 

\newtheorem{theorem}{Theorem}[section]
\newtheorem{corollary}[theorem]{Corollary}
\newtheorem{lemma}[theorem]{Lemma}
\newtheorem{proposition}[theorem]{Proposition}

\newtheorem{innercustomthm}{Theorem}
\newenvironment{customthm}[1]
{\renewcommand\theinnercustomthm{#1}\innercustomthm}
{\endinnercustomthm}
\theoremstyle{definition}
\newtheorem{definition}[theorem]{Definition}
\newtheorem{example}[theorem]{Example}
\theoremstyle{remark}
\newtheorem{remark}[theorem]{Remark}

\numberwithin{equation}{section}
\newcommand{\nrm}[1]{\Vert#1\Vert}

\newcommand{\nnrm}[1]{{\vert\kern-0.25ex\vert\kern-0.25ex\vert #1 
		\vert\kern-0.25ex\vert\kern-0.25ex\vert}}

\newcommand{\lap}{\Delta}

\newcommand{\rd}{\partial}
\newcommand{\nb}{\nabla}

\newcommand{\lmb}{\lambda}

\newcommand{\tht}{\theta}

\newcommand{\omg}{\omega}

\newcommand{\bbR}{\mathbb R}
\newcommand{\bbS}{\mathbb S}
\newcommand{\bbT}{\mathbb T}

\newcommand{\calG}{\mathcal G}

\newcommand{\calO}{\mathcal O}

\newcommand{\calR}{\mathcal R}

\newcommand{\frka}{\mathfrak a}

\newcommand{\mrI}{\mathrm{I}}
\newcommand{\mrII}{\mathrm{II}}

\newcommand{\tomg}{\tilde{\omg}}				

\begin{document}
	
	\title{The incompressible Euler equations under octahedral symmetry: singularity formation in a fundamental domain}
	\author{Tarek M. Elgindi\thanks{Department of Mathematics, UC San Diego, San Diego, USA. E-mail: telgindi@ucsd.edu}\and In-Jee Jeong\thanks{School of Mathematics, Korea Institute for Advanced Study, Seoul, Republic of Korea. E-mail: ijeong@kias.re.kr}}
	\date{\today}
	
	
	

	\maketitle
	
	
	\begin{abstract}
		We consider the 3D incompressible Euler equations in vorticity form in the following fundamental domain for the octahedral symmetry group: $\{ (x_1,x_2,x_3): 0<x_3<x_2<x_1 \}.$ In this domain, we prove local well-posedness for $C^\alpha$ vorticities not necessarily vanishing on the boundary with any $0<\alpha<1$, and establish finite-time singularity formation within the same class for smooth and compactly supported initial data. The solutions can be extended to all of $\bbR^3$ via a sequence of reflections, and therefore we obtain finite-time singularity formation for the 3D Euler equations in $\bbR^3$ with bounded and piecewise smooth vorticities. 
	\end{abstract}
	
	\tableofcontents

\section{Introduction}

In this paper, we consider solutions of the 3D incompressible Euler equations in $\bbR^3$ which are invariant under a specific group of rotations and reflections. In terms of the velocity $u : \bbR \times \bbR^3 \rightarrow \bbR^3$, the 3D Euler equations have the following form: \begin{equation}  \label{eq:Euler}
\left\{
\begin{aligned} 
\rd_t u + u\cdot\nb u = -\nabla p, \\
\nabla \cdot u = 0. 
\end{aligned}
\right.
\end{equation} We shall use the vorticity form of the above, which is obtained by defining $\nabla \times u = \omega$: \begin{equation}  \label{eq:Euler-vort}
\left\{
\begin{aligned} 
\rd_t \omg + u\cdot\nb\omg = [\nabla u]\omg, \\
u = \nabla\times (-\lap)^{-1}\omg. 
\end{aligned}
\right.
\end{equation} In the latter formulation, we need to impose that the initial vorticity is divergence-free. The Euler equations (in either form) respect rotation and reflection symmetries, which means that if the vorticity is invariant under such a symmetry, then the same property holds for the solution as well. 

\subsection{Octahedral Symmetry}

We recall that the octahedral symmetry group of $\mathbb{R}^3$, which will be denoted by $\mathcal{O}$, is generated by the following three rotations: \begin{equation*}
\begin{split}
\begin{cases}
P_1(x_1,x_2,x_3) := (x_1,-x_3,x_2) \\
P_2(x_1,x_2,x_3) := (x_3,x_2,-x_1) \\
P_3(x_1,x_2,x_3) := (-x_2,x_1,x_3). 
\end{cases}
\end{split}
\end{equation*} This group is isomorphic to the symmetry group of 4 elements. We shall consider a extended symmetry group $\tilde{\calO}$, which is generated by $\{ P_i, R_i \}_{i=1}^3$ where \begin{equation*}
\begin{split}
\begin{cases}
R_1(x_1,x_2,x_3) := (-x_1,x_2,x_3) \\
R_2(x_1,x_2,x_3) := (x_1,-x_2,x_3) \\
R_3(x_1,x_2,x_3) := (x_1,x_2,-x_3). 
\end{cases} 
\end{split}
\end{equation*} Note that $\tilde{\calO}$ has 48 elements. We shall fix the following \textit{fundamental domains} for $\calO$ and $\tilde{\calO}$: 
 \begin{equation}\label{eq:U}
\begin{split}
U = \{ (x_1,x_2,x_3) \in \bbR^3 : x_1 , x_2 > x_3 >0 \}, \quad \tilde{U} =  \{ (x_1,x_2,x_3) \in \bbR^3 : x_1 > x_2 > x_3 >0 \}. 
\end{split}
\end{equation} The latter domain $\tilde{U}$ is obtained by cutting the positive octant $(\mathbb{R}_+)^3$ into six identical pieces. Note that the intersection $\tilde{U} \cap \bbS^2$ is a spherical triangle with vertices having angles $\frac{\pi}{2}, \frac{\pi}{3}$, and $\frac{\pi}{4}$. Here, $\bbS^2$ is the unit sphere. We shall denote these vertices by $\frka_2$, $\frka_3$, and $\frka_4$, respectively. Explicitly, we have $\frka_2 = (\frac{1}{\sqrt{2}},\frac{1}{\sqrt{2}},0)$, $\frka_3 = (\frac{1}{\sqrt{3}},\frac{1}{\sqrt{3}},\frac{1}{\sqrt{3}})$, and $\frka_4 = (1,0,0)$; see Figure \ref{fig:fundamental}.

\begin{figure}
	\includegraphics[scale=0.5]{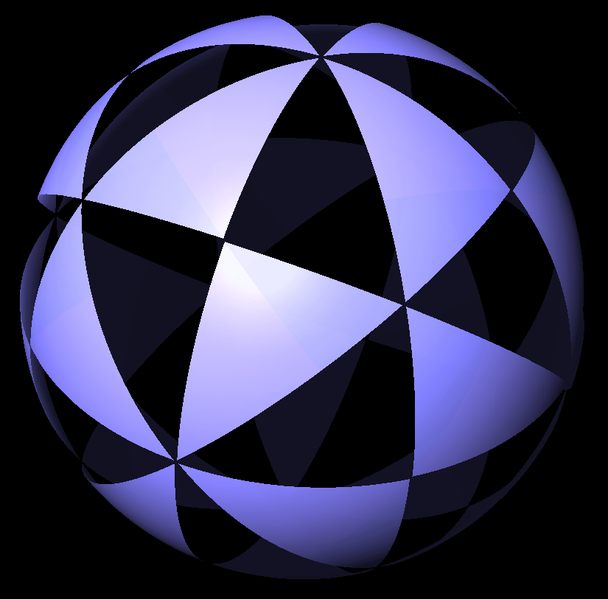} 
	\centering
	\caption{Extended octahedral symmetry group dividing $\mathbb{S}^2$ into 48 pieces. }
	\label{fig:octahedral}
\end{figure}

\begin{figure}
	\includegraphics[scale=0.4]{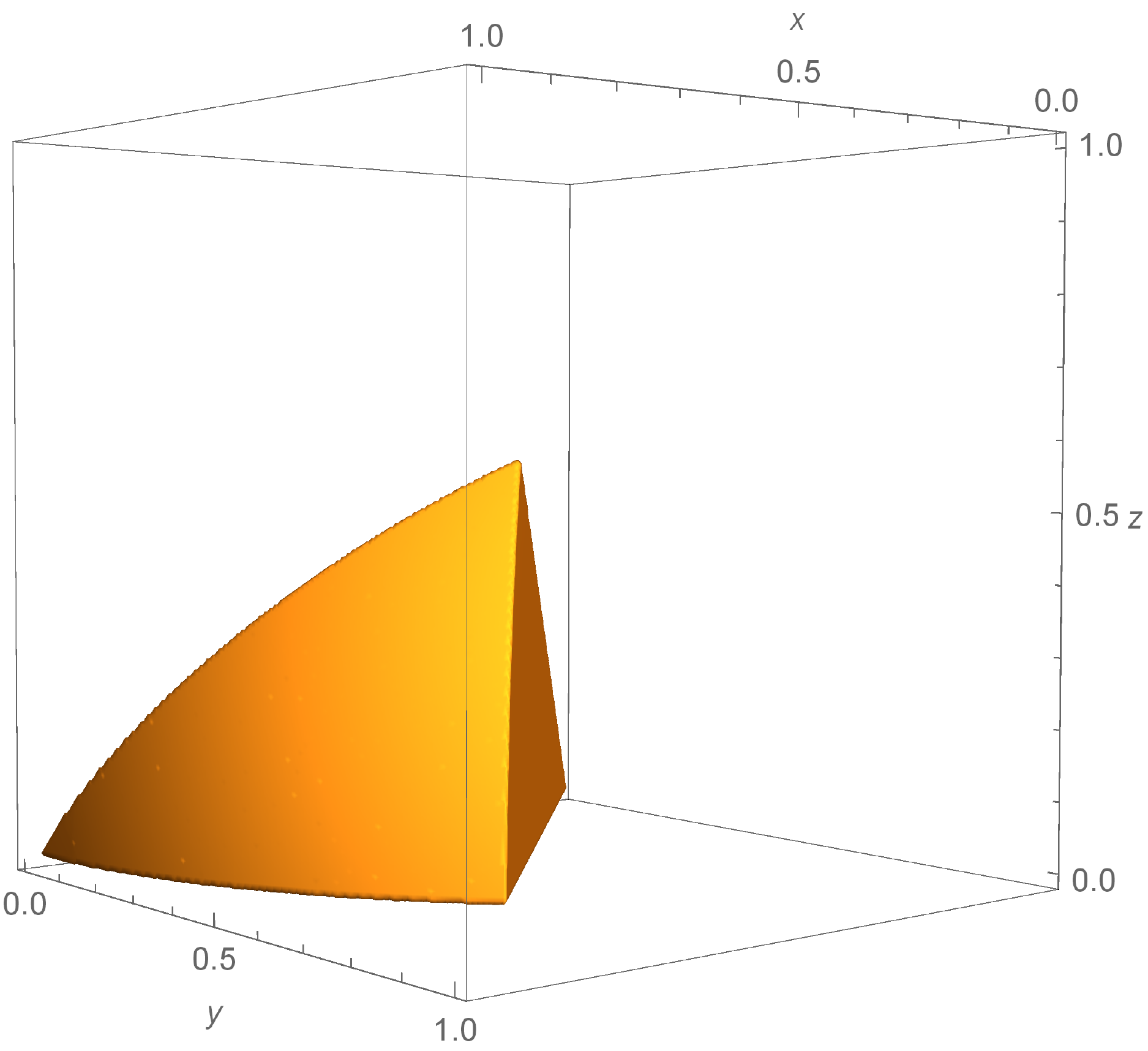} 
	\centering
	\caption{The region $\tilde{U} \cap \{|x|<1 \}$. The intersection $\tilde{U} \cap \{|x|=1\}$ is given by a spherical triangle with vertices $(\frac{1}{\sqrt{2}},\frac{1}{\sqrt{2}},0), (\frac{1}{\sqrt{3}},\frac{1}{\sqrt{3}},\frac{1}{\sqrt{3}})$, and $(1,0,0)$.}
	\label{fig:fundamental}
\end{figure} 

In the following definition, we make precise the notion of rotation and reflection invariance for a vector valued function. In the following, we shall only consider reflections across a hyperplane of $\bbR^3$ containing the origin $\textbf{0}$. 
\begin{definition}\label{def:symm}
	We say that a vector-valued function $f = (f_1,f_2,f_3)^T : \mathbb{R}^3 \rightarrow \mathbb{R}^3$ is (rotationally) symmetric with respect to a rotation $O$ of $\bbR^3$ if $f(Ox) = O(f(x))$ for any $x \in \bbR^3$. On the other hand, we say $f$ is even (odd, resp.) symmetric with respect to the reflection across a hyperplane $R$ if $f(Rx) = R(f(x))$ ($f(Rx) = -R(f(x))$, resp) for any $x \in \bbR^3$.
	
	Next, we say that for a group of rotations $\calG$, $f$ is even (odd, resp.) symmetric with respect to $\calG$ if $f$ is symmetric with respect to all rotations in $\calG$ and even (odd, resp.) symmetric with respect to all reflections in $\calG$. 
\end{definition} 

As it is well-known, the incompressible Euler equations respect rotation and reflection symmetries. We state this property specifically for $\calO$ and $\tilde{\calO}$:  \begin{proposition}
	Assume that $\omega_0$ is (odd, resp.) symmetric with respect to $\mathcal{O}$ ($\tilde{\calO}$, resp.) and belongs to a well-posedness class\footnote{This means that at least for some non-empty time interval, the solution exists uniquely in the given class.} to the $3D$ Euler equations. Then, the unique local-in-time solution $\omega(t,\cdot)$ stays symmetric with respect to $\mathcal{O}$ ($\tilde{\calO}$, resp.). 
\end{proposition}
\begin{proof}
	Take any rotation matrix $O$, and note that the relation \begin{equation*}
	\begin{split}
	O^{-1} K(Ox)O = K(x)
	\end{split}
	\end{equation*} holds for all $x \in \mathbb{R}^3$. This implies that $u = K*\omega$ is symmetric with respect to $\mathcal{O}$ whenever $\omega$ is. 	Next, given a vector $f$ symmetric with respect to $\mathcal{O}$, the gradient matrix $\nabla f$ is symmetric in the sense that \begin{equation*}
	\begin{split}
	O^{-1} \nabla f(Ox) O = \nabla f(x)
	\end{split}
	\end{equation*} holds for all $O \in \mathcal{O}$. In turn, this implies that when $\omega$ is symmetric, both $u\cdot\nabla\omega$ and $\omega\cdot\nabla u$ are symmetric as well. This shows that the $3D$ Euler equations respect rotational symmetries. In particular, symmetry breaking can only occur from non-uniqueness. 
	
	A similar argument shows that the odd symmetry with respect to a reflection propagates in time. Note that if $\omg(t,\cdot)$ is odd symmetric then $u(t,\cdot)$ is even symmetric. This finishes the proof. 
\end{proof}

Given a vector valued function $f$ defined on $\tilde{U}$, we define $\tilde{f}: \bbR^3 \rightarrow \bbR^3$ to be the unique extension of $f$ which is \textit{odd} symmetric with respect to $\tilde{\mathcal{O}}$. Note that strictly speaking $\tilde{f}$ is in general not well-defined on the whole of $\bbR^3$, but the set of such points belongs to a finite union of hyperplanes which has measure zero. In this paper, we shall consider the system \eqref{eq:Euler-vort} in the domain $\tilde{U} \subset \bbR^3$. Note that given a divergence-free vector field $\omega : \tilde{U} \rightarrow\bbR^3$, one can extend it as a function $\tilde{\omega} : \bbR^3 \rightarrow\bbR^3$ using reflections. Then, the corresponding velocity on $\tilde{U}$ has the representation \begin{equation*}
\begin{split}
u(x) = \frac{1}{4\pi} \int_{\bbR^3} \frac{x-y}{|x-y|^3} \times \tilde{\omega}(y)dy 
\end{split}
\end{equation*} (which is simply the convolution against the kernel for $\nabla\times(-\lap)^{-1}$ in $\bbR^3$). It is not difficult to see that this velocity satisfies the slip boundary condition $u \cdot n = 0$ on $\partial\tilde{U}$ with the unit normal vector $n$, which is the natural boundary condition for the Euler equations. Moreover, $u$ satisfies $\nabla\cdot u = 0$ and $\nabla\times u = \omega$ in $\tilde{U}$. With this reflection principle in mind, we shall view Euler solutions defined in $\tilde{U}$ also as solutions in $\bbR^3$. 

It seems that the set of symmetries $\tilde{\calO}$ coincide with those used in the so-called ``high symmetry flows'' of Kida \cite{K1}, although Kida consider symmetric solutions in $\bbT^3$ rather than in $\bbR^3$. We note that this high symmetry flows were numerically investigated in \cite{BoPe,Pe1,Pe2,Pe3} as a candidate for obtaining finite-time singularity formation. See \cite{NB} for some theoretical results in this direction. 

In this work, we view this group of symmetries as a generalization to 3D of certain symmetry groups for 2D flows. In two dimensions, the vorticity is a scalar, and one can consider the following symmetry groups for the vorticity: \begin{enumerate}
	\item[(1)] The group $\calO_{2D}$ generated by $P: v\mapsto v^\perp$ (rotation by $\frac{\pi}{2}$): then the vorticity satisfies $\omega(x) = \omega(x^\perp)$ for all $x \in \bbR^2$,
	\item[(2)] The group $\calR_{2D}$ generated by reflections $R_1$ and $R_2$: this imposes odd symmetry with respect to $x_1, x_2$ axes (``odd-odd") on the vorticity, i.e., $\omega(x_1,x_2) = -\omega(-x_1,x_2) = -\omega(x_1,-x_2)$ for all $x = (x_1,x_2) \in \bbR^2.$ 
	\item[(3)] The group $\tilde{\calO}_{2D}$ generated by $P$, $R_1$, and $R_2$. 
\end{enumerate} We may take the fundamental domains $U_{2D} = (\bbR_+)^2$ for groups $\calO_{2D}$ and $\calR_{2D}$ and $\tilde{U}_{2D} = \{ (x_1,x_2): 0<x_1<x_2 \}$ for $\tilde{\calO}_{2D}$. Note that one can already note some similarity between the group $\calO$ in $\calO_{2D}$, and also between $\tilde{\calO}$ and $\tilde{\calO}_{2D}$.

Recently, solutions to the 2D vorticity equation with symmetries of the above form were intensively studied, mainly towards the goal of achieving double exponential growth rate of the vorticity gradient (\cite{KS,Z,EJ1,IMY1,IMY2,E1,HR2,Xu}). This problem of exhibiting double exponential growth is sometimes regarded as the 2D version of the blow-up problem in 3D, since double exponential growth rate for the vorticity gradient is the best known upper bound for smooth solutions to the 2D Euler equations. The groundbreaking work of Kiselev and Severak \cite{KS} showed that the double exponential growth can be indeed achieved when the domain has physical boundary. In this work, the most important points were the odd-odd symmetry (item (2) in the above) and the vorticity \textit{not} vanishing on the boundary. These points together provided certain stability on the solutions, which allowed the vorticities to sustain their growth for all times. Then it became a natural question to ask what happens for solutions satisfying different types of symmetries, for instance the ones given by groups $\calO_{2D}$ and $\tilde{\calO}_{2D}$. In these cases, it turns out that the stability becomes so strong (in a sense that can be made precise; see \cite{EJ1,IMY1,IMY2}) that the double exponential rate of growth is unachievable, at least ``exactly" at the origin.

Indeed, the attempt to establish blow-up by considering solutions satisfying a group of symmetries is classical. If the symmetry group is continuous (rather than discrete), then it reduces the dimension of the system; in the context of the 3D Euler equations, one may assume that the solution is invariant under all rotations with respect to some fixed axis, which results in a 2D system commonly referred to as the axisymmetric 3D Euler equations. Very recently it was shown that blow-up for this axisymmetric system is possible, for $C^\alpha$ vorticity with $0<\alpha$ small \cite{E3}. Another continuous group of symmetry one can put for the 3D Euler equations goes by the name of ``stagnation point similitude'' ansatz, and a well-known work of Constantin \cite{Con} established blow-up in this case. We note that, however, in this ansatz the vorticity grows linearly at spatial infinity\footnote{So far no local well-posedness result is known for vorticities growing linearly at infinity.}, and it is unclear whether the blow-up will persist upon ``cutting off'' the vorticity at infinity. In principle, one can try to reduce the dimension even further, by looking at a one-dimensional subdomain which is left invariant by the Euler flow. In the case of axisymmetry, one can formally consider the 1D system defined on the symmetry axis, which is unfortunately not closed by itself (see \cite{ChaeNon,ChaeDCDS,ChaeJDE}). Still, in the work of Chae \cite{ChaeNon}, a very interesting result is proved which states that as long as the pressure has a positive second derivative on the symmetry axis, blow-up is bound to occur in finite time. Alternatively, one can take the axisymmetric domain to be a bounded cylinder and consider the reduced system on a vertical line contained in the boundary of the cylinder. While the system cannot be closed by itself again in this case, some formal models have appeared to describe the dynamics on this line and finite-time blow up has been established (\cite{HouLuo,HouLiu,CKY,CHKLSY}). In the same spirit, even when one is concerned with solutions satisfying a discrete set of symmetries, one can consider lower-dimensional equations posed on invariant subdomains for the corresponding symmetry group. This has been considered in \cite{CCW,CCW2} and some conditions which guarantees finite time blow-up were obtained, which has a similar flavor with \cite{ChaeNon}.

In view of these, the goal of this paper is to investigate the blow-up problem (and more generally growth of solutions with time) in 3D using vorticities which are symmetric with respect to either $\calO$ and $\tilde{\calO}$. Then, the Euler equations can be reduced to the fundamental domains $U$ and $\tilde{U}$ and in principle, one can consider vorticities which do not vanish on the boundary $\partial U$ and $\partial\tilde{U}$. In such domains, local well-posedness in classical function spaces is not trivial at all, since the boundary is not smooth but just Lipschitz continuous. Once it is achieved, the question of finite time singularity formation is legitimate, and we answer it in the affirmative in the case of $\tilde{\calO}$ and $\tilde{U}$. The initial data can be $C^\infty$ in  $\tilde{U}$ uniformly up to the boundary and compactly supported, with finite kinetic energy. Roughly speaking, this demonstrates that although the symmetry group $\tilde{\calO}$ provides a strong stability on the solution (analogously to the 2D case), it can still yield finite time blow-up. This is a clear manifestation of the following general principle, seemingly counter-intuitive: the more drastic growth one wants to prove, the more stability is required on the solution (see \cite{KRYZ,Kis,Kis2}). 

In the process of achieving the blow-up result mentioned above, we find that the analogy between the groups $\tilde{\calO}$ and $\tilde{\calO}_{2D}$ goes further than one would naively expect. This shall be demonstrated by the statements and proofs of various singular integral transform estimates in 3D, where the heart of the matter lies in a few explicit computations involving some special functions in 2D, including the so-called Bahouri-Chemin solution \cite{BC}.

\subsection{Main Results}

To state the main results, let us first explicitly write down our convention for the H\"older norms: for $0<\alpha<1$ and an open set $V \subset \bbR^d$, we define \begin{equation*}
\begin{split}
\nrm{f}_{C^\alpha_*(V)} = \sup_{x \ne x', x,x' \in V} \frac{|f(x)-f(x')|}{|x-x'|^\alpha},\quad \nrm{f}_{C^\alpha(V)} = \nrm{f}_{L^\infty(V)} + \nrm{f}_{C^\alpha_*(V)}.
\end{split}
\end{equation*} We recall the scale-invariant H\"older norms introduced in \cite{EJ1}\footnote{Note that if $\textbf{0} \in \bar{V}$ (which will be the case for our applications here) and $|x'|^\alpha f(x') \rightarrow 0$ as $x' \to \textbf{0}$ then from the definition of $\mathring{C}^\alpha$, we have $\nrm{f}_{L^\infty(V)} \le \nrm{f}_{\mathring{C}^\alpha_*(V)}.$}: \begin{equation*}
\begin{split}
\nrm{f}_{\mathring{C}^\alpha_*(V)} = \sup_{x \ne x', x,x' \in V} \frac{||x|^\alpha f(x)-|x'|^\alpha f(x')|}{|x-x'|^\alpha},\quad \nrm{f}_{\mathring{C}^\alpha(V)} = \nrm{f}_{L^\infty(V)} + \nrm{f}_{\mathring{C}^\alpha_*(V)}.
\end{split}
\end{equation*}  Finally, \begin{equation*}
\begin{split}
\nrm{f}_{C^\alpha\cap \mathring{C}^\alpha(V)} = \nrm{f}_{{C}^\alpha(V)}+ \nrm{f}_{\mathring{C}^\alpha(V)}.
\end{split}
\end{equation*} The merits in introducing the space $\mathring{C}^\alpha$ will be explained shortly after stating our main results. In the following we shall just fix some $0<\alpha<1$; its specific choice will not make any essential difference for the results in this paper, although the constants in the inequalities would depend in $\alpha$. 

To begin with, we state the local well-posedness theorem in $\mathring{C}^\alpha(\bbR^3)$ and in $C^\alpha \cap \mathring{C}^\alpha(\bbR^3)$ for the vorticity symmetric with respect to $\calO$. The proof was outlined in the thesis of the second author \cite{Jthesis}. 
\begin{theorem}[{{cf. \cite[Theorem 3.4.7]{Jthesis}}}]\label{thm:lwp}
	Let $\omega_0 \in \mathring{C}^\alpha(\bbR^3)$ (resp. $\omega_0 \in C^\alpha\cap\mathring{C}^\alpha(\bbR^3)$) be a divergence-free vector field and symmetric with respect to $\mathcal{O}$. Then,  there is a $T = T(\nrm{\omega_0}_{\mathring{C}^\alpha})>0$ ($T = T(\nrm{\omega_0}_{C^\alpha\cap\mathring{C}^\alpha})>0$, resp.) and a unique $\mathcal{O}$-symmetric solution $\omega \in C^0([0,T); \mathring{C}^\alpha(\bbR^3))$ ($\omega \in C^0([0,T); C^\alpha\cap\mathring{C}^\alpha(\bbR^3))$, resp.) to the 3D vorticity equations with initial data $\omega_0$. 
\end{theorem}
This result applies to initial vorticities which are odd symmetric with respect to the extended group $\tilde{\calO}$ since $\tilde{\calO} \supset \calO$ and we have shown already in the above that uniqueness implies propagation of odd symmetry with respect to $\tilde{\calO}$. Now note that the boundary of the fundamental domain $\tilde{U}$ consists of a finite union of infinite sectors. If $\omega_0 \in \mathring{C}^\alpha(\tilde{U})$ is normal to the boundary planes, then $\omega_0$ extends to $\tilde{\omega}_0$ which is odd symmetric with respect to $\tilde{\calO}$ and belongs to $\mathring{C}^\alpha(\bbR^3)$. Therefore, as a simple corollary of the above, we obtain the following:  
\begin{corollary}\label{cor:lwp}
Let $\omega_0 \in \mathring{C}^\alpha(\tilde{U})$ (resp. $\omega_0 \in C^\alpha\cap\mathring{C}^\alpha(\tilde{U})$) be divergence-free and further satisfy $\omega_0 \parallel n$ on $\partial\tilde{U}$ where $n$ is the unit normal vector. Then, there is a $T = T(\nrm{\omega_0}_{\mathring{C}^\alpha})>0$ ($T = T(\nrm{\omega_0}_{C^\alpha\cap\mathring{C}^\alpha})>0$, resp.) and a unique solution $\omega \in C^0([0,T); \mathring{C}^\alpha(\tilde{U}))$ ($\omega \in C^0([0,T); C^\alpha\cap\mathring{C}^\alpha(\tilde{U}))$, resp.) to the 3D vorticity equations with initial data $\omega_0$. 
\end{corollary}
We now present our main result in this paper, which is a local well-posedness result in $\tilde{U}$ which contains the above theorem as a special case: we treat vorticities which do not vanish at the origin. 
\begin{customthm}{A}[Local well-posedness]\label{thm:lwp-corner}
	Let $\omega_0 \in C^\alpha\cap\mathring{C}^\alpha(\tilde{U})$ be divergence-free and further satisfies that $\omega_0^1 + \omega_0^2$ vanishes on $\{ (x_1,x_2,x_3): x_1 = x_2 \ge 0, x_3 = 0 \}$. Then, there exists $T>0$ and a unique solution $\omega \in C([0,T); C^\alpha\cap\mathring{C}^\alpha(\tilde{U}))$ to the 3D Euler equations, satisfying that  $\omega^1(t) + \omega^2(t)$ vanishes on $\{ (x_1,x_2,x_3): x_1 = x_2 \ge 0, x_3 = 0 \}$ for any $0\le t <T$. 
\end{customthm}
 
Within this class of solutions, we can establish finite time singularity formation. 
\begin{customthm}{B}[Finite-time singularity formation]\label{thm:blowup-corner}
	There exists a set of smooth and compactly supported initial data satisfying the assumptions of Theorem \ref{thm:lwp-corner} whose unique local solutions blow-up in finite time: more specifically, given $\omega_0$ in the set, there exists $0<T^*<+\infty$ such that the solution $\omega(t)$ satisfies \begin{equation*}
	\begin{split}
	\liminf_{ t\rightarrow T^*} \nrm{\omega(t)}_{L^\infty(\tilde{U})} = +\infty. 
	\end{split}
	\end{equation*} Either one of the following conditions on the initial data $\omg_0 = (\omg_0^1,\omg_0^2,\omg_0^3)$ is sufficient for finite time blow-up: \begin{enumerate}
		\item $(-\omega_0^1 +\omega_0^2)(\mathbf{0}) \ne 2\omega_0^3(\mathbf{0}) $ and $(-\omega_0^1 +\omega_0^2)(\mathbf{0})\omega_0^3(\mathbf{0})\ne0$,    or
		\item $ (-\omega_0^1 +\omega_0^2)(\mathbf{0}) = 2\omega_0^3(\mathbf{0}) > 0$. 
	\end{enumerate}
\end{customthm}

At this point, let us mention two merits for introducing the space $\mathring{C}^\alpha$. The first point is that, as far as one is concerned with vorticities decaying fast at infinity, the space $\mathring{C}^\alpha$ is much larger than the classical H\"older space. It goes without saying that obtaining a local well-posedness result in a larger space is more difficult. More importantly, one gains access to the (well-defined) dynamics of scale-invariant vorticities; i.e. \begin{equation*}
\begin{split}
\omega(\lmb x) = \omega(x) 
\end{split}
\end{equation*} for any $\lmb > 0$ and $x \in \bbR^3$. Indeed, the 3D Euler equations has the following scale-invariance which fixes time: if $\omega(t,\cdot)$ is a solution, then  \begin{equation*}
\begin{split}
\omega^\lmb(t,x) := \omega(t,\lmb x)
\end{split}
\end{equation*} is again a solution for any $\lmb>0$. In particular, if one has initial vorticity $\omega_0$ which is scale-invariant in the sense defined above, then, upon having \textit{uniqueness}, it is guaranteed that the solution is automatically scale-invariant as long as it exists. Note that a scale-invariant vorticity cannot decay at infinity and is not even continuous at $\bf 0$ unless it is trivial. However, such a function still belongs to $\mathring{C}^\alpha(\bbR^3)$, if it is smooth in the angular directions; to be more precise, a scale-invariant function can be written as the form $\omega(x) = h(x/|x|)$ where $h$ is defined on the unit sphere $\bbS^2$. Then we have $\nrm{\omega}_{\mathring{C}^\alpha(\bbR^3)} \approx \nrm{h}_{C^\alpha(\bbS^2)}$, and the 3D Euler equations reduce to a 2D equation for $h$ defined on the unit sphere. It should be emphasized that the symmetry condition is necessary; already in 2D, without some appropriate symmetry condition, there is no \textit{uniqueness} for smooth and bounded initial vorticity which does not decay at infinity (\cite{EJ1}). It is not too difficult to see that, as demonstrated in \cite{EJ1}, the Biot-Savart integral cannot be convergent in general for non-decaying vorticities. 

The second point is that the space $\mathring{C}^\alpha$ encodes a decay condition for the vorticity which is natural with respect to the Euler equations, even if one is only concerned with $C^\alpha$ vorticities. To see that $\mathring{C}^\alpha$ implies some decay, one can just note that for $|x| \gg 1$ and $|x-x'| \lesssim 1$, we have from the definition that \begin{equation*}
\begin{split}
\frac{|f(x)-f(x')|}{|x-x'|^\alpha} \lesssim |x|^{-\alpha}. 
\end{split}
\end{equation*}  Usually, the decay on the vorticity is imposed by requiring $\omega$ to belong to some finite $L^p$-space or to have compact support. Although this can be done in the present context, we find it much more elegant to use the space $C^\alpha\cap\mathring{C}^\alpha$ to close the a priori estimates for local well-posedness. After this is done, it is easy to show by bootstrapping that if the initial vorticity is compactly supported, the unique local solution still has the same property. In relation to this, a fundamental property of the scale-invariant H\"older norm is the following product rule: \begin{equation*}
\begin{split}
\nrm{gf}_{C^\alpha} \le C \nrm{g}_{\mathring{C}^\alpha}\nrm{f}_{C^\alpha},\quad \mbox{if}\quad f(\textbf{0}) = 0. 
\end{split}
\end{equation*} This shows that, as long as one is concerned with the class of $C^\alpha$ functions vanishing at $\bf 0$, the $\mathring{C}^\alpha$-space works as a module over the class. Indeed, as we shall see below, the function space $\mathring{C}^\alpha$ as well as the above product rule appears naturally in singular integral estimates which involve only $C^\alpha$ functions. Now let us state a result which shows that the dynamics of the scale-invariant solutions described in the above is robust under smooth and vanishing perturbations. Not surprisingly the heart of the matter is the above product rule between $C^\alpha$ and $\mathring{C}^\alpha$ functions.

\begin{theorem}\label{thm:robust}
	Let $\omg_0 \in \mathring{C}^\alpha(\bbR^3)$ be a divergence-free vector field, symmetric with respect to $\calO$, and locally scale-invariant in the following sense: there exists a function $h_0 \in C^\alpha(\bbS^2)$ such that \begin{equation*}
	\begin{split}
	\tomg_0(x):= \omg_0(x) - h_0(\frac{x}{|x|}) \in C^\alpha(\bbR^3),\quad \lim_{|x|\to 0} \tomg_0(x) = 0. 
	\end{split}
	\end{equation*} Then, the unique local solution $\omg(t,x) \in C^0([0,T);\mathring{C}^\alpha(\bbR^3))$  provided by Theorem \ref{thm:lwp} satisfies \begin{equation*}
	\begin{split}
	\tomg(t,x):= \omg(t,x) - h(t,\frac{x}{|x|}) \in C^\alpha(\bbR^3),\quad \lim_{|x|\to 0} \tomg(t,x) = 0
	\end{split}
	\end{equation*} for all $0<t<T$, where $h(t,\cdot)$ is the unique local solution defined on $[0,T)$ to \begin{equation}\label{eq:Euler-SI}
	\begin{split}
	\rd_t h + v \cdot \nb  h = h \cdot \nb v 
	\end{split}
	\end{equation} on $\bbS^2$, where $v(t,\cdot)$ is defined on $\bbR^3$ by $v(t,x) = \nabla\times (-\lap)^{-1}(h(t,\frac{x}{|x|}))$. Note that $v$ is $1-$homogeneous in $|x|$ and the equation \eqref{eq:Euler-SI} reduces to a system on the unit sphere. 
\end{theorem}

The proof of Theorem \ref{thm:robust} will be omitted since it can be established without much difficulty following the proof of Theorem \ref{thm:blowup-corner}, which is a special case (see also \cite{EJ1}). This result shows that if the initial data is locally scale-invariant, the unique solution stays so as long as it does not blow up. (It is implied in the statement that the maximal time of local existence for $\omega_0$ is not larger than that of $h(0,\cdot)$.) This local scale-invariant profile can be found using the reduced system defined on $\bbS^2$. For the simplicity of presentation, we have defined $v$ by first extending $h$ on $\bbR^3$ and then applying the usual Biot-Savart law. Upon a concrete choice of coordinate system on $\bbS^2$ (e.g. standard spherical coordinates), one can write down a relation between $v$ and $h$ expressed in that coordinate system. It is an interesting open problem to see whether the system \eqref{eq:Euler-SI} admits solutions blowing up in finite time. 

\begin{remark} We give a few remarks regarding the statements of Theorems \ref{thm:lwp-corner} and \ref{thm:blowup-corner}. 
	\begin{itemize}
		\item We emphasize again that the use of spaces $\mathring{C}^\alpha$ is mainly to allow for more general local well-posedness theory. One can completely avoid using these norms and instead work with vorticities which are either compactly supported or belonging to $L^p$ for some finite $p>1$.
		\item In Theorem \ref{thm:lwp-corner}, the condition $\omega_0 \parallel n$ on $\partial\tilde{U}$ is now replaced with a much weaker condition that $\omega_0^1 + \omega_0^2$  vanishes on $\{ (x_1,x_2,x_3): x_1 = x_2 \ge 0, x_3 = 0 \}$. (The latter is implied by the former.) This vanishing condition is not artificial; one can see from the slip boundary condition that this condition is \textit{necessary} for the corresponding velocity to be at least $C^1$ regular in $\tilde{U}$. 	Indeed, assume that $u_0 \in C^1$, and recall that we have on the plane $\{ x_1 = x_2 \}$ the boundary condition  $u^1_0=u^2_0$. Similarly, on $\{x_3=0\}$, we have $u^3_0=0$. Then, \begin{equation*}
		\begin{split}
		\rd_3 u^1_0 = \rd_3 u^2_0, \quad \rd_1 u^3_0 = \rd_2u^3_0 = 0
		\end{split}
		\end{equation*} on $\{ (x_1,x_2,x_3): x_1 = x_2 \ge 0, x_3 = 0 \}$. It follows that \begin{equation*}
		\begin{split}
		\omega_0^1 = \rd_2 u^3_0 - \rd_3 u^2_0 = -\rd_3u_0^1 = -\omega_0^2
		\end{split}
		\end{equation*} and hence $\omega_0^1 + \omega_0^2 = 0$. 
		\item One may extend the solutions provided by Theorem \ref{thm:lwp-corner} to $\mathbb{R}^3$ and obtain $\omega \in L^\infty([0,T); L^\infty(\mathbb{R}^3))$ which can be considered as a 3D vortex patch. In 2D, a global well-posedness result for vortex patches with corner singularity has been established in \cite{SVP1} under the assumption that the vorticity is $m$-fold rotationally symmetric for some $m\ge 3$. Hence Theorem \ref{thm:lwp-corner} can be considered as an extension of this result to the 3D case, although in the current setting the patch boundary near the origin is fixed in time by reflection symmetries. Moreover, Theorem \ref{thm:blowup-corner} can be viewed as a finite-time blow up result for singular vortex patches defined in all of $\bbR^3$. 
		\item Under more restrictive vanishing assumptions on the initial vorticity, one can propagate higher regularity in space up to the blow-up time. The proof for $C^{1,\alpha}$ propagation of the vorticity is sketched in Proposition \ref{prop:higher}.
	\end{itemize} 
\end{remark}
\begin{remark}
	We note that the local well-posedness and singularity formation can be stated on the compact domain $\tilde{U}\cap\mathbb{T}^3$ where $\mathbb{T}^3 = [-1,1)^3$. 
\end{remark}

\subsection{Ideas of the proof}
Let us briefly explain the key ideas involved in the proof. 

\medskip

\noindent (1) Local well-posedness (Theorem \ref{thm:lwp-corner})

\medskip

\noindent 


The main ingredient in the local well-posedness is simply the $C^\alpha$-estimate for the double Riesz transforms $\{ R_{ij} \}_{1\le i,j\le 3}$ in $\tilde{U}$. We emphasize that the transform $R_{ij}$ defined in the following sense: given a \textit{vector-valued} function $f$ in $\tilde{U}$, we first extend $f$ to the whole of $\bbR^3$ as an odd function, denoted as $\tilde{f}$, with respect to $\tilde{\calO}$ and set $R_{ij}f := (\rd_{x_j}\rd_{x_j}(-\lap)^{-1}\tilde{f})|_{\tilde{U}}$. Taking $f$ to be the vorticity defined in $\tilde{U}$, the corresponding velocity gradient is given by a linear combination of the double Riesz transforms defined in this way (see the Appendix for the explicit formulas). 

As we mentioned earlier, rather than directly showing the $C^\alpha$-bound, we prove instead a $C^\alpha\cap\mathring{C}^\alpha$-estimate, which then gives as an immediate consequence the $C^\alpha$-bound for double Riesz transforms for $C^\alpha_c$-functions.  As a general rule, given a function $f \in C^\alpha\cap\mathring{C}^\alpha(\tilde{U})$ and a singular integral transform $T$, the $L^\infty$ and $\mathring{C}^\alpha$ bounds of $Tf$ will follow from the $\mathring{C}^\alpha$ bound of $f$, while the $C^\alpha_*$ estimate follows from the $C^\alpha_*$ bound of $f$. 

We apply a number of reductions to obtain the desired $C^\alpha\cap\mathring{C}^\alpha$-estimate. First, if $f$ is a constant vector-valued function inside $\tilde{U}$, explicit computations show that double Riesz transforms are well-defined as long as $f$ satisfies the vanishing condition in Theorem \ref{thm:lwp-corner}, and that they are again given by constant vector-valued functions. This allows us to assume that, without loss of generality, $f$ vanishes at the origin. Next, since $\partial\tilde{U}$ is smooth away from three half-lines denoted by $\vec{\frka}_2, \vec{\frka}_3$, and $\vec{\frka}_4$, it suffices to consider regularity of $Tf$ close to these half-lines. Moreover, with a radially-homogeneous partition of unity, we may assume that the support of $f$ is ``adjacent'' to one of the half-lines, say $\vec{\frka}_j$. Then $Tf$ is possibly singular only near $\vec{\frka}_j$, as a simple consequence of the ``symmetry reduction'' lemma \ref{lem:symmetry-general}. To prove regularity of $Tf$ up to $\vec{\frka}_j$,  we make an orthogonal change of coordinates so that $\vec{\frka}_j$ is parallel the new $x_3$-axis. In this new coordinates system, any double Riesz transform involving $\rd_{x_3}$ satisfies the $C^\alpha$-bound; roughly speaking, this is because locally $\partial{\tilde{U}}$ is parallel to $\vec{\frka}_j$. More precisely, this is due to the intrinsic property of double Riesz transforms that the integral of the kernel over any hemisphere vanishes. It still remains to treat the Riesz transforms $R_{11}$, $R_{22}$, and $R_{12}$: for these we apply again Lemma \ref{lem:symmetry-general} and reduce the statements to 2D H\"older estimates. The necessary 2D bounds are collected and proved in \ref{subsec:estimate-2D}. The most tricky part is to obtain estimates near $\vec{\frka}_2$, which creates a corner with angle $\frac{\pi}{2}$. Applying the aforementioned 2D reduction near $\vec{\frka}_2$, the vanishing condition in Theorem \ref{thm:lwp-corner} naturally comes into play, since we are now concerned with functions defined on the positive quadrant $Q = (\bbR_+)^2$. It is well-known that (\cite{Grisvard1985}) even for $f \in C^\infty_c(Q)$, $\nb^2(-\lap_Q)^{-1}f \notin L^\infty(Q)$. We restore the $L^\infty$-bound under the vanishing condition of $f$ at the origin, which corresponds to the vanishing condition of the parallel component vorticity along the whole half-line $\vec{\frka}_2$.  



Given the a priori estimates, the arguments for existence and uniqueness are completely straightforward, since the vorticity does not necessarily decay at infinity. For decaying vorticities, uniqueness follows immediately in our setting since the velocity has finite energy and bounded in Lipschitz norm. To deal with this issue with decay, we adapt an argument from our previous work \cite{EJB} on the 2D Boussinesq system which is based on proving stability of the quantity $\nrm{|x|^{-1}u(t,x)}_{L^\infty_x}$. The existence statement is proved by a careful iteration scheme which ensures the crucial vanishing condition for the sequence of vorticities.

\medskip

\noindent (2) Finite-time singularity formation (Theorem \ref{thm:blowup-corner})

\medskip

\noindent The starting point of the blow-up proof is to find explicit solutions which blow up in finite time. They are given by constant vector functions satisfy a system of ODEs obtained from the 3D Euler equations. The existence of such solutions in the case of the 2D Boussinesq system was obtained in \cite{EJB} in domains with angles less than or equal to $\frac{\pi}{4}$. Then using a cut-off argument introduced in \cite{EJB} we can localize the blow-up solution. 



\subsubsection*{Other blow-up results for the 3D Euler equations}

Let us briefly comment on other blow-up results for the 3D incompressible Euler equations. We restrict ourselves to blow-up results concerning Lipschitz and finite-energy velocities. In our previous work \cite{EJEuler}, we have proved finite-time singularity formation for the 3D axisymmetric Euler equations in corner domains which can be arbitrarily close to the flat cylinder. The philosophy in \cite{EJEuler} is similar to the present work, although here we can deal with data which is smooth in a corner domain and extends to $\bbR^3$ as a globally Lipschitz function.

In the recent works \cite{E3}, \cite{ETM}, and \cite{CH}, finite-time singularity formation was achieved with $C^{1,\alpha}$ velocities in axisymmetric setting. The domain is simply $\bbR^3$ for \cite{E3}, \cite{ETM} and $\bbR^3_+$ for \cite{CH}. The approach taken in these works is ``orthogonal'' to the present work; the authors take advantage of small $\alpha>0$ and the ``unbounded'' (in the limit $\alpha\to 0$) term in the double Riesz transforms  to achieve singularity formation. Such unbounded terms do not appear at all in our context, which is precisely the role of rotational symmetries we impose (see Section \ref{sec:expansion} for more details on this).

\subsubsection*{Organization of the paper}

In Section \ref{sec:expansion}, we obtain an expansion of the velocity in terms of the vorticity. In particular, we show sharp estimates on the velocity for vorticity uniformly bounded and symmetric with respect to $\calO$. All the results automatically apply to bounded vorticities defined in $\tilde{U}$. In Section \ref{sec:estimate}, we prove $C^\alpha$-estimates for the double Riesz transforms in domains with corners. We first consider the estimates for the 2D case, since then the 3D estimates can be viewed as a natural extension. Finally, collecting the estimates, we establish Theorems \ref{thm:lwp-corner} and \ref{thm:blowup-corner} in Sections \ref{sec:lwp} and \ref{sec:blowup}, respectively.

\subsubsection*{Notations} 
For convenience of the reader, we shall collect the notations and definitions that will be used throughout the paper. 
\begin{itemize}
	\item Reflection across the plane $\{ x_i = 0 \}$: $R_i$.
	\item Counterclockwise rotation by angle $\frac{\pi}{2}$ fixing the $x_i$-axis: $P_i$.
	\item Symmetry groups and fundamental domains: Recall that the pair $(\calO,U)$ is generated by $\{P_i\}_{1\le i \le 3}$ and $(\tilde{\calO},\tilde{U})$ by $\{P_i,R_i\}_{1\le i \le 3}$.
	\item The origin of $\bbR^n$ will be denoted by $\mathbf{0}$. Moreover, given an open set $A\subset \bbR^n$, the indicator function on $A$ is defined by $\mathbf{1}_A$. 
\end{itemize}

We shall now define the extension rules for a function defined in a fundamental domain of a symmetry group. 
\begin{definition}[Extension rules in the vector-valued case]
	Assume that we are given a group $\mathcal{G}$ of isometries of $\bbR^n$ fixing $\mathbf{0}$, which divides $\bbR^n$ into finitely many fundamental domains. Let $U=U_{\calG}$ be one of the fundamental domains, and $f:U\rightarrow\bbR^n$ be a given vector-valued function. We define $\tilde{f}^{\calG}:\bbR^n\rightarrow\bbR^n$ as follows: for any $y\in\bbR^n$, there is a unique element $g\in\calG$ satisfying $y=gx$ for some $x\in U$. Then we set $\tilde{f}^{\calG}(y)= g(f(x))$. Strictly speaking, $\tilde{f}^\calG$ is not well-defined on $\partial U$ and its images by $g\in\calG$, but this will not cause any trouble in this work. We define the extension of the identity $\mathrm{Id}_{U}$ as follows: $\widetilde{\mathrm{Id}}^\calG(y)$ is the matrix-valued function which equals $g\in\calG$ on the fundamental domain $g(U)$. 
\end{definition}

\begin{definition}[Extension rules in the scalar case]
	In the same setting as above, if $f:U\rightarrow\bbR$ is a scalar-valued function, then we simply define $\tilde{f}^\calG(y) =(-1)^{\mathrm{sgn}(g)}  f(g^{-1}y)$, where we take the $-$ sign if and only if the group element $g$ is orientation reversing (i.e. a reflection). In particular, the extension of the indicator function $\mathbf{1}_U$ will be simply $\tilde{\mathbf{1}}^\calG = \sum_{g\in\calG} (-1)^{\mathrm{sgn}(g)}\mathbf{1}_{g(U)}$.  
\end{definition}

We shall be mainly concerned with $(\calG,U) = (\tilde{\calO},\tilde{U})$ in the 3D case and $(\tilde{\calO}_{2D},\tilde{U}_{2D})$ in the 2D case. We shall omit the superscript $\calG$ and subscript $U$ when the pair $(\calG,U)$ is understood from the context. 

\subsubsection*{Acknowledgments}
Research of TE was partially supported by NSF-DMS 1817134. IJ was supported by the Science Fellowship of POSCO TJ Park Foundation and the National Research Foundation of Korea (NRF) grant No. 2019R1F1A1058486. We remark that Figure \ref{fig:octahedral} belongs to the public domain\footnote{https://commons.wikimedia.org/wiki/File:Octahedral\_reflection\_domains.png}.

\section{An expansion for the velocity}\label{sec:expansion}

The goal of this section is to present an expansion for the velocity in terms of the vorticity in 3D which can be viewed as a generalization of ``Key Lemma'' appeared in \cite{KS,SVP1}. Using the expansion, we deduce that for a bounded vorticity which is symmetric with respect to $\mathcal{O}$, the Biot-Savart law is well-defined pointwise without any decay assumptions on the vorticity.  

\subsection{Two-dimensional case}\label{subsec:expansion-2D}

Recall that in two dimensions, we have the following decomposition of the velocity vector: 
\begin{lemma}\label{lem:vel_expansion}
	Assume that $\omega \in L^\infty_c(\mathbb{R}^2)$. Then, the corresponding velocity $u = (u_1,u_2) := \nabla^\perp\Delta^{-1}\omega$ satisfies the estimate \begin{equation*}
	\begin{split}
	\left| u_1(x_1,x_2) - \mrI^s -  {x_1} \mrII^s(|x|) +  {x_2}\mrII^c(|x|)  \right| \le C|x|  \nrm{\omega}_{L^\infty},
	\end{split}
	\end{equation*}\begin{equation*}
	\begin{split}
	\left| u_2(x_1,x_2) + \mrI^c +  {x_2}\mrII^s(|x|) +  {x_1}\mrII^c(|x|) \right| \le C|x| \nrm{\omega}_{L^\infty},
	\end{split}
	\end{equation*} with some absolute constant $C > 0$ independent on the size of the support of $\omega$. Here, writing $\omega$ in polar coordinates, 
	\begin{equation*} 
	\begin{split}
	& \mrI^s = \frac{1}{2\pi} \int_0^\infty \int_0^{2\pi} \sin(\tht) \omg(r,\tht)  d\tht dr , \quad \mrI^c =  \frac{1}{2\pi} \int_0^\infty \int_0^{2\pi} \cos(\tht) \omg(r,\tht) d\tht dr ,
	\end{split}
	\end{equation*} and \begin{equation*} 
	\begin{split}
	& \mrII^s(r) =  \frac{1}{2\pi} \int_r^\infty \int_0^{2\pi} \sin(2\tht) \omg(s,\tht) ,  d\tht ds \quad  \mrII^c(r) =  \frac{1}{2\pi} \int_r^\infty \int_0^{2\pi} \cos(2\tht) \omg(s,\tht)   d\tht ds .
	\end{split}
	\end{equation*} 
\end{lemma} This appeared in \cite{SVP1} but earlier results can be traced back to \cite{E1,KS,IMY1,Z}. The idea of the proof is very simple: from the explicit representation formula \begin{equation*} 
\begin{split}
& u(x) = \frac{1}{2\pi} \int_{ \bbR^2} \frac{(x-y)^\perp}{|x-y|^2} \omega(y) dy, 
\end{split}
\end{equation*} one sees that the kernel is integrable in the region $|y| \lesssim |x|$ with integral estimate of the form $C|x|  \nrm{\omega}_{L^\infty}$. Therefore it suffices to consider the region $|y| \gtrsim |x|$, where $|x-y| \approx |y|$. Then one may subtract expressions of the form $p(y)/|y|^n$ where $p$ is some homogeneous polynomial from the kernel until the new kernel decays like $|y|^{-3}$ for fixed $x$. Then integrating this fast decaying kernel gives a bound of $C|x|  \nrm{\omega}_{L^\infty}$, while the subtracted expressions evaluate to quantities like $\mrI^{s,c}, \mrII^{s,c}$ defined in the above. We shall utilize the exact same strategy in the proof of the 3D version below. 

Before we proceed to the three-dimensional case, we consider the simplifications of the above expansion obtained by assuming symmetry conditions on $\omg$, which is preserved by the 2D Euler equations. 

\begin{corollary}
	Under the same assumptions as in Lemma \ref{lem:vel_expansion}, suppose in addition that \begin{enumerate}
		\item the vorticity is odd with respect to both axes; or
		\item the vorticity is 4-fold rotationally symmetric; that is, $\omg(x) = \omg(x^\perp)$ for all $x \in \bbR^2$.
	\end{enumerate} Then, the bounds simplify into \begin{equation}\label{eq:oddodd-2D}
\begin{split}
\left| u(x) - \begin{pmatrix}
-x_1 \\ x_2
\end{pmatrix} \mrII^s(|x|) \right| \le C|x|  \nrm{\omega}_{L^\infty}
\end{split}
\end{equation} and \begin{equation}\label{eq:rot-2D}
\begin{split}
\left| u(x) \right|\le C|x|  \nrm{\omega}_{L^\infty},
\end{split}
\end{equation} respectively. 
\end{corollary} The estimate \eqref{eq:oddodd-2D} clearly shows that for vorticity which is bounded and odd-odd, the velocity can be log-Lipschitz near the origin, which is responsible for both double exponential growth of the vorticity gradient \cite{KS} and ill-posedness of the Euler equations at critical regularity \cite{BL1,BL2,EJ,EM1}. On the other hand, such a logarithmic divergence is removed at the origin once one imposes an appropriate symmetry assumption on the vorticity, as in \eqref{eq:rot-2D}.

\subsection{Three-dimensional case}

We now state and prove a corresponding result for the 3D velocity, which is given by the following Biot-Savart law: \begin{equation}\label{eq:3D_Biot-Savart}
\begin{split}
u(x) = \frac{1}{4\pi}\int_{ \mathbb{R}^3} \nabla_x \times \left(|x-y|^{-1} \omega(y)\right) dy. 
\end{split}
\end{equation} We first define auxiliary integral operators \begin{equation}\label{eq:int_first}
\begin{split}
\mrI_{j}[f] := \frac{1}{4\pi} \int_{\mathbb{R}^3} \frac{y_j}{|y|^3}f(y) dy,
\end{split}
\end{equation} \begin{equation}\label{eq:int_second_1}
\begin{split}
\mrII^{(1)}_{j}[f](r) := \frac{1}{4\pi} \int_{| y|\ge 2r} \frac{y_{j+1}^2-y_{j-1}^2}{|y|^5} f(y) dy,
\end{split}
\end{equation} and \begin{equation}\label{eq:int_second_2}
\begin{split}
\mrII^{(2)}_{j}[f](r) := \frac{1}{4\pi}\int_{| y|\ge 2r} \frac{y_{j+1}y_{j-1}}{|y|^5} f(y) dy.
\end{split}
\end{equation} Here $j \in \{ 1, 2, 3\}$ and the indices are defined modulo 3; that is, $y_4 := y_1$ and $y_0 := y_3$. 

\begin{lemma}\label{lem:3D_vel_expansion}
	Assume that $\omega = (\omega_1,\omega_2,\omega_3)$ with $\omega_i  \in L^\infty_c(\mathbb{R}^3)$ for $i = 1, 2, 3$. Then the velocity given by \eqref{eq:3D_Biot-Savart} satisfies the estimate \begin{equation}\label{eq:3D_vel_estimate}
	\begin{split}
	u_i(x) &= B_i(x)  + \mrI_{i+1}[\omg_{i-1}] - \mrI_{i-1}[\omg_{i+1}]  \\
	&\quad - x_{i+1} \mrII^{(1)}_{i-1}[\omg_{i-1}](|x|) + x_{i+1}\mrII^{(1)}_{i}[\omg_{i-1}](|x|) + x_{i-1} \mrII^{(1)}_{i}[\omg_{i+1}](|x|) - x_{i-1} \mrII^{(1)}_{i+1}[\omg_{i+1}](|x|) \\
	&\quad + 3x_i \mrII^{(2)}_{i-1}[\omg_{i-1}](|x|) + 3x_{i-1} \mrII^{(2)}_{i}[\omg_{i-1}](|x|) - 3x_{i+1} \mrII^{(2)}_{i}[\omg_{i+1}](|x|) - 3x_{i} \mrII^{(2)}_{i+1}[\omg_{i+1}](|x|) 
	\end{split}
	\end{equation} where \begin{equation}\label{eq:3D_vel_remainder_est}
	\begin{split}
	|B_i(x)| \le C|x|\left( \nrm{\omega_{i+1}}_{L^\infty}+ \nrm{\omega_{i-1}}_{L^\infty}\right)
	\end{split}
	\end{equation} with some absolute constant $C>0$ independent on the size of the support of $\omega$. 

\end{lemma}
\begin{proof}
	It suffices to prove the estimates for $i = 1$, since then other cases can be obtained by shifting the indices. We begin with the following explicit formula for the first component of velocity: \begin{equation}\label{eq:3D_vel_firstcomp}
	\begin{split}
	u_1(x) = -\frac{1}{4\pi} \int_{ \mathbb{R}^3} \frac{(x_2-y_2)\omega_3 - (x_3-y_3)\omega_2}{|x-y|^3} dy,
	\end{split}
	\end{equation} and it suffices to consider the term \begin{equation}\label{eq:3D_vel_oneterm}
	\begin{split}
	\int_{ \mathbb{R}^3} \frac{x_2-y_2}{|x-y|^3} \omega_3(y) dy
	\end{split}
	\end{equation} since the corresponding estimate for the other term can be obtained by simply relabeling the indices. Next, note that \begin{equation*}
	\begin{split}
	\left|\int_{|x-y|\le 2|x|} \frac{x_2-y_2}{|x-y|^3} \omega_3(y) dy\right| \le \nrm{\omega_3}_{L^\infty}\int_{|x-y|\le 2|x|} \frac{1}{|x-y|^2}  dy \le C\nrm{\omega_3}_{L^\infty}|x|
	\end{split}
	\end{equation*} with an absolute constant $C > 0$ so that for the purpose of establishing \eqref{eq:3D_vel_estimate}, one only needs to deal with the region where $|x-y| \gtrsim |x|$. (Hence, $|x-y| \approx |y|$ holds.) We first write \eqref{eq:3D_vel_oneterm} as \begin{equation*}
	\begin{split}
	\int_{ \mathbb{R}^3} \left[\frac{x_2-y_2}{|x-y|^3} + \frac{y_2}{|y|^3} \right] \omega_3(y) dy - \int_{ \mathbb{R}^3} \frac{y_2}{|y|^3} \omega_3(y) dy
	\end{split}
	\end{equation*} and combining the terms in the large brackets, we obtain \begin{equation}\label{eq:inside_bracket}
	\begin{split}
	&\frac{1}{|x-y|^3|y|^3}\left[ (x_2-y_2)(y_1^2+y_2^2+y_3^2)|y| + y_2((x_1-y_1)^2+(x_2-y_2)^2+(x_3-y_3)^2)|x-y| \right] \\
	&\qquad = \frac{1}{|x-y|^3|y|^3}\left[ -y_2|y|^2(|y| - |x-y|) + x_2|y|^3 + y_2|x|^2|x-y| - 2y_2 x\cdot y|x-y| \right] \\
	&\qquad = \frac{1}{|x-y|^3|y|^3}\left[ -y_2|y|^2 \frac{2x\cdot y-|x|^2}{|y|+|x-y|} + x_2|y|^3 + y_2|x|^2|x-y| - 2y_2 x\cdot y|x-y| \right] \\
	&\qquad = \frac{1}{|x-y|^3|y|^3}\left[ \frac{-2y_2|y|^2}{|y|+|x-y|} x\cdot y + x_2|y|^3 - 2y_2|x-y| x\cdot y +|x|^2\left( \frac{y_2|y|^2}{|y|+|x-y|} + y_2|x-y|\right) \right] 
	\end{split}
	\end{equation} To the above expression, we add and subtract the following quantity \begin{equation}\label{eq:subtract}
	\begin{split}
	\frac{1}{|y|^3|x-y|^3} \left[ -2y_2(x\cdot y) \frac{3|y|}{2} + |y|x_2 |y|^2 \right]. 
	\end{split}
	\end{equation} Subtracting \eqref{eq:subtract} from \eqref{eq:inside_bracket} gives \begin{equation}\label{eq:3}
	\begin{split}
	\frac{1}{|x-y|^3|y|^3}\left[ -2y_2(x\cdot y)\left( \frac{|y|^2}{|y|+|x-y|} - \frac{|y|}{2} + |x-y| - |y|\right) \right] + \frac{|x|^2}{|x-y|^3|y|^3}\left[ \frac{y_2|y|^2}{|y|+|x-y|} + y_2|x-y|\right].
	\end{split}
	\end{equation} Denoting \eqref{eq:3} as $A$, note that \begin{equation*}
	\begin{split}
	\left| \int_{|x-y|\ge 2|x|} A(x,y)\omega_3(y) dy  \right| \le C|x|^2 \nrm{\omega_3}_{L^\infty} \int_{|x-y|\ge 2|x|} \frac{1}{|y|^4} dy \le  C|x| \nrm{\omega_3}_{L^\infty} 
	\end{split}
	\end{equation*} since we have a pointwise estimate \begin{equation*}
	\begin{split}
	|A(x,y)| \le C|x|^2|y|^{-4} 
	\end{split}
	\end{equation*} in the region $\{ |x-y|\ge 2|x|\}$ for some absolute constant $C > 0$. Similarly, one may replace the integral \begin{equation*}
	\begin{split}
	\int_{|x-y|\ge 2|x|} \frac{1}{|y|^3|x-y|^3} \left[ -2y_2(x\cdot y) \frac{3|y|}{2} + |y|x_2 |y|^2 \right] \omega_3(y) dy 
	\end{split}
	\end{equation*} with \begin{equation*}
	\begin{split}
	\int_{|x-y|\ge 2|x|} \frac{1}{|y|^6} \left[ -2y_2(x\cdot y) \frac{3|y|}{2} + |y|x_2 |y|^2 \right] \omega_3(y) dy 
	\end{split}
	\end{equation*} at the cost of introducing an error of size $C|x| \nrm{\omega_3}_{L^\infty} $ with some absolute constant $ C > 0$. Hence we have shown that \begin{equation*}
	\begin{split}
	\left| \int_{ \mathbb{R}^3} \frac{x_2-y_2}{|x-y|^3} \omega_3(y) dy + \int_{ \mathbb{R}^3} \frac{ y_2}{| y|^3} \omega_3(y) dy
	- \int_{|x-y|\ge 2|x|} \frac{1}{|y|^5}\left( -3y_2(x\cdot y) + |y|^2 x_2 \right)\omega_3(y) dy    \right| \le C|x| \nrm{\omega_3}_{L^\infty}. 
	\end{split}
	\end{equation*} Finally, \begin{equation*}
	\begin{split}
	&\left| \int_{| y|\ge 2|x|} \frac{1}{|y|^5}\left(x_2(y_1^2-y_2^2) + x_2(y_3^2-y_2^2) -3x_1y_1y_2 - 3x_3 y_2y_3 \right) \omega_3(y) dy \right. \\
	&\left. \qquad - \int_{|x-y|\ge 2|x|} \frac{1}{|y|^5}\left( -3y_2(x\cdot y) + |y|^2 x_2 \right)\omega_3(y) dy    \right| \le C|x| \nrm{\omega_3}_{L^\infty}. 
	\end{split}
	\end{equation*} This finishes the proof. 
\end{proof}

\begin{corollary}
	Assume that $\omega \in (L^\infty_c(\mathbb{R}^3))^3$ is symmetric with respect to $\mathcal{O}$. Then, the corresponding velocity satisfies the estimate \begin{equation*}
	\begin{split}
	|u(x)| \le C|x|\nrm{\omega}_{L^\infty}. 
	\end{split}
	\end{equation*} Similarly, for $\omega \in (L^\infty(\mathbb{R}^3))^3 $ and symmetric with respect to $\mathcal{O}$ but not necessarily compactly supported, the following principal value integral \begin{equation*}
	\begin{split}
	u(x) := \lim_{R \rightarrow +\infty}\frac{1}{4\pi}\int_{ |y| \le R} \nabla_x \times \left(|x-y|^{-1} \omega(y)\right) dy
	\end{split}
	\end{equation*} is pointwise well-defined and satisfies the estimate \begin{equation*}
	\begin{split}
	|u(x)| \le C|x|\nrm{\omega}_{L^\infty}.
	\end{split}
	\end{equation*}
\end{corollary}
\begin{proof}
	We only consider the component $u_1$. To prove the estimate, it suffices to show that the auxiliary integrals involved in \eqref{eq:3D_vel_estimate} vanishes altogether. For this we show that for each fixed radius, the corresponding ``angular integration'' vanishes. Below $d\sigma$ will denote the Lebesgue measure on the sphere $\rd B_0(R)$. 
	
	We begin with noting that from $P_2^2\omega(x) = \omega(P_2^2x)$, we obtain that $\omega_2(x) = \omega_2(-x_1,x_2,-x_3)$. This implies that \begin{equation*}
	\begin{split}
	\int_{\partial B_0(R)} \frac{y_3}{|y|^3}\omega_2(y) d\sigma(y) = 0
	\end{split}
	\end{equation*} for any $R > 0$. Moreover, \begin{equation*}
	\begin{split}
	\int_{\partial B_0(R)} \frac{y_1y_2}{|y|^5}\omega_2(y) d\sigma(y) = \int_{\partial B_0(R)} \frac{y_3y_2}{|y|^5}\omega_2(y) d\sigma(y) =  0.
	\end{split}
	\end{equation*} Then, $P_2\omega(x) = \omega(P_2x)$ gives $\omega_2(x) = \omega_2(-x_3,x_2,x_1)$. This establishes \begin{equation*}
	\begin{split}
	\int_{\partial B_0(R)} \frac{y_1y_3}{|y|^5}\omega_2(y) d\sigma(y) = 0.
	\end{split}
	\end{equation*}  Finally, from $P_1^2\omega(x) = \omega(P_1^2 x) $, we have $\omega_2(x) = -\omega_2(x_1,-x_2,-x_3)$. This allows us to show that \begin{equation*}
	\begin{split}
	\int_{\partial B_0(R)} \frac{y_1^2}{|y|^5}\omega_2(y) d\sigma(y) = \int_{\partial B_0(R)} \frac{y_2^2}{|y|^5}\omega_2(y) d\sigma(y) = \int_{\partial B_0(R)} \frac{y_3^2}{|y|^5}\omega_2(y) d\sigma(y) = 0.
	\end{split}
	\end{equation*} Similarly, it can be shown that \begin{equation*}
	\begin{split}
	\int_{\partial B_0(R)} \frac{y_2}{|y|^3}\omega_3(y) d\sigma(y) = 0
	\end{split}
	\end{equation*} and \begin{equation*}
	\begin{split}
	\int_{\partial B_0(R)} \frac{p(y)}{|y|^5}\omega_3(y) d\sigma(y) = 0
	\end{split}
	\end{equation*} where $p(\cdot)$ is any second order homogeneous polynomial. 
\end{proof}

The following result shows that $u$ defined by the above principal value is the ``correct'' velocity field associated with $\omega$. 
\begin{proposition}\label{prop:velocity}
	Assume that $\omega \in (L^\infty(\mathbb{R}^3))^3$ is symmetric with respect to $\mathcal{O}$ and divergence-free (in the weak sense). Then, for any $1\le p<\infty$, there exists a unique $u \in (W^{1,p}_{loc}(\bbR^3))^3$ satisfying \begin{itemize}
		\item $u$ is symmetric with respect to $\calO$, 
		\item $\nabla\times u = \omg$, 
		\item $\nabla\cdot u = 0$,
		\item $|u(x)| \le C \nrm{\omg}_{L^\infty}|x|$ for all $x \in \bbR^3$ with some constant $C>0$. 
	\end{itemize}
\end{proposition}
\begin{proof}
	Existence is provided by the above lemma. For uniqueness one can repeat the proof given in \cite{EJ1} for the 2D Euler equations. We omit the details. 
\end{proof}

In light of the above proposition, given any $\omega \in (L^\infty(\tilde{U}))^3$, we may first extend $\omega$ to $\tilde{\omega} \in (L^\infty(\mathbb{R}^3))^3$ and the corresponding velocity $u \in (W^{1,p}_{loc}(\tilde{U}))^3$ is well-defined as the (principal value of the) Biot-Savart integral against $\tilde{\omega}$. The difficult part is to prove that, under suitable additional assumptions on $\omega$, we have $u \in W^{1,\infty}_{loc}$ (among other things), and this is the content of the next section.

\section{H\"older estimates under octahedral symmetry}\label{sec:estimate}

\subsection{Two-dimensional case}\label{subsec:estimate-2D}

\subsubsection{Main Lemmas}\label{subsubsec:lemmas}

We shall state a few sharp H\"older estimates for the double Riesz transformations for functions defined on $\bbR^2$. The proofs are carried out later in \ref{subsubsec:proofs}, after performing some illuminating computations in \ref{subsubsec:explicit}. We shall define the following sector domains in 2D using polar coordinates: \begin{equation*}
\begin{split}
\Omega_m := \{ (r,\theta): 0 < \theta < \frac{\pi}{m} \}.
\end{split}
\end{equation*} We are particularly interested in the cases $m = 2, 3$, and $4$. Given $f \in L^\infty(\Omega_m)$, we shall define $R_{ij}f$ to be $\rd_{x_i}\rd_{x_j} (-\lap_D)^{-1}f$ where $\lap_D$ is the Dirichlet Laplacian. Let us present three simple lemmas which establish $C^\alpha$ estimates for $R_{ij}$ on $\Omega_m$. The first result is well known and a standard reference is the book of Grisvard \cite{Grisvard1985} (see also \cite{EJB,EJEuler}). 
\begin{lemma}\label{lem:2D-1}
	Let $f \in C^\alpha(\Omega_m)$ with $m = 3, 4$. Then we have for any $1\le i,j\le2$, \begin{equation*}
	\begin{split}
	\Vert R_{ij}f\Vert_{L^\infty(\Omega_m)} \le C\nrm{f}_{L^\infty(\bbR^2)} \log(10 + \frac{\Vert f\Vert_{\mathring{C}^\alpha(\Omega_m)}}{\nrm{f}_{L^\infty(\bbR^2)}} ),
	\end{split}
	\end{equation*} and \begin{equation*}
	\begin{split}
	\Vert R_{ij}f\Vert_{\mathring{C}_*^\alpha(\Omega_m)} \le C\Vert f\Vert_{\mathring{C}_*^\alpha(\Omega_m)},\quad \Vert R_{ij}f\Vert_{C_*^\alpha(\Omega_m)} \le C \Vert f\Vert_{C_*^\alpha(\Omega_m)}.
	\end{split}
	\end{equation*}  In particular, we have \begin{equation*}
	\begin{split}
	\nrm{R_{ij}f}_{ C^\alpha\cap\mathring{C}^\alpha(\Omega_m) } \le C  \nrm{f}_{C^\alpha\cap\mathring{C}^\alpha(\Omega_m)}
	\end{split}
	\end{equation*} and if in addition $f$ is supported in $\{ |x|<R \}$, \begin{equation*}
	\begin{split}
	\nrm{R_{ij}f}_{C^\alpha(\Omega_m)} \le C(R)  \nrm{f}_{C^\alpha_c(\Omega_m)}.
	\end{split}
	\end{equation*}
\end{lemma}
On the other hand, in the quadrant case $\Omega_2$, it has been well known that one cannot even get $C^0$ estimate: \begin{equation*}
\begin{split}
\Vert R_{ij} f\Vert_{C^0(\Omega_2)} \nleq C\Vert  f\Vert_{C^\alpha_c(\Omega_2)}.
\end{split}
\end{equation*} Indeed, using the definition of the Dirichlet Laplacian and \textit{assuming} that $R_{ij}f \in C^0$, \begin{equation*}
\begin{split}
&\partial_{x_2}(-\Delta_D)^{-1} f(0,x_2) = 0 \implies \partial_{x_2}\partial_{x_2}(-\Delta_D)^{-1} f(0,x_2) = 0, \\
&\partial_{x_1}(-\Delta_D)^{-1} f(x_1,0) = 0 \implies \partial_{x_1}\partial_{x_1}(-\Delta_D)^{-1} f(x_1,0) = 0.
\end{split}
\end{equation*} Hence \begin{equation*}
\begin{split}
R_{11}f, R_{22}f \in C^0 \implies f(0,0) = 0. 
\end{split}
\end{equation*} Our next lemma shows that the vanishing condition at the origin is the only obstruction for the sharp $C^\alpha$ estimate. \begin{lemma}\label{lem:2D-2}
	Let $f \in C^\alpha_c(\Omega_2)$ with $f(0,0)=0$ and $f = 0$ in $\{ r>R \}$ for some $R>0$. Then we have for any $1\le i,j\le2$, \begin{equation*}
	\begin{split}
	\nrm{R_{ij}f}_{C^\alpha(\Omega_2)} \le C(R) \nrm{f}_{C^\alpha(\Omega_2)}
	\end{split}
	\end{equation*} where $C(R)>0$ is a constant depending on the radius of the support of $f$.
\end{lemma}
The final lemma concerns functions which are odd in $\bbR^2$ and $C^\alpha$ when restricted to the sectors. For convenience we shall define \begin{equation*}
\begin{split}
\Omega_m^i = \{ (r,\theta) : (i-1)\frac{\pi}{m} < \theta < i\frac{\pi}{m} \}, \quad m \ge 3, 1 \le i \le 2m. 
\end{split}
\end{equation*} We have in particular that $\Omega_m=\Omega_m^1$. \begin{lemma}\label{lem:2D-3}
	Let $f \in L^\infty(\bbR^2)$ satisfy $f(x)=-f(-x)$ for all $x \in \bbR^2$. Furthermore, assume that $f \in C^\alpha\cap\mathring{C}^\alpha(\Omega_m^i)$ for some $m \ge 3$ and all $1 \le i \le 2m$. Then for all $1\le i,j\le2$ and $R_{ij} = \rd_{x_i}\rd_{x_j}(-\lap_{\bbR^2})^{-1}$, we have \begin{equation*}
	\begin{split}
	\Vert R_{ij}f\Vert_{L^\infty(\Omega_m)} \le C\nrm{f}_{L^\infty(\bbR^2)} \log(10 + \frac{\Vert f\Vert_{\mathring{C}^\alpha(\Omega_m)}}{\nrm{f}_{L^\infty(\bbR^2)}} ),
	\end{split}
	\end{equation*} \begin{equation*}
	\begin{split}
	\Vert R_{ij}f\Vert_{\mathring{C}_*^\alpha(\Omega_m)} \le C\sum_{1 \le i \le m} \Vert f\Vert_{\mathring{C}_*^\alpha(\Omega_m^i)},\quad \Vert R_{ij}f\Vert_{C_*^\alpha(\Omega_m)} \le C\sum_{1 \le i \le m} \Vert f\Vert_{C_*^\alpha(\Omega_m^i)}.
	\end{split}
	\end{equation*}  In particular, we have \begin{equation*}
\begin{split}
\nrm{R_{ij}f}_{ C^\alpha\cap\mathring{C}^\alpha(\Omega_m) } \le C \sum_{1 \le i \le m}\nrm{f}_{C^\alpha\cap\mathring{C}^\alpha(\Omega_m^i)}.
\end{split}
\end{equation*} 
\end{lemma} \begin{remark}
Without the odd symmetry, the $L^\infty$ bound fails. A simple example is provided by the indicator function $\textbf{1}_{\Omega_4}$ with a radial cut-off. Explicit computations (cf. \cite{Bertozzi1991,SVP1,CS}) show that the double Riesz transforms have logarithmic singularity at the origin. The above lemma shows that the situation becomes completely different when one considers the odd symmetrization, which is simply $\textbf{1}_{\Omega_4^1} - \textbf{1}_{\Omega_4^5}$. 
\end{remark}

\subsubsection{A few explicit computations}\label{subsubsec:explicit}

The following explicit computation will clarify when we can expect H\"older estimates for the double Riesz transforms for functions defined in $\Omega_2$. Given a function $f$ on $\Omega_2$, we denote $\tilde{f}$ to be its odd-odd extension onto all of $\mathbb{R}^2$: that is, $\tilde{f} \equiv f $ on $\Omega_2$ and $\tilde{f}(x_1,-x_2) = -\tilde{f}(x_1,x_2) = \tilde{f}(-x_1,x_2)$ for all $x_1, x_2 \in \mathbb{R}$.\footnote{Strictly speaking $\tilde{f}$ is not well-defined on the axes but this will not cause any issues.}
 
\begin{example}
	We take the following function on $\Omega_2$ in polar coordinates: \begin{equation*}
	\begin{split}
	f(x) = r^\alpha \chi(r) \sin(m\theta)
	\end{split}
	\end{equation*} for some $0 \le \alpha \le 1$ and $m\ge 2$ even. Here $\chi$ is a smooth cutoff function satisfying $\chi(r)=1$ for $r\le1$ and $0$ for $r\ge2$. Once we make the ansatz \begin{equation*}
	\begin{split}
	\Delta^{-1}_{\Omega_2}f = \psi(r)\sin(m\theta),
	\end{split}
	\end{equation*} (note that this satisfies the Dirichlet boundary condition on $\partial \Omega_2$) then from \begin{equation*}
	\begin{split}
	\Delta = \frac{1}{r^2}\rd_{\theta\theta} + \rd_{rr} + \frac{1}{r}\rd_r,
	\end{split}
	\end{equation*}  we obtain an ordinary differential equation for $\psi$: \begin{equation*}
	\begin{split}
	\psi'' + \frac{1}{r}\psi' - \frac{m^2}{r^2}\psi = r^\alpha\chi(r). 
	\end{split}
	\end{equation*} Together with boundary conditions $\psi(0) = \psi'(0) = 0$, the solution is unique. We consider the region $r \le 1$ and it is easy to see that \begin{equation}\label{eq:stream_answer}
	\begin{split}
	\psi(r) = \frac{1}{(2+\alpha)^2 - m^2} r^{2+\alpha}
	\end{split}
	\end{equation} is the solution whenever the denominator is nonzero. In the special case when $m = 2$ and $\alpha = 0$, the solution is instead given by \begin{equation}\label{eq:stream_answer_2}
	\begin{split}
	\psi(r) = \frac{1}{4}r^2\ln r. 
	\end{split}
	\end{equation} Note that this expression can be obtained by fixing some $0 < r \le 1$ in \eqref{eq:stream_answer} with $m = 2$ and taking the limit $\alpha \rightarrow 0^+$. Hence, we see that while the function $\sin(2\theta)\chi(r)$ belongs to $\mathring{C}^\alpha(\Omega_2)$ with any $0 < \alpha \le 1$, $$G(x):= \rd_{12} \Delta^{-1}_{\Omega_2}(\sin(2\theta)\chi(r)) \notin L^\infty_{loc}(\Omega_2).$$ On the other hand, one can see that since \begin{equation}\label{eq:G}
	\begin{split}
	G = \frac{1}{4}\ln|x| - \frac{1}{2} \frac{x_1^2x_2^2}{|x|^4} + \tilde{F}
	\end{split}
	\end{equation} with $\tilde{F} \in C^1(\Omega_2)$, $G$ has a scale-invariant H\"older regularity in the sense that  \begin{equation}\label{eq:G2}
	\begin{split}
	\sup_{ x \ne 0, |x|\le |x'|} |x|^\alpha \frac{|G(x)-G(x')|}{|x-x'|^\alpha}<+\infty 
	\end{split}
	\end{equation} for all $0 < \alpha \le 1$. In the case $\alpha=1$ this is trivial from $|r \nb(\ln r)| \lesssim 1$. Similarly, we see that up to a bounded and smooth term denoted by $\tilde{F}_{11}$, \begin{equation*}
	\begin{split}
	\rd_{11} \Delta^{-1}_{\Omega_2}(\sin(2\theta)\chi(r)) = \frac{x_1x_2}{|x|^2}\left(2+  \frac{x_2^2-x_1^2}{|x|^2}\right) + \tilde{F}_{11}. 
	\end{split}
	\end{equation*} A similar computation can be done for $\rd_{11}$ replaced with $\rd_{22}$ as well. In particular, $\rd_{11} \Delta^{-1}_{\Omega_2}(\sin(2\theta)\chi(r))$ and $\rd_{22} \Delta^{-1}_{\Omega_2}(\sin(2\theta)\chi(r))$ has the same scale-invariant H\"older regularity as in \eqref{eq:G2}. This property will be fundamental in what follows. 
	
\end{example}

We are now ready to perform explicit computations for the Bahouri-Chemin function, which is directly relevant for the estimate in Lemma \ref{lem:2D-2}. 
\begin{lemma}[Estimates for the Bahouri-Chemin]
	Consider the function \begin{equation*}
	\begin{split}
	g(x_1,x_2) = \chi(|x|)\mathrm{sgn}(x_1x_2),
	\end{split}
	\end{equation*} where $\chi(|x|)$ is a smooth cutoff function satisfying $\chi = 1$ on $|x|\le1$ and $\chi = 0$ on $|x| \ge 2$. Then, the function $h$ defined by \begin{equation*}
	\begin{split}
	h(x) := g(x) - 2\chi(|x|)\frac{x_1x_2}{|x|^2}
	\end{split}
	\end{equation*} satisfies \begin{equation*}
	\begin{split}
	\Vert R_{ij}h\Vert_{\mathring{C}^1(\Omega_2)} \lesssim 1.
	\end{split}
	\end{equation*} Moreover, for any $0 < \alpha < 1$, \begin{equation*}
	\begin{split}
	\Vert R_{ij}(|x|^\alpha g) \Vert_{C^\alpha_{loc}(\Omega_2)} \lesssim_\alpha 1. 
	\end{split}
	\end{equation*}
\end{lemma}
\begin{proof}
	We note that \begin{equation*}
	\begin{split}
	\mathrm{sgn}(x_1x_2) = \sum_{k \ge 0} \frac{1}{(2k+1)} \sin(2(2k+1)\theta),
	\end{split}
	\end{equation*} and therefore \begin{equation*}
	\begin{split}
	h(x) = \chi(|x|) \cdot  \sum_{k \ge 1} \frac{1}{(2k+1)} \sin(2(2k+1)\theta) =: \chi(|x|) A(\theta).
	\end{split}
	\end{equation*} For $m \ge 3$, we explicitly have that \begin{equation*}
	\begin{split}
	(-\Delta)^{-1}_{\bbR^2}(\sin(m\theta)) = \frac{r^2}{m^2-4}\sin(m\theta).
	\end{split}
	\end{equation*} We compute that \begin{equation*}
	\begin{split}
	\rd_{x_1}\rd_{x_2}(-\Delta)^{-1}(\sin(m\theta)) = \frac{cm^2}{m^2-4}\left( \cos((m+2)\theta) - \cos(m\theta) \right) + O(\frac{1}{m}).
	\end{split}
	\end{equation*} From the above it is not difficult to see that, recalling the definition of $A$, \begin{equation*}
	\begin{split}
	\left(\rd_{x_1}\rd_{x_2}(-\Delta)^{-1} A\right)(\theta) \in C^1([0,\pi/2]),
	\end{split}
	\end{equation*} and therefore we deduce that $h \in \mathring{C}^1(\Omega_2). $
	
	To prove the second statement, we may only consider the region $r \le 1/2$, and note that \begin{equation*}
	\begin{split}
	\Delta^{-1}(|x|^\alpha g) = \sum_{k \ge 0} \frac{1}{2k+1} \frac{\sin( 2(2k+1)\theta )}{(2+\alpha)^2 - 4(2k+1)^2}  r^{2+\alpha}.
	\end{split}
	\end{equation*} Then, \begin{equation*}
	\begin{split}
	\frac{\rd_{\theta\theta}}{r^2} \Delta^{-1}(|x|^\alpha g) &= - \sum_{k \ge 0} \frac{1}{2k+1} \frac{4(2k+1)^2}{(2+\alpha)^2 - 4(2k+1)^2}  \sin( 2(2k+1)\theta )r^{ \alpha} \\
	&= |x|^\alpha g - \sum_{k \ge 0} \frac{1}{2k+1} \frac{(2+\alpha)^2}{(2+\alpha)^2 - 4(2k+1)^2} \sin( 2(2k+1)\theta )r^{ \alpha}.
	\end{split}
	\end{equation*} From the last expression it is clear that \begin{equation*}
	\begin{split}
	\left\Vert \frac{\rd_{\theta\theta}}{r^2} \Delta^{-1}(|x|^\alpha g)  \right\Vert_{C^\alpha(\Omega_2 \cap \{ |x| \le 1/2 \})} \lesssim_\alpha 1. 
	\end{split}
	\end{equation*} Other second order derivatives can be treated in a similar way. 
\end{proof}
Combining the previous lemma with the example above, we conclude the following: \begin{corollary}\label{cor:BC}
	We have \begin{equation*}
	\begin{split}
	\nrm{R_{ij}(\chi(|x|)\mathrm{sgn}(x_1x_2)) - c \ln |x| }_{\mathring{C}^1(\Omega_2)} \lesssim 1,
	\end{split}
	\end{equation*} where $c = \frac{1}{4}$ for $R_{12}$ and $c=0$ for $R_{11}$ and $R_{22}$. 
\end{corollary}

We now consider the simplest functions which satisfy the assumptions of Lemma \ref{lem:2D-3}, namely piecewise constant functions. We explicitly show that their double Riesz transforms are given again by piecewise constants, which must be if Lemma \ref{lem:2D-3} were valid. 
\begin{example}\label{example:Riesz}
	We now consider the following functions on $\bbR^2$: \begin{equation*}
	\begin{split}
	\tilde{{\bf 1}}_{R_4} := \sum_{i=1}^{8} (-1)^{i-1} \mathbf{1}_{\Omega_4^i},
	\end{split}
	\end{equation*} \begin{equation*}
	\begin{split}
	\tilde{{\bf 1}}_{\Omega_4} :=  \mathbf{1}_{\Omega_4^1}-\mathbf{1}_{\Omega_4^5}.
	\end{split}
	\end{equation*} Using the definition of the principal value integral, one can explicitly compute that for $1\le i,j\le2$, $R_{ij}\tilde{{\bf 1}}_{R_4}$ is given by piecewise constant functions depicted in Figure \ref{fig:Riesz1}. The same holds for $R_{ij}\tilde{{\bf 1}}_{\Omega_4}$, which is depicted in Figure \ref{fig:Riesz2}. In both cases, one can verify that $R_{11}+R_{22}=-\mathrm{Id}$. To illustrate how one can obtain these facts, we consider the case $R_{12}$ applied to $\tilde{{\bf 1}}_{\Omega_4}$, evaluated at $x = (x_1,x_2)$ with $0<x_2<x_1$. By definition, we have \begin{equation*}
	\begin{split}
	R_{12}\tilde{{\bf 1}}_{\Omega_4}(x) = \frac{1}{\pi} \lim_{R\rightarrow\infty}\left[\int_{\Omega_4^1 \cap[0,R]^2} \frac{(x_1-y_1)(x_2-y_2)}{|x-y|^4} dy - \int_{\Omega_4^5\cap[-R,0]^2} \frac{(x_1-y_1)(x_2-y_2)}{|x-y|^4} dy\right].
	\end{split}
	\end{equation*} We evaluate the first integral as follows: \begin{equation*}
	\begin{split}
	&\int_{\Omega_4^1 \cap[0,R]^2} \frac{(x_1-y_1)(x_2-y_2)}{|x-y|^4} dy  = \frac{1}{2}\int_{ 0\le y_1\le R} \frac{x_1-y_1}{(x_1-y_1)^2+(x_2-y_2)^2} - \frac{x_1-y_1}{(x_1-y_1)^2+x_2^2}  dy_1 \\
	&\quad = \frac{1}{8}\left( 2\arctan(\frac{x_1+x_2}{x_1-x_2}) - 2\arctan(\frac{x_1+x_2-2R}{x_1-x_2})  - \ln(x_1^2+x_2^2)\right. \\
	&\quad\quad\quad\quad  \left.  +2\ln(x_2^2+(x_1-R)^2)-\ln(x_1^2+x_2^2-2x_1R-2x_2R+2R^2) \right). 
	\end{split}
	\end{equation*} Similarly evaluating the second integral and subtracting, we obtain that \begin{equation*}
	\begin{split}
	&R_{12}\tilde{{\bf 1}}_{\Omega_4}(x) = \frac{1}{8\pi} \lim_{R\rightarrow\infty} \left[ - 2\arctan(\frac{x_1+x_2-2R}{x_1-x_2})+ 2\arctan(\frac{x_1+x_2+2R}{x_1-x_2})+2\ln(x_2^2+(x_1-R)^2) \right. \\ 
	&\qquad\left. -2\ln(x_2^2+(x_1+R)^2) -\ln(x_1^2+x_2^2-2x_1R-2x_2R+2R^2) +\ln(x_1^2+x_2^2+2x_1R+2x_2R+2R^2)  \right]. 
	\end{split}
	\end{equation*} Then it is clear that the logarithmic terms cancel each other, and the limit is exactly $\frac{1}{4}$, independent of $x_1,x_2$ as long as $0<x_2<x_1$. 
	
	One may perform similar computations for functions corresponding to any $m\ge3$. We omit the details. 
\end{example}

 \begin{figure}
 	\includegraphics[scale=0.5]{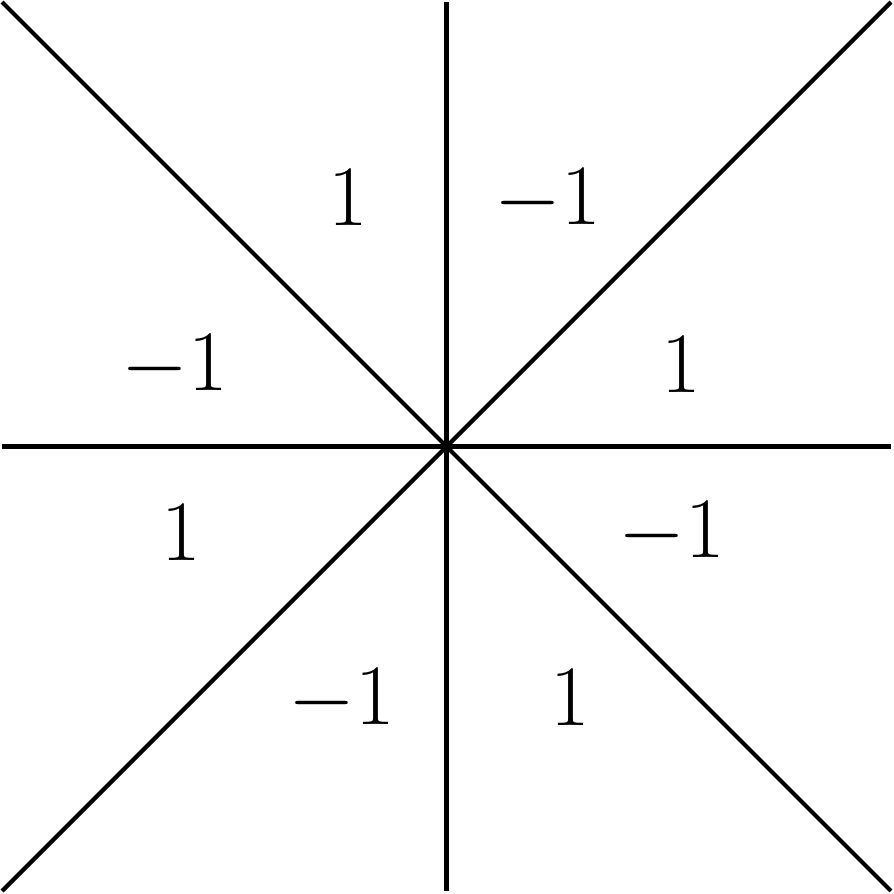}\qquad  \includegraphics[scale=0.5]{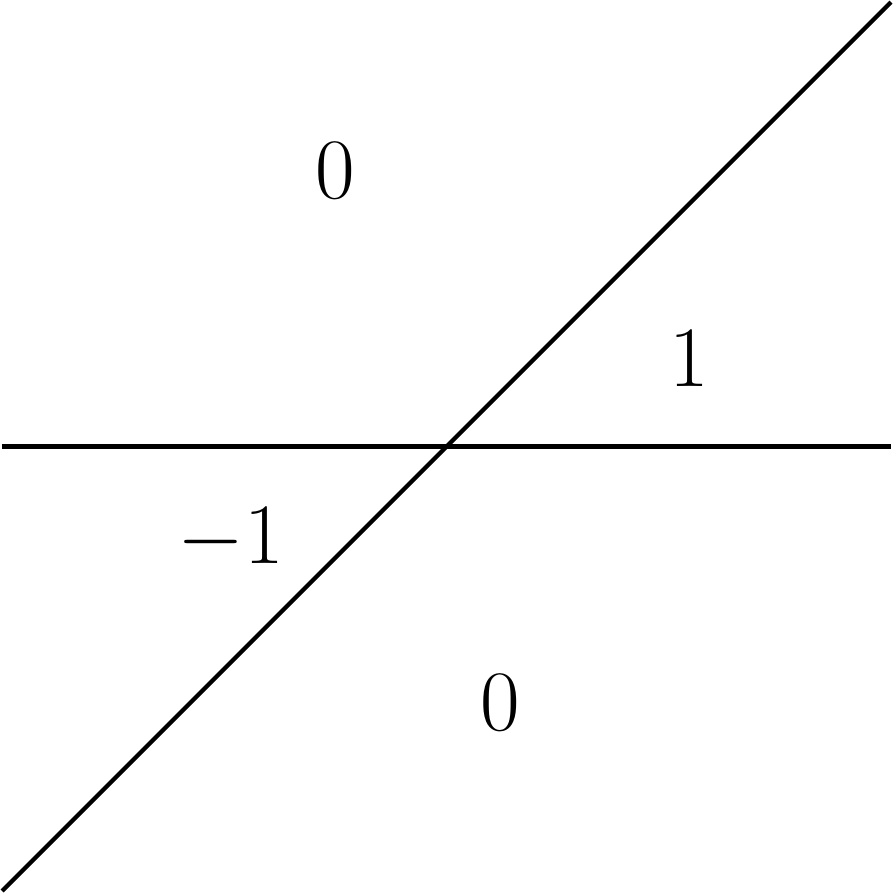} 
 	\centering
 	\caption{The functions $\tilde{{\bf 1}}_{R_4}$ and $\tilde{{\bf 1}}_{\Omega_4}$ (left, right).}
 	\label{fig:functions}
 \end{figure}

\begin{figure}
	\includegraphics[scale=0.5]{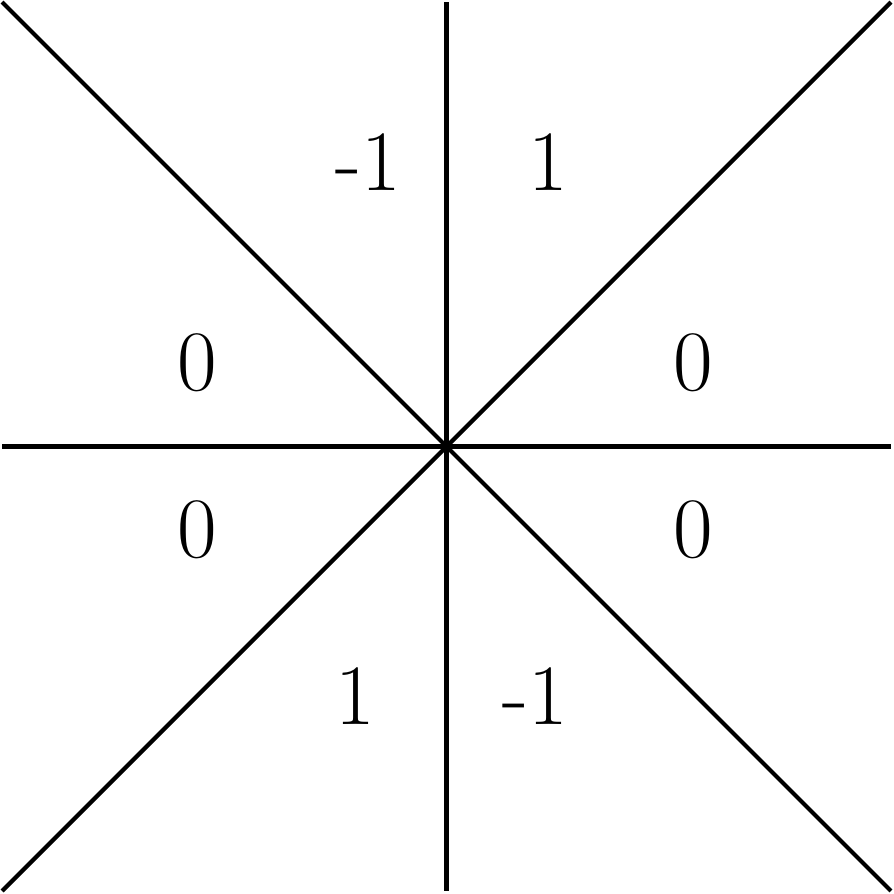}\qquad  \includegraphics[scale=0.5]{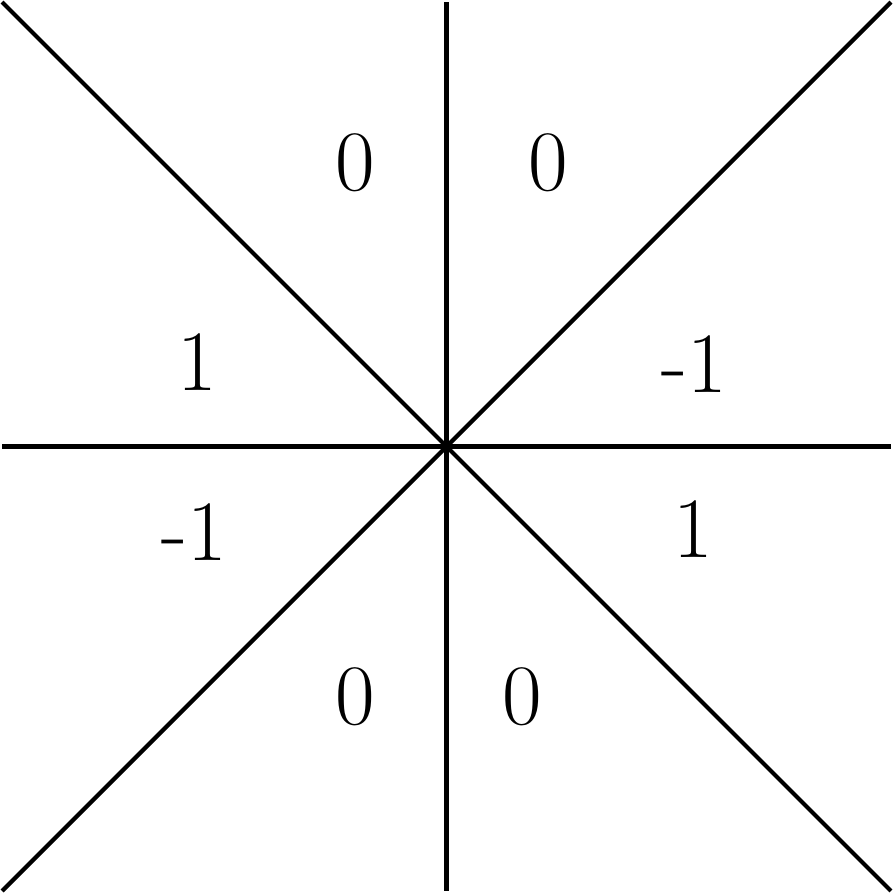}\qquad  \includegraphics[scale=0.5]{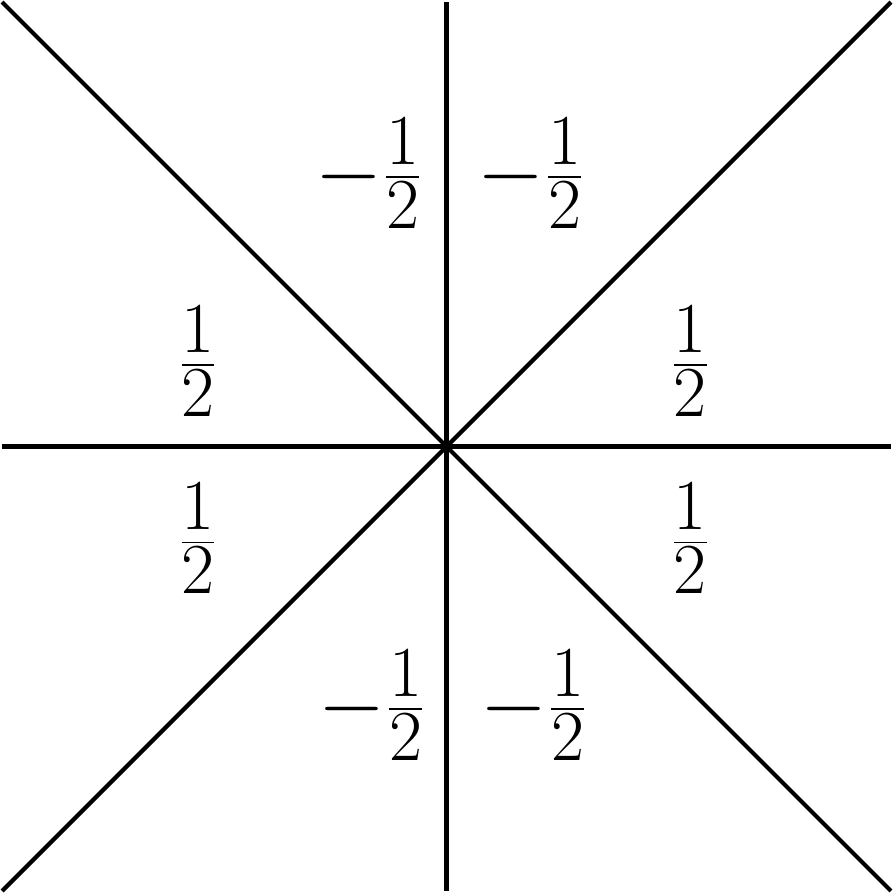} 
	\centering
	\caption{The operators $R_{11}, R_{22}, R_{12}$ acting on $\tilde{{\bf 1}}_{R_4}$ (from left to right).}
		\label{fig:Riesz1}
	\end{figure}

\begin{figure}
	\includegraphics[scale=0.5]{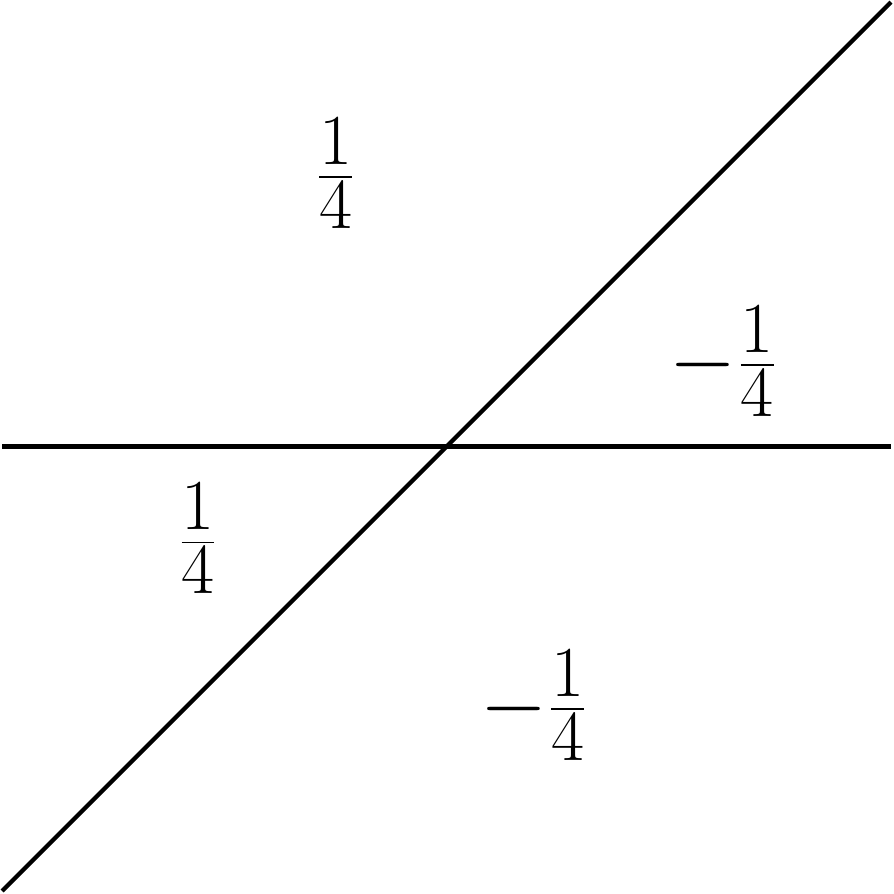}\qquad  \includegraphics[scale=0.5]{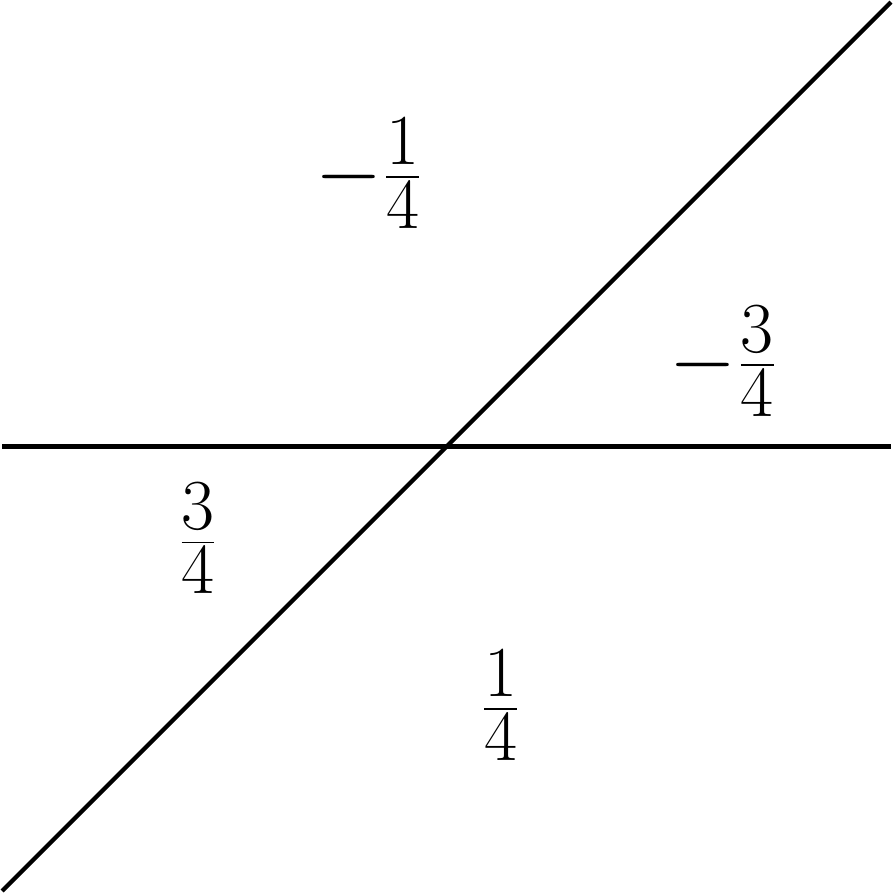}\qquad  \includegraphics[scale=0.5]{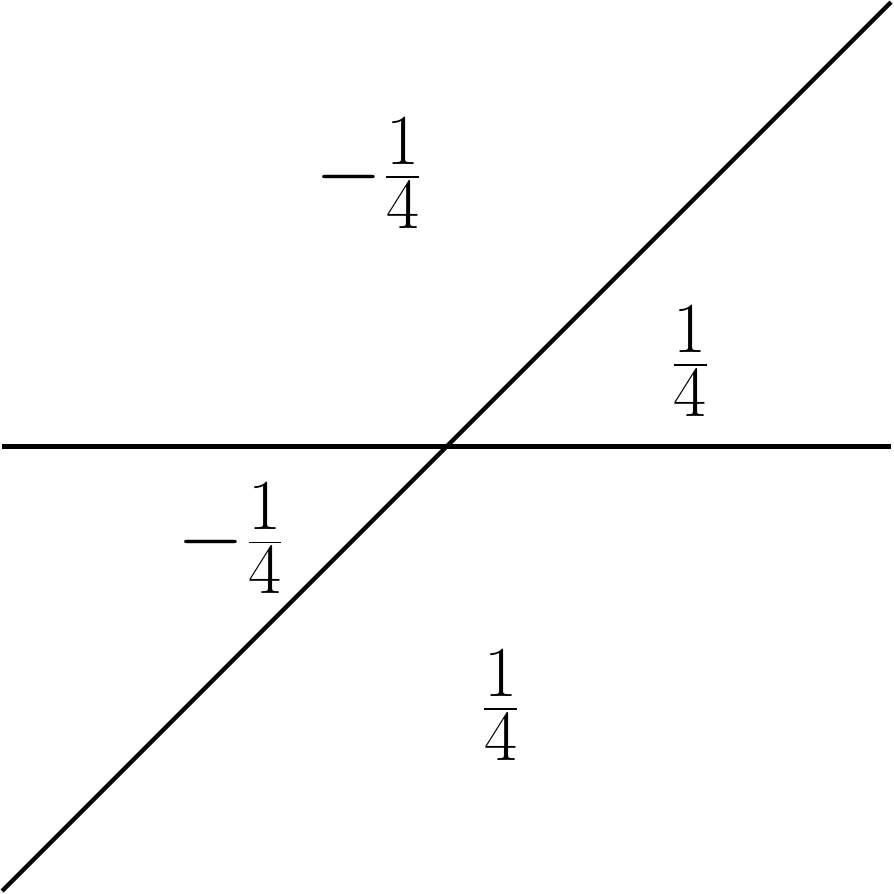} 
	\centering
	\caption{The operators $R_{11}, R_{22}, R_{12}$ acting on $\tilde{{\bf 1}}_{\Omega_4}$ (from left to right).}
	\label{fig:Riesz2}
\end{figure}

\begin{lemma}[$L^\infty$-estimate for $\mathring{C}^\alpha$ functions]\label{lem:C^circlealpha_quad}
	Let $f \in \mathring{C}^\alpha(\Omega_2)$ and assume further that $f$ satisfies the following cancellation properties: \begin{equation}\label{eq:cancellation}
	\begin{split}
	\int_{ \{|y| \ge \epsilon\} \cap \Omega_2} \frac{y_1y_2}{|y|^4} f(y) dy = \int_{ \{|y| \ge \epsilon\} \cap \Omega_2} \frac{y_1^2-y_2^2}{|y|^4} f(y) dy  = 0
	\end{split}
	\end{equation} for all $\epsilon > 0$. Then, we have the logarithmic bound \begin{equation}\label{eq:logbound_quadrant}
	\begin{split}
	\nrm{R_{ij} f}_{L^\infty} \le C \nrm{f}_{L^\infty}\left( 1 + \log\left(1 + \frac{\nrm{f}_{\mathring{C}^\alpha}}{\nrm{f}_{L^\infty}} \right) \right).
	\end{split}
	\end{equation}
\end{lemma}

\begin{remark}
	The cancellation property in \eqref{eq:cancellation} is essential for the $L^\infty$-bound, as one can see from the explicit computation for the Bahouri-Chemin function. (This is not necessary for the $\mathring{C}^\alpha_*$-estimate.) On the other hand, it is not necessary for the function to have compact support. 
\end{remark}

\begin{proof}
	We need to consider the integrals (defined by the principal value) \begin{equation*}
	\begin{split}
	\int_{\mathbb{R}^2} \frac{(x_1-y_1)(x_2-y_2)}{|x-y|^4} \tilde{f}(y)dy,\qquad \int_{\mathbb{R}^2} \frac{(x_1-y_1)^2-(x_2-y_2)^2}{|x-y|^4} \tilde{f}(y)dy,
	\end{split}
	\end{equation*} where $x \in \Omega_2$. Let us treat the first integral. Consider regions (i) $\{ |x-y| \le l|x| \}$, (ii) $\{ l|x|< |x-y|\} \cap \{ 2|x| <|y| \} $, and (iii) $\{ 2|x| < |y| \}$. Here $l \le 1/2$ is a number to be specified below. We first treat the regions (ii) and (iii): regarding (ii), we note that the domain is contained in the set $\{l|x|< |x-y| \le 3|x|\}$, so that  \begin{equation*}
	\begin{split}
	&\left| \int_{} \frac{(x_1-y_1)(x_2-y_2)}{|x-y|^4} \tilde{f}(y)dy \right| \le	  \int_{l|x|< |x-y| \le 3|x|}	\left| \frac{(x_1-y_1)(x_2-y_2)}{|x-y|^4} \tilde{f}(y) \right|dy \\
	&\qquad \le C\nrm{f}_{L^\infty} \int_{l|x|< |x-y| \le 3|x|} \frac{1}{|x-y|^2} dy \le C\nrm{f}_{L^\infty}\ln \frac{3}{l}.
	\end{split}
	\end{equation*} Then, for (iii), we may re-write the integral as \begin{equation*}
	\begin{split}
	\int_{|y| > 2|x| } \left[  \frac{(x_1-y_1)(x_2-y_2)}{|x-y|^4}  - \frac{y_1y_2}{|y|^4}   \right] \tilde{f}(y) dy,
	\end{split}
	\end{equation*} and then the kernel in large brackets decays as $|y|^{-3}$ as $|y| \rightarrow +\infty$. Hence, we may bound the above by $C\nrm{f}_{L^\infty}$.
	Lastly, in region (i), we rewrite the integral as \begin{equation*}
	\begin{split}
	\int_{|x-y| \le  l|x|} \frac{(x_1-y_1)(x_2-y_2)}{|x-y|^4} (f(y)-f(x)\tilde{\mathbf{1}}(y))\sim dy + f(x) \int_{|x-y| \le  l|x|} \frac{(x_1-y_1)(x_2-y_2)}{|x-y|^4} \tilde{\mathbf{1}}(y) dy.
	\end{split}
	\end{equation*} From \begin{equation*}
	\begin{split}
	| (f(y)-f(x)\tilde{\mathbf{1}}(y))| \le \frac{C|x-y|^\alpha}{|x|^\alpha} \nrm{f}_{\mathring{C}^\alpha},
	\end{split}
	\end{equation*} we bound \begin{equation*}
	\begin{split}
	\left| \int_{|x-y| \le  l|x|} \frac{(x_1-y_1)(x_2-y_2)}{|x-y|^4} (f(y)-f(x)\tilde{\mathbf{1}}(y)) dy  \right| \le C\nrm{f}_{\mathring{C}^\alpha}\frac{1}{|x|^\alpha} \int_{|x-y| \le  l|x|} |x-y|^{\alpha-2} dy \le  Cl^\alpha \nrm{f}_{\mathring{C}^\alpha}.
	\end{split}
	\end{equation*} On the other hand, explicit computations show that the integral $$\int_{|x-y| \le  l|x|} \frac{(x_1-y_1)(x_2-y_2)}{|x-y|^4} \tilde{\mathbf{1}}(y) dy$$ is uniformly bounded in $x$. To see this, after a rescaling of variables $y = |x|z$, the integral equals \begin{equation*}
	\begin{split}
	\int_{|v-z| \le l} \frac{(v_1-z_1)(v_2-z_2)}{|v-z|^4} \tilde{\mathbf{1}}(z) dz,
	\end{split}
	\end{equation*} where $v :=x/|x|$. The integral vanishes when $v_1 > l$, since then $\tilde{\mathbf{1}} \equiv\mathbf{1}$ in the domain of integration. Assuming that $v_1 < l$, using polar coordinates centered at $v$, we further rewrite it as \begin{equation*}
	\begin{split}
	&\int_{v_1}^{l} \int_{-\pi}^{\pi} \frac{\sin(2\theta)}{r} \left( \mathbf{1}_{[-(\pi-\theta_r),\pi-\theta_r]} - \mathbf{1}_{[-\pi,-(\pi-\theta_r)]} - \mathbf{1}_{[\pi-\theta_r,\pi]}  \right)  d\theta dr
	\end{split}
	\end{equation*} where $\cos(\theta_r) = v_1/r$. This integral actually vanishes. At this point, we also note that \begin{equation*}
	\begin{split}
	&\int_{|x-y| \le  l|x|} \frac{(x_1-y_1)^2-(x_2-y_2)^2}{|x-y|^4} \tilde{\mathbf{1}}(y) dy \\
	&\qquad = \int_{v_1}^{l} \int_{-\pi}^{\pi} \frac{\cos(2\theta)}{r} \left( \mathbf{1}_{[-(\pi-\theta_r),\pi-\theta_r]} - \mathbf{1}_{[-\pi,-(\pi-\theta_r)]} - \mathbf{1}_{[\pi-\theta_r,\pi]}  \right)  d\theta dr \\
	&\qquad = c \int_{v_1}^{l} \frac{1}{r} \sin(\theta_r)\cos(\theta_r) dr,
	\end{split}
	\end{equation*} and since $\cos(\theta_r) = v_1/r$, the last expression is bounded by \begin{equation*}
	\begin{split}
	\le C\int_{v_1}^l \frac{v_1}{r^2}dr \le C v_1 \left( \frac{1}{v_1} - \frac{1}{l} \right) \le C. 
	\end{split}
	\end{equation*} This simple estimate will be referred to as a ``half-moon'' computation in what follows. (This is just a simple case of the so-called ``Geometric Lemma'' in \cite{BeCo}.) Collecting the bounds, we have \begin{equation*}
	\begin{split}
	\left|\int_{\mathbb{R}^2} \frac{(x_1-y_1)(x_2-y_2)}{|x-y|^4} \tilde{f}(y)dy\right| \le Cl^\alpha \nrm{f}_{\mathring{C}^\alpha} + C\nrm{f}_{L^\infty}\ln\frac{3}{l}. 
	\end{split}
	\end{equation*} Optimizing in $l$ gives the desired logarithmic bound \eqref{eq:logbound_quadrant}. The same arguments carry over to deal with the other expression \begin{equation*}
	\begin{split}
	\int_{ \mathbb{R}^2} \frac{(x_1-y_1)^2-(x_2-y_2)^2}{|x-y|^4}\tilde{f}(y)dy.
	\end{split}
	\end{equation*} The proof is complete.
\end{proof}

\subsubsection{Proof of the lemmas}\label{subsubsec:proofs}

In this section, we complete the proof of the lemmas stated in \ref{subsubsec:lemmas}. We omit the proof of Lemma \ref{lem:2D-1}, which is covered in \cite{Grisvard1985,EJB}. Alternatively, one can follow the arguments given for the proof of Lemma \ref{lem:2D-3} below, which is somewhat more involved. 

\begin{proof}[Proof of Lemma \ref{lem:2D-2}]
	We consider the kernel \begin{equation*}
	\begin{split}
	K(x,y) = \frac{(x_1-y_1)(x_2-y_2)}{|x-y|^4}
	\end{split}
	\end{equation*} for $R_{12}$ as well as its symmetrized version \begin{equation*}
	\begin{split}
	\tilde{K}(x,y) := K(x,y) - K(x,\hat{y}) +K(x,-y) - K(x,-\hat{y}),
	\end{split}
	\end{equation*} where $\hat{y} := (y_1,-y_2)$. (The proof for $R_{11}$ and $R_{22}$ will be completely analogous.) Then, we first need to estimate \begin{equation*}
	\begin{split}
	A(x):=\int_{\Omega_2} \tilde{K}(x,y)f(y)dy
	\end{split}
	\end{equation*} in $L^\infty$. By rescaling, we may assume without loss of generality that $R = 1$; i.e. $f(y)$ vanishes for $|y| \ge 1$. Moreover, in the above, we may insert a smooth cut-off $\chi(y)\ge0$ which satisfy $\chi(y)=1$ for $|y|=1$. Then, we write \begin{equation*}
	\begin{split}
	A(x) = \int_{\{|y| \le 1\} \cap \Omega_2} \tilde{K}(x,y)(f(y)-f(x)) \chi(y) dy + f(x) \int_{\{|y| \le 1\} \cap \Omega_2} \tilde{K}(x,y)\chi(y)dy =: I(x) + J(x) , 
	\end{split}
	\end{equation*}  and we bound \begin{equation*}
	\begin{split}
	|I(x)| = \left| \int_{\{|y| \le 1\} \cap \Omega_2} \tilde{K}(x,y)(f(y)-f(x))\chi(y) dy  \right| &\le \nrm{f}_{C_*^\alpha} \int_{\{|y| \le 1\} \cap \Omega_2} |\tilde{K}(x,y)| |x-y|^\alpha dy \\
	&\le C \nrm{f}_{C_*^\alpha} \int_{\{|y| \le 1\} \cap \Omega_2} \frac{dy}{|x-y|^{2-\alpha}} \\
	&\le C\nrm{f}_{C_*^\alpha} \int_{\{|y| \le 1\} \cap \Omega_2} \frac{dy}{|y|^{2-\alpha}} \le C \nrm{f}_{C_*^\alpha}
	\end{split}
	\end{equation*} as well as \begin{equation*}
	\begin{split}
	|J(x)| = \left| f(x)  \int_{\{|y| \le 1\} \cap \Omega_2} \tilde{K}(x,y) \chi(y)dy \right| \le C|f(x)| (1 + | \ln|x||) \le C(\nrm{f}_{C_*^\alpha} + \nrm{f}_{L^\infty}),
	\end{split}
	\end{equation*} where we have used that $|f(x)| \le \min\{ \nrm{f}_{C_*^\alpha} |x|^\alpha \mathbf{1}_{|x| \le 1}, \nrm{f}_{L^\infty} \}$.
	
	We now consider the difference $A(x) - A(x')$, for some $x \ne x' \in \Omega_2$. We may assume that $|x| \le |x'|$. We use the fact that (recall Corollary \ref{cor:BC}) \begin{equation}\label{eq:corBC}
	\begin{split}
	\int_{\{|y| \le 1\} \cap \Omega_2} \tilde{K}(x,y) \chi(y)dy = c \ln|x| + F(x)
	\end{split}
	\end{equation} where $F$ is bounded and belongs to $\mathring{C}^\alpha(\Omega_2)$. Take the region \begin{equation*}
	\begin{split}
	D := \{ |y| \le 1 : |y-x'| \le 2|x-x'| \}.
	\end{split}
	\end{equation*} Then, write \begin{equation*}
	\begin{split}
	I(x) - I(x') &= \left[ \int_{D\cap \Omega_2} + \int_{D^c \cap \Omega_2} \right] \tilde{K}(x,y)(f(y) - f(x)) dy - \left[ \int_{D\cap \Omega_2} + \int_{D^c \cap \Omega_2} \right] \tilde{K}(x',y)(f(y)-f(x'))dy \\
	&= \left[ \int_{D\cap \Omega_2}  \tilde{K}(x,y)(f(y) - f(x)) dy  \right] - \left[ \int_{D\cap \Omega_2}  \tilde{K}(x',y)(f(y) - f(x')) dy  \right] \\
	&\qquad +\left[ \int_{D^c \cap \Omega_2} \tilde{K}(x',y)(f(x')-f(x)) dy \right] +\left[ \int_{D^c \cap \Omega_2} (\tilde{K}(x,y)-\tilde{K}(x',y))(f(y)-f(x))dy \right] \\
	& =: (1) + (2) + (3) + (4).
	\end{split}
	\end{equation*} We first bound the terms $(1), (2)$, and $(4)$. Note that \begin{equation*}
	\begin{split}
	|(2)| \le C \nrm{f}_{C^\alpha_*} \int_{ |y-x'| \le 2|x-x'| } |y-x'|^{\alpha -2} dy \le C\nrm{f}_{C^\alpha_*}|x-x'|^\alpha 
	\end{split}
	\end{equation*} and the term $(1)$ can be treated in a parallel manner. Then, \begin{equation*}
	\begin{split}
	|(4)| \le C\nrm{f}_{C^\alpha_*}  \int_{D^c \cap \Omega_2} |\nabla\tilde{K}(x^*,y)||x-x'||y-x|^\alpha dy  
	\end{split}
	\end{equation*} where $x^*$ is some point lying on the line segment connecting $x$ and $x'$. Note that we have \begin{equation*}
	\begin{split}
	|\nabla\tilde{K}(x^*,y)| \le \frac{C}{|x-y|^3},\qquad y \in D^c. 
	\end{split}
	\end{equation*} Then we conclude that \begin{equation*}
	\begin{split}
	|(4)| \le  C\nrm{f}_{C^\alpha_*} |x-x'| \int_{|x-y| \ge |x-x'|} |y-x|^{\alpha-3} dy \le  C\nrm{f}_{C^\alpha_*} |x-x'|^\alpha.
	\end{split}
	\end{equation*} 
	Finally, we combine the remaining terms as follows: \begin{equation*}
	\begin{split}
	(3) + J(x) - J(x') &= - \left[ f(x) \int_{\Omega_2} (\tilde{K}(x',y) - \tilde{K}(x,y))\chi(y)dy \right] + (f(x) - f(x'))\int_{D \cap \Omega_2} \tilde{K}(x',y)\chi(y)dy \\
	&= -f(x)(F(x') - F(x)) - cf(x)\ln\frac{|x'|}{|x|} + (f(x) - f(x'))\int_{D \cap \Omega_2} \tilde{K}(x',y)\chi(y)dy 
	\end{split}
	\end{equation*} where $F$ is as in \eqref{eq:corBC}. Note that \begin{equation*}
	\begin{split}
	\left| \frac{-f(x)(F(x') - F(x))  }{|x-x'|^\alpha} \right| \le C\nrm{f}_{C^\alpha}
	\end{split}
	\end{equation*} since $F $ belongs to $\mathring{C}^\alpha$ and $|f(x)| \le \nrm{f}_{C^\alpha}|x|^\alpha \le \nrm{f}_{C^\alpha} |x'|^\alpha$. Next, \begin{equation*}
	\begin{split}
	\frac{1}{|x-x'|^\alpha}\left| f(x) \ln \frac{|x'|}{|x|} \right| \le \frac{|f(x)|}{|x-x'|^\alpha}  \ln\left(1 + \frac{|x'|-|x|}{|x|}\right) \le C\frac{|f(x)|}{|x-x'|^\alpha} \frac{|x'-x|^\alpha}{|x|^\alpha}  \le C \nrm{f}_{C^\alpha}.
	\end{split}
	\end{equation*} It only remains to obtain a uniform bound for the following integral:  \begin{equation*}
	\begin{split}
	\frac{|f(x)-f(x')|}{|x-x'|^\alpha} \left| \int_{D \cap \Omega_2} \tilde{K}(x',y)\chi(y) dy \right| \le \nrm{f}_{C^\alpha} \left| \int_{D \cap \Omega_2} \tilde{K}(x',y)\chi(y) dy \right|.
	\end{split}
	\end{equation*} We expand $\tilde{K}(x',y) = K(x',y) - K(x',\hat{y}) + K(x',-y) - K(x',-\hat{y})$ and bound each term separately. The main contribution comes from $K(x',y)$: \begin{equation*}
	\begin{split}
	\left|\int_{D \cap \Omega_2}  {K}(x',y) \chi(y) dy\right| \le \int_{ |x' - y| \le l|x'| } |K(x',y)| dy 
	\end{split}
	\end{equation*} where $l = 2|x-x'|/|x'|$. We could have assumed that $|x-x'| \le |x'|/2$ so that $l \le 1$. The fact that this expression is uniformly bounded in $x'$ and $0 < l \le 1$ follows from ``half-moon'' computations contained in the proof of Lemma \ref{lem:C^circlealpha_quad}. The other terms in $\tilde{K}$ can be treated similarly. 
\end{proof}

\begin{proof}[Proof of Lemma \ref{lem:2D-3}]
	We only consider the case $m = 4$, argument for the case $m=3$ being completely parallel. Without loss of generality, we can assume that the function $f$ is supported only on $\Omega_4^1 \cup \Omega_4^5$. We define $\tilde{\textbf{1}}_{\Omega_4}:=\textbf{1}_{\Omega_4^1} - \textbf{1}_{\Omega_4^5}$. We consider only the case of $R_{12}$ with kernel $K(x,y)$. 
	
	\medskip
	
	\noindent (i) $L^\infty$ bound
	
	\medskip
	
	\noindent We fix some $\textbf{0} \ne x \in \Omega_4^1$ and consider \begin{equation*}
	\begin{split}
	\int_{\bbR^2} K(x-y) f(y)dy = \left[ \int_{\{|x-y|<\ell|x|\}  }+ \int_{B}  + \int_{\{|y|>10|x|\}}  \right] K(x-y)f(y) dy = I+II+III,
	\end{split}
	\end{equation*} where $0<\ell\le\frac{1}{2}$ will be optimized later and $B = \mathbb{R}^2 \backslash ( \{ |x-y|<\ell|x|\}   \cup \{ |y|>10|x| \} ) $. First, we write \begin{equation*}
	\begin{split}
	I = \int_{\{|x-y|<\ell|x| \} \cap \Omega_4^1 } K(x-y)(f(y)-f(x)) dy \, + f(x) \int_{\{|x-y|< \ell|x| \} \cap \Omega_4^1 } K(x-y) dy. 
	\end{split}
	\end{equation*} Then, we observe that in the domain of integration, $|f(y)-f(x)| \le \nrm{f}_{\mathring{C}^\alpha(\Omega_4^1)}|x|^{-\alpha}|x-y|^\alpha$. The fact that the second term is uniformly bounded follows from ``half-moon'' computation again. Hence,  \begin{equation*}
	\begin{split}
	|I| \le \nrm{f}_{\mathring{C}^\alpha(\Omega_4^1)}|x|^{-\alpha}\int_{\{ |x-y|< \ell|x| \}} |x-y|^{\alpha-2} dy + C\nrm{f}_{L^\infty} \le C\left( \nrm{f}_{L^\infty} + \ell^\alpha \nrm{f}_{\mathring{C}^\alpha(\Omega_4^1)}  \right).
	\end{split}
	\end{equation*} Next, it is clear that \begin{equation*}
	\begin{split}
	|II| \le C\nrm{f}_{L^\infty} \ln \frac{20}{\ell}. 
	\end{split}
	\end{equation*} Lastly, we note that from the odd symmetry of $f$, \begin{equation*}
	\begin{split}
	\int_{ \{ |y|>10|x| \} } \frac{y_1y_2}{|y|^4} f(y)dy = 0
	\end{split}
	\end{equation*} in the sense of principal value integration. Therefore we estimate \begin{equation*}
	\begin{split}
	|III| \le C\left| \int_{ \{ |y|>10|x| \} }  \left[\frac{(x_1-y_1)(x_2-y_2)}{|x-y|^4} -\frac{y_1y_2}{|y|^4}\right] f(y)dy \right| \le C\nrm{f}_{L^\infty} |x|\int_{ \{ |y|>10|x| \} }|y|^{-3}dy \le C\nrm{f}_{L^\infty}. 
	\end{split}
	\end{equation*} Optimizing in $\ell$ finishes the proof. 
	
	\medskip
	
	\noindent (ii) $C^\alpha_*$ bound 
	
	\medskip 
	
	\noindent We now take $x \ne  x' \in \Omega_4^1$, write $\tilde{\textbf{1}}(y):=\tilde{{\bf 1}}_{\Omega_4}(y)$ for simplicity, and estimate the difference \begin{equation*}
	\begin{split}
	R_{12}f(x)-R_{12}f(x')=\int_{\Omega_4^1\cup\Omega_4^5 } K(x-y)f(y)dy - \int_{\Omega_4^1\cup\Omega_4^5} K(x'-y)f(y) dy. 
	\end{split}
	\end{equation*} We shall write $\int = \int_{\Omega_4^1\cup\Omega_4^5}$ from now on and rewrite the above as \begin{equation*}
	\begin{split}
	&\int K(x-y)(f(y)-f(x)\tilde{\textbf{1}}(y)) dy - \int K(x'-y)(f(y)-f(x')\tilde{\textbf{1}}(y))  dy \\
	&\qquad  +  f(x) \int K(x-y)\tilde{\textbf{1}}(y) dy- f(x') \int K(x'-y)\tilde{\textbf{1}}(y) dy. 
	\end{split}
	\end{equation*} Recall from the explicit computations in \ref{subsubsec:explicit} that \begin{equation*}
	\begin{split}
	\int K(x-y)\tilde{\textbf{1}}(y) dy = \int K(x'-y)\tilde{\textbf{1}}(y) dy = \frac{1}{4}
	\end{split}
	\end{equation*} and hence \begin{equation*}
	\begin{split}
	\left|  f(x) \int K(x-y)\tilde{\textbf{1}}(y) dy- f(x') \int K(x'-y)\tilde{\textbf{1}}(y) dy. 
\right| \le C|f(x)-f(x')| \le C|x-x'|^\alpha\nrm{f}_{C^\alpha_*(\Omega_4)}. 
	\end{split}
	\end{equation*} We now turn to the first two terms and split the integral into $\{ |x-y| > 10|x-x'| \}$ and $\{ |x-y|\le 10|x-x'| \}$ for \textit{both} integrals. In the latter regions, we simply estimate \begin{equation*}
	\begin{split}
	&\left|\int_{\{ |x-y|\le 10|x-x'| \}} K(x-y)(f(y)-f(x)\tilde{\textbf{1}}(y)) dy \right| + \left|\int_{\{ |x-y|\le 10|x-x'| \}} K(x'-y)(f(y)-f(x')\tilde{\textbf{1}}(y)) dy \right| \\
	&\qquad \le C \nrm{f}_{C^\alpha_*(\Omega_4)}\int_{\{ |x-y|\le 10|x-x'| \}} |x-y|^{\alpha-2} dy \le C \nrm{f}_{C^\alpha_*(\Omega_4)}|x-x'|^\alpha.  
	\end{split}
	\end{equation*} Here, we have used the pointwise bound \begin{equation*}
	\begin{split}
	\left|f(y)-f(x)\tilde{\textbf{1}}(y)\right|\le |x-y|^\alpha \nrm{f}_{C^\alpha_*(\Omega_4)}
	\end{split}
	\end{equation*} (and with $x$ replaced by $x'$) which follows from the definition of $\tilde{{\bf 1}}$ and the odd symmetry of $f$. Then, we just combine the remaining terms to obtain \begin{equation*}
	\begin{split}
	\int_{\{ |x-y| > 10|x-x'| \}} (K(x-y)-K(x'-y))(f(y)-f(x') \tilde{\textbf{1}}(y) )dy - \int_{\{ |x-y| > 10|x-x'| \}} K(x-y) (f(x)-f(x'))\tilde{\textbf{1}}(y) dy. 
	\end{split}
	\end{equation*} Estimating the second term is straightforward, and for the first term, we use the mean value theorem as well as the decay of $\nabla K$ to obtain \begin{equation*}
	\begin{split}
	|K(x-y)-K(x'-y)| \le |\nabla K(x^*-y)||x-x'| \le C \frac{|x-x'|}{|x-y|^3}
	\end{split}
	\end{equation*} (where $x^*$ is some point lying on the line segment defined by $x,x'$) and then \begin{equation*}
	\begin{split}
	&\left|\int_{\{ |x-y| > 10|x-x'| \}} (K(x-y)-K(x'-y))(f(y)-f(x') \tilde{\textbf{1}}(y) )dy \right| \\
	&\qquad \le C\nrm{f}_{C^\alpha_*(\Omega_4)}|x-x'|\int_{\{ |x-y| > 10|x-x'| \}}|x-y|^{\alpha-3}dy \le C\nrm{f}_{C^\alpha_*(\Omega_4)}|x-x'|^\alpha. 
	\end{split}
	\end{equation*} Collecting the bounds, we obtain that \begin{equation*}
	\begin{split}
	\frac{|R_{12}f(x)-R_{12}f(x')|}{|x-x'|^\alpha} \le C \nrm{f}_{C^\alpha_*(\Omega_4)}. 
	\end{split}
	\end{equation*} The same proof carries over to the case when $x \ne x' \in \Omega_4^i$ for any $i>1$. 
	
	We omit the proof of the $\mathring{C}^\alpha_*$ bound, which is a straightforward adaptation of the $C^\alpha_*$ bound. 
\end{proof}

\subsection{Three-dimensional case}\label{subsec:estimate-3D}

Equipped with the two-dimensional H\"older estimates for the double Riesz transforms, we now move on to the corresponding 3D estimates which are directly responsible for the local regularity result \ref{thm:lwp-corner}. The goal of this section is to establish the following \begin{proposition}\label{prop:3D-estimate}
	Let $f = (f^1,f^2,f^3) \in (C^\alpha\cap\mathring{C}^\alpha(\tilde{U}))^3$ satisfy $f(\mathbf{0}) = 0$ with $f^1+f^2$ vanishing on $\{x_3=0,x_1=x_2\}\cap\tilde{U}$. Then, we have \begin{equation*}
	\begin{split}
	\nrm{(R_{ij}f)^k}_{C^\alpha\cap\mathring{C}^\alpha(\tilde{U})} \le C \left( \sum_{\ell=1}^3\nrm{f^\ell}_{C^\alpha\cap\mathring{C}^\alpha(\tilde{U})}  \right)
	\end{split}
	\end{equation*} for any $1 \le i,j,k\le3$. 
\end{proposition}
It is important to clarify the definition of $R_{ij}$, which is defined on the vector $f$ (not on individual components $f^k$). Recalling the extension rule for the vorticity in Definition \ref{def:extension}, we first extend $f$ to all of $\bbR^3$ and then apply $\rd_{x_i}\rd_{x_j}(-\lap)_{\bbR^3}^{-1}$, whose $k$-th component is what we define as $(R_{ij}f)^k$.

We note that the functions $f^\ell$ may not be compactly supported. Moreover, it suffices to assume that $f$ vanishes at the origin thanks to the explicit computations given in Subsection \ref{subsec:explicit}. Next, the $C^\alpha$ is trivial in $\tilde{U}$ away from the half-lines generated by $\frka_2 = (\frac{1}{\sqrt{2}},\frac{1}{\sqrt{2}},0)$, $\frka_3 = (\frac{1}{\sqrt{3}},\frac{1}{\sqrt{3}},\frac{1}{\sqrt{3}})$, and $\frka_4 = (1,0,0)$ since otherwise the boundary of $\tilde{U}$ is $C^\infty$-smooth. Let us denote those half-lines by $\vec{\frka}_d$ for $d = 2, 3, 4$. 

Hence it suffices to obtain H\"older estimates close to those half-lines. To this end we consider a partition of unity $\{ \chi_{\frka_d} \}_{d = 2,3,4 }$ on $\tilde{U}$: \begin{equation*}
\begin{split}
\sum_{d =2,3,4} \chi_{\frka_d} \equiv 1, 
\end{split}
\end{equation*} $\chi_{\frka_d}$ is supported in some cone containing $\vec{\frka}_d$ and away from half-lines generated by the others. We may take $\chi_{\frka_d}$ as radially 0-homogeneous functions and impose regularity $\mathring{C}^1(\bbR^3)$. We shall prove Proposition \ref{prop:3D-estimate} with $f$ replaced by $\chi_{\frka_d}f$ for $d = 2, 3, 4$. This is sufficient as we have \begin{equation*}
\begin{split}
 \nrm{\chi_{\frka_d}f}_{C^\alpha\cap\mathring{C}^\alpha(\tilde{U})} \le C \nrm{f}_{C^\alpha\cap\mathring{C}^\alpha(\tilde{U})}
\end{split}
\end{equation*}  as well as \begin{equation*}
\begin{split}
\nrm{f}_{C^\alpha\cap\mathring{C}^\alpha(\tilde{U})} \le \sum_{d} \nrm{\chi_{\frka_d}f }_{C^\alpha\cap\mathring{C}^\alpha(\tilde{U})} .
\end{split}
\end{equation*} The first inequality uses the product rule in $\mathring{C}^\alpha$ as well as the special product rule \begin{equation*}
\begin{split}
\nrm{gf}_{C^\alpha} \le C \nrm{g}_{\mathring{C}^\alpha}\nrm{f}_{C^\alpha}, \quad f(\textbf{0}) = 0. 
\end{split}
\end{equation*} In this section, we shall use the term ``$f$ is supported near $\vec{\frka}_d$'' to mean that $f = \chi_{\frka_d}f$. An alternative way to define this notion is as follows: for any $x \in \tilde{U}\cap\mathrm{supp}(f)$, $d(x,\vec{\frka}_d) \le c\min_{k\ne d}\{ d(x,\vec{\frka}_k) \}$ for some universal $c>0$.

Next, we make a simple observation on the invariance of double Riesz transforms under rotations: if $x, x'$ are two coordinate systems related by a rotation matrix $M$ such that $x' = Mx$, then we have that each double Riesz transform $R_{ij}$ defined in the $x$-coordinates is expressed by a linear combination of Riesz transforms $R_{i'j'}$ defined in the $x'$-coordinates (with coefficients depending only on the elements of $M$). This follows since $(-\lap)^{-1}$ is rotation invariant and one can explicitly represent $\rd_{x_i}\rd_{x_j}$ as a linear combination of second order derivatives in the $x'$-coordinates. Therefore, we have that if for some function $f$ and norm $\nrm{ \cdot }_X$, if $\nrm{R_{ij}f}_{X}\le A$ for all $i,j$ then the same property holds with double Riesz transforms in the $x'$-coordinates with $A$ possibly replaced with $CA$ where $C>0$ is an absolute constant (since the matrix norm of a rotation satisfies $\nrm{M}\lesssim 1$). 


Restricting the function near the singular half-lines has an additional simplifying consequence, which we explain in detail in the context of two dimensions. Recall from the introduction that the analogous domain to $\tilde{U}$ in 2D is given by $\Omega_4^1 = \{ (r,\theta): 0<\theta < \frac{\pi}{4} \} $. Given $g \in L^\infty(\Omega_4^1)$, the natural extension $\tilde{g}$ is obtained by keep reflecting $g$ along the boundaries of $\Omega_4^i$. That is, \begin{equation*}
\begin{split}
\tilde{g}(x) = \tilde{g}(x^\perp),\quad \tilde{g}(x_1,x_2)=-\tilde{g}(x_1,-x_2).
\end{split}
\end{equation*} Similarly to what we have done in the above, using 0-homogeneous cutoff functions, we can decompose $g=g_1+g_2$ where $g_1$ and $g_2$ are respectively supported near the half-line $\{ r>0,\theta=0 \}$ and $\{ r>0,\theta=\frac{\pi}{4} \}$. Then, we further decompose the   extensions $\tilde{g}_1$ and $\tilde{g}_2$ by \begin{equation*}
\begin{split}
\tilde{g}_1 = \tilde{g}_1^{m} + \tilde{g}_1^{r}, \quad \tilde{g}_2 = \tilde{g}_2^{m} + \tilde{g}_2^{r}
\end{split}
\end{equation*} where $\tilde{g}_1^{m} = \tilde{g}_1 \cdot \textbf{1}_{ \{ (r,\theta): -\frac{\pi}{4}<\theta<\frac{\pi}{4} \} }$ and $\tilde{g}_2^{m} = \tilde{g}_2 \cdot \textbf{1}_{ \{ (r,\theta):0<\theta<\frac{\pi}{2} \} }$. That is, $\tilde{g}_{\ell}^{m}$ is the part of $\tilde{g}_{\ell}$ restricted to fundamental domains \textit{adjacent} to the support of $g_{\ell}$. We have the following support separation property: \begin{equation*}
\begin{split}
x \in \Omega_4^1 \implies d(x,\mathrm{supp}(\tilde{g}_1^{r} )) \gtrsim |x|,
\end{split}
\end{equation*} where $d(x,A)=\inf_{y\in A}|x-y|$ with $A \subset \bbR^2$. The superscripts $m$ and $r$ refer to ``main'' and ``remainder'', respectively; the following proposition tells us why we can regard $\tilde{g}_{\ell}^r$ as a remainder term in $\tilde{g}_{\ell}$. 
\begin{proposition}\label{prop:symmetry}
	Let $g \in C^\alpha\cap\mathring{C}^\alpha(\Omega_4^1)$ satisfy $g(\mathbf{0})=0$. Then we have \begin{equation*}
	\begin{split}
	\nrm{R_{ij} \tilde{g}_{\ell}^{r} }_{\mathring{C}_*^\alpha(\Omega_4^1)} \le C \nrm{g_\ell}_{\mathring{C}_*^\alpha(\Omega_4^1)}, \quad \nrm{R_{ij} \tilde{g}_{\ell}^{r} }_{{C}^\alpha_*(\Omega_4^1)} \le C \nrm{g_\ell}_{{C}^\alpha_*(\Omega_4^1)}
	\end{split} 
	\end{equation*} for any $1\le i,j,\ell\le2$. 
\end{proposition} 

The symmetry reduction lemma will be an immediate consequence of the following general estimate in $\bbR^n$. The statements as well as the proofs will be referred in later sections frequently. 
\begin{lemma}[Symmetry reduction lemma]\label{lem:symmetry-general}
	In $\bbR^n$, let $T$ be the convolution operator against a kernel $K$ satisfying \begin{equation*}
	\begin{split}
	|K(z)|\le C|z|^{-n},\quad |\nabla K(z)| \le C|z|^{-n-1}
	\end{split}
	\end{equation*} and $h: \bbR^n\rightarrow\bbR$ satisfy \begin{equation*}
	\begin{split}
	|h(x)| \le A|x|^\alpha,\quad \forall x \in \bbR^n
	\end{split}
	\end{equation*} for some constant $A>0$. Finally, let $\Omega\subset \bbR^n$ be a convex cone\footnote{This means that if $x\in\Omega$, then $\lambda x \in \Omega$ for all $\lambda>0$.} satisfying the following support separation property: \begin{equation*}
	\begin{split}
	\begin{split}
	x \in \Omega \implies d(x,\mathrm{supp}(h)) \ge c |x|
	\end{split}
	\end{split}
	\end{equation*} where $c>0$ is some universal constant. Then, we have that \begin{equation*}
	\begin{split}
	\nrm{T[h]}_{C^\alpha_*(\Omega)}\le CA,\quad \nrm{T[h]}_{\mathring{C}^\alpha_*(\Omega)}\le C. 
	\end{split}
	\end{equation*} Furthermore, if $\mathrm{supp}(h) \subset B_0(R)$, then \begin{equation*}
	\begin{split}
	\nrm{T[h]}_{L^\infty(\Omega)} \le C(R)(1+A). 
	\end{split}
	\end{equation*}
\end{lemma}
\begin{proof}
	Take some $x \ne x' \in \Omega$ and let us estimate \begin{equation*}
	\begin{split}
	\int_{\bbR^n} K(x-y)h(y)dy - \int_{\bbR^n} K(x'-y)h(y)dy,
	\end{split}
	\end{equation*} We may assume, without loss of generality, that $|x'|\le|x|$. We consider two cases: (i) $|x'|>\frac{1}{10}|x-x'|$ and (ii) $|x'|\le \frac{1}{10}|x-x'|$. 
	
	In the case (i), we have that $|x'|>\frac{1}{2}|x|$. We then directly estimate \begin{equation*}
	\begin{split}
	\left|\int_{\bbR^n} K(x-y)h(y)dy - \int_{\bbR^n} K(x'-y)h(y)dy\right| \le A|x-x'| \int_{ \mathrm{supp}(h) } |\nabla K(x^*-y)| |y|^\alpha dy ,
	\end{split}
	\end{equation*} where $x^* = \lambda x + (1-\lambda)x'$ for some $0\le\lambda=\lambda(y)\le1$. We have that $|x^*|\approx |x'|\approx |x|$. Moreover, from the support separation property, $|x^*-y|\gtrsim|x|$ for $y \in \mathrm{supp}(h)$. Hence, we bound the above by \begin{equation*}
	\begin{split}
	&\le CA|x-x'|\left[\int_{ \mathrm{supp}(h)\cap  \{ |y|\le 10|x'| \} } |x^*-y|^{-n-1}|y|^\alpha dy +  \int_{ \{ |y|> 10|x'| \} } |x^*-y|^{-n-1}|y|^\alpha dy\right] \\
	&\le CA|x-x'|\left[ |x'|^n|x|^{-n-1}|x'|^\alpha + |x'|^{\alpha-1} \right] \le CA|x-x'|^\alpha. 
	\end{split}
	\end{equation*} We now treat the case (ii). Then we consider the integral \begin{equation*}
	\begin{split}
	\int_{ \{ |x'-y| <10|x-x'| \} } K(x'-y)h(y)dy = \left[\int_{ \{ |x'-y| <10|x-x'| \} \cap \{ |y|\le 2|x'| \} } + \int_{ \{ |x'-y| <10|x-x'| \} \cap \{ |y|> 2|x'| \} }\right] K(x'-y)h(y)dy. 
	\end{split}
	\end{equation*} The first term is bounded in absolute value by \begin{equation*}
	\begin{split}
	CA|x'|^n|x'|^{-n}|x'|^\alpha \le CA|x-x'|^\alpha 
	\end{split}
	\end{equation*} where we have used the support separation property to deduce $|K(x'-y)|\le C|x'-y|^{-n} \le C|x'|^{-n}$. On the other hand, in the second region we have $|x'-y|\ge c|y|$ and then the integral in absolute value is bounded by \begin{equation*}
	\begin{split}
	CA\int_{ 2|x'|<|y|<20|x'-x|} |y|^{\alpha-n} dy \le CA|x-x'|^\alpha. 
	\end{split}
	\end{equation*} Similarly, one can estimate \begin{equation*}
	\begin{split}
	\left|\int_{ \{ |x'-y| <10|x-x'| \} } K(x-y)h(y)dy \right|\le CA|x-x'|^\alpha. 
	\end{split}
	\end{equation*} In the region $ \{ |x'-y| \ge 10|x-x'| \}$, we combine the integrals to bound \begin{equation*}
	\begin{split}
	A|x-x'|\int_{  \{ |x'-y| \ge 10|x-x'| \} \cap \mathrm{supp}(h) } |\nabla K(x^*-y)||y|^\alpha dy. 
	\end{split}
	\end{equation*} We then observe for $y \in \mathrm{supp}(h)$ and satisfying $ \{ |x'-y| <10|x-x'| \}$, $|x^*-y| \ge \frac{1}{2}|x'-y|$. Then, we bound the above simply by \begin{equation*}
	\begin{split}
	&\le CA|x-x'|\int_{  \{ |x'-y| \ge 10|x-x'| \} } |x'-y|^{-n-1}(|x'-y|^\alpha + |x'|^\alpha ) dy \\
	&\le CA|x-x'|\left( |x-x'|^{\alpha-1} + |x'|^\alpha|x-x'|^{-1} \right) \le CA|x-x'|^\alpha 
	\end{split}
	\end{equation*} since $|x'|\le C|x-x'|$. The proof of the $C^\alpha_*$-estimate is complete.

	We omit the proof of $\mathring{C}^\alpha_*$ and $L^\infty$ bounds, which can be done in a similar way. 
\end{proof}

\begin{proof}[Proof of Proposition \ref{prop:symmetry}]
	To deduce Proposition \ref{prop:symmetry} from Lemma \ref{lem:symmetry-general}, one just needs to observe that $\Omega = \Omega_4^1$ and the singular integral kernel for $R_{ij}$ satisfy the assumptions of Lemma \ref{lem:symmetry-general} and $g(\mathbf{0})=0$, $g_{\ell} \in C^\alpha_*$ imply $|\tilde{g}_{\ell}^r(x)|\le C \nrm{g_{\ell}}_{C^\alpha_*}|x|^\alpha$. 
\end{proof}

\begin{definition}\label{def:extension}
	In the following, given $f \in (C^\alpha\cap\mathring{C}^\alpha(\tilde{U}))^3$, we denote $\tilde{f}^{m}$ to be the extension (governed by the reflection rule of $\tilde{O}$) of $f$ restricted to fundamental domains \textit{adjacent} to the support of $f$. As usual, $\tilde{f}$ is the full extension of $f$ onto $\bbR^3$. We accordingly define the adjacent extension of $\mathbf{1}_{\tilde{U}}$ to be $\tilde{\textbf{1}}^m$. Note that the definition of $\tilde{\textbf{1}}^m$ depends on the support of $f$. 
\end{definition}

\subsubsection{Estimates near  $\vec{\frka}_2$ }\label{subsubsec:est2}

In this section, we consider $f$ defined near $\vec{\frka}_2$ and away from $\vec{\frka}_3$, $\vec{\frka}_4$. With a rotation in $\bbR^3$, we consider the new orthogonal coordinate system with basis $\{ e_1=(\frac{1}{\sqrt{2}},-\frac{1}{\sqrt{2}},0 ),e_2=(0,0,1),e_3=(-\frac{1}{\sqrt{2}},-\frac{1}{\sqrt{2}},0 )  \}$. We shall write $x=(x_1,x_2,x_3)$ as well as $f = (f^1,f^2,f^3)$ with respect to this new system. Then, the assumption in Proposition \ref{prop:3D-estimate} translates to that $f^3$ is vanishing on $\{ x_2=x_1=0 \}$. For $1\le k\le 3$, we define $\tilde{f}^{k,m} \in (L^\infty(\bbR^3))^3$ to be the adjacent extension of $f^ke_k$ into $\bbR^3$ as in Definition \ref{def:extension}. To avoid confusion, we explicitly write out the extension rules: \begin{equation*}
\begin{split}
\tilde{f}^{3,m}(x_1,x_2,x_3) &:= f^3(x_1,x_2,x_3)\mathbf{1}_{ \{x_1,x_2>0\} } - f^3(-x_1,x_2,x_3)\mathbf{1}_{ \{x_1<0,x_2>0\} } \\ &\quad -f^3(x_1,-x_2,x_3)\mathbf{1}_{ \{x_1>0,x_2<0 \}}+f^3(-x_1,-x_2,x_3)\mathbf{1}_{ \{x_1,x_2<0\} },
\end{split}
\end{equation*} \begin{equation*}
\begin{split}
\tilde{f}^{2,m}(x_1,x_2,x_3) &:= f^2(x_1,x_2,x_3)\mathbf{1}_{ \{x_1,x_2>0\} } - f^2(-x_1,x_2,x_3)\mathbf{1}_{ \{x_1<0,x_2>0\} } \\ &\quad +f^2(x_1,-x_2,x_3)\mathbf{1}_{ \{x_1>0,x_2<0 \}}-f^2(-x_1,-x_2,x_3)\mathbf{1}_{ \{x_1,x_2<0\} },
\end{split}
\end{equation*}\begin{equation*}
\begin{split}
\tilde{f}^{1,m}(x_1,x_2,x_3) &:= f^1(x_1,x_2,x_3)\mathbf{1}_{ \{x_1,x_2>0\} } + f^1(-x_1,x_2,x_3)\mathbf{1}_{ \{x_1<0,x_2>0\} } \\ &\quad -f^1(x_1,-x_2,x_3)\mathbf{1}_{ \{x_1>0,x_2<0 \}}-f^1(-x_1,-x_2,x_3)\mathbf{1}_{ \{x_1,x_2<0\} }.
\end{split}
\end{equation*} That is, $\tilde{f}^{3,m}$, $\tilde{f}^{2,m}$, and $\tilde{f}^{1,m}$ are respectively odd-odd, odd-even, even-odd in $(x_1,x_2)$. The same is true for the full extensions $\tilde{f}^3, \tilde{f}^2,$ and $\tilde{f}^1$. To establish Proposition \ref{prop:3D-estimate} near $\vec{\frka}_2$, we need to prove 
\begin{lemma}\label{lem:est-2}
	Under the same assumptions as in Proposition \ref{prop:3D-estimate}, for any $1\le i,j,k\le3$, we have with $R_{ij} = \rd_{x_i}\rd_{x_j}(-\lap_{\bbR^3})^{-1}$, \begin{equation*}
	\begin{split}
	\nrm{R_{ij} \tilde{f}^k }_{C^\alpha\cap\mathring{C}^\alpha(\tilde{U})} \le C \nrm{f^k}_{C^\alpha\cap\mathring{C}^\alpha(\tilde{U})}. 
	\end{split}
	\end{equation*}
\end{lemma}

\begin{remark}
	As an immediate consequence of Lemma \ref{lem:symmetry-general}, it is sufficient to $R_{ij} \tilde{f}^k$ only near $\vec{\frka}_2$. 
\end{remark}

\begin{proof}[Proof in the case $k=3$]
	We consider $\tilde{f} := \tilde{f}^k$ which is vanishing on $\{ x_2=x_1=0 \}$ from the assumption. Note that $\tilde{f}$ is odd in both $x_1$ and $x_2$. With slight abuse of notation, we shall write $x_h := (x_1,x_2) = (x_1,x_2,0)$ and $y_h:=(y_1,y_2)=(y_1,y_2,0)$ given $x = (x_1,x_2,x_3)$ and $y=(y_1,y_2,y_3)$ ($h$ stands for ``horizontal''). We also write $|x-y|_h := |x_h-y_h|$. 
	
	\medskip
	
	\noindent (i) $L^\infty$ bound 
	
	\medskip
	
	\noindent We fix some $x \in \tilde{U}$ and consider 
	\begin{equation*}
	\begin{split}
	\int_{\bbR^3} K(x-y)\tilde{f}(y) dy = \left[ \int_{ \{|x-y|  <\frac{|x|}{2} \}} + \int_{ \{ |x-y|\ge \frac{|x|}{2}, |y|\le 10|x|    \} } + \int_{ \{|y|>10|x|\} } \right] K(x-y)\tilde{f}(y) dy = I + II + III
	\end{split}
	\end{equation*} where $K$ is the kernel for $R_{ij}$ with some $1\le i,j\le3$. We note that \begin{equation*}
	\begin{split}
	K(z) = \frac{p_{ij}(z)}{|z|^5}, 
	\end{split}
	\end{equation*} where $p_{ij}(\cdot)$ is a homogeneous polynomial of degree 2. Now recall from Section \ref{sec:expansion} that the integral of $\tilde{f}$ against any second order polynomial on spheres centered at the origin vanishes. Hence, \begin{equation*}
	\begin{split}
	|III| = \left|\int_{ \{|y|>10|x|\}}[K(x-y)-K(y)]\tilde{f}(y) dy \right| \le C\nrm{f}_{L^\infty}|x| \int_{ \{|y|>10|x|\}} |y|^{-4} dy\le C\nrm{f}_{L^\infty}. 
	\end{split}
	\end{equation*}
	In the second region, one use simply that $|K(z)|\le C|z|^{-3}$: \begin{equation*}
	\begin{split}
	II\le \nrm{\tilde{f}}_{L^\infty} \int_{ \{ |x-y|\ge \frac{|x|}{2}, |y|\le 10|x|    \}  }|x-y|^{-3}dy \le C \nrm{{f}}_{L^\infty}.
	\end{split}
	\end{equation*}
	It only remains to bound the local region; we rewrite \begin{equation*}
	\begin{split}
	I = \int_{ \{|x-y| <\frac{|x|}{2} \}}K(x-y)(\tilde{f}(y)-f(x) \tilde{\bf 1}^m(y) ) dy \, + f(x)  \int_{ \{|x-y| <\frac{|x|}{2} \}}K(x-y)  \tilde{\bf 1}^m(y) dy
	\end{split}
	\end{equation*} The point is that in the region $\{|x-y| <\frac{|x|}{2} \}$, \begin{equation*}
	\begin{split}
	|\tilde{f}(y)-f(x) \tilde{\bf 1}^m(y)|\le |x-y|^\alpha \nrm{f}_{C^\alpha_*}.
	\end{split}
	\end{equation*} Hence, the first term is bounded in absolute value by \begin{equation*}
	\begin{split}
	\left| \int_{ \{|x-y| <\frac{|x|}{2} \}}K(x-y)(\tilde{f}(y)-f(x) \tilde{\bf 1}^m(y) ) dy \right| \le  C\nrm{f}_{\mathring{C}^\alpha_*}. 
	\end{split}
	\end{equation*} To treat the second term, we note that up to a bounded term, we may consider (with change of variables $y_h=|x_h|z_h$) \begin{equation*}
	\begin{split}
	&\left|\int_{ \{ |x-y|_h <\frac{|x|_h+|x_3|}{2} \}}K_2(x_h-y_h)  \tilde{\bf 1}^m(y_h) dy_h\right| = \left| \int_{  \{ |\frac{x_h}{|x_h|} -z_h| < \frac{1}{2}(1 + \frac{|x_3|}{|x_h|}) \}   } K_2(\frac{x_h}{|x_h|}-z_h)   \tilde{\bf 1}^m(z_h) dz_h\right| \\
	&\qquad \le C\ln\left( 1+ \frac{|x_3|}{|x_h|} \right) 
	\end{split}
	\end{equation*} but then \begin{equation*}
	\begin{split}
	C|f(x)|\ln\left( 1+ \frac{|x_3|}{|x_h|} \right) = 	C|f(x)-f(0,0,x_3)|\ln\left( 1+ \frac{|x_3|}{|x_h|} \right) \le C|x_h|^\alpha|x|^{-\alpha} \ln\left( 1+ \frac{|x_3|}{|x_h|} \right) \nrm{f}_{\mathring{C}^\alpha_*} \le C\nrm{f}_{\mathring{C}^\alpha_*} 
	\end{split}
	\end{equation*} since $|x_h|\le C|x_3|$ in $\tilde{U}$. Collecting the bounds, \begin{equation*}
	\begin{split}
	|I|+|II|+|III| \le C( \nrm{f}_{\mathring{C}^\alpha_*} + \nrm{f}_{L^\infty} ).
	\end{split}
	\end{equation*}
	
	\medskip
	
	\noindent (ii) $C^\alpha_*$ bound 
	
	\medskip
	
	\noindent To show $C^\alpha_*$, it suffices to consider $\tilde{f}^m$ rather than $\tilde{f}$, appealing to the symmetry reduction lemma \ref{lem:symmetry-general}. We note that $\tilde{f}^m$ is scalar-valued, and supported near the half-line $\vec{\frka}_2$. Then, from this support property of $\tilde{f}^m$, we may identify $\tilde{\mathbf{1}}^m(y)$ with $\mathrm{sgn}(y_1y_2)$ since $\tilde{\mathbf{1}}^m(y)$ enters the proof only through the expression $\tilde{f}^m(y) -\tilde{\mathbf{1}}^m(y)f(x)$. From now on, for simplicity we shall even drop the superscript $m$. Moreover, it suffices to show $C^\alpha_*$-estimate in each coordinate. We first consider variations in $x_h$. Then, proving $C^\alpha_*$ in $x_h$ reduces to a 2D computation. 
	
	To see this, take $x \ne x' \in \tilde{U}$ with $x_3 = x'_3$ and rewrite \begin{equation*}
	\begin{split}
	&\int_{\bbR^3} K(x-y)\tilde{f}(y) dy - \int_{\bbR^3} K(x'-y)\tilde{f}(y) dy = \int_{\bbR^3} K(x-y)(\tilde{f}(y) - \tilde{\mathbf{1}}(y)f(x) )dy - \int_{\bbR^3} K(x'-y)(\tilde{f}(y)- \tilde{\mathbf{1}}(y)f(x')) dy \\
	& \qquad + f(x)\int_{\bbR^3} K(x-y) \tilde{\mathbf{1}}(y) dy - f(x')\int_{\bbR^3} K(x'-y) \tilde{\mathbf{1}}(y)dy. 
	\end{split}
	\end{equation*} Splitting the integration into $\{ |x-y|\le 10|x-x'| \}$ and its complement, we further rewrite \begin{equation*}
	\begin{split}
	&=  \int_{\{ |x-y|\le 10|x-x'| \}} K(x-y)(\tilde{f}(y) - \tilde{\mathbf{1}}(y)f(x) )dy - \int_{\{ |x-y|\le 10|x-x'| \}} K(x'-y)(\tilde{f}(y)- \tilde{\mathbf{1}}(y)f(x')) dy \\
	&\qquad +  \int_{\{ |x-y|> 10|x-x'| \}} \left( K(x-y)-K(x'-y) \right)( \tilde{f}(y)- \tilde{\mathbf{1}}(y)f(x) ) - \int_{\{ |x-y|\le 10|x-x'| \}} (f(x)-f(x')) K(x'-y)\tilde{\mathbf{1}}(y) dy \\
	&\qquad + f(x)\int_{\bbR^3} K(x-y) \tilde{\mathbf{1}}(y) dy - f(x')\int_{\bbR^3} K(x'-y) \tilde{\mathbf{1}}(y)dy. 
	\end{split}
	\end{equation*} The first three terms are straightforward to estimate by $C\nrm{f}_{C^\alpha_*(\tilde{U})}|x-x'|^\alpha$, simply using the pointwise estimates \begin{equation*}
	\begin{split}
	|\tilde{f}(y) - \tilde{\mathbf{1}}(y)f(x) | \le C\nrm{f}_{C^\alpha_*(\tilde{U})}|x-y|^\alpha, \quad |K(z)|\le C|z|^{-3}, \quad |\nabla K(z)| \le C|z|^{-4}. 
	\end{split}
	\end{equation*} Combining last three terms gives \begin{equation*}
	\begin{split}
	f(x)\int_{\bbR^3} \left[ K(x'-y) - K(x-y) \right] \tilde{\bf 1}(y) dy - (f(x')-f(x))\int_{\{|x-y|\le 10|x-x'|\}} K(x'-y) \tilde{\bf 1}(y) dy. 
	\end{split}
	\end{equation*}  Then, we can integrate the first expression in $y_3$, and rewrite $f(x) = f(x)-f(0,0,x_3)$: \begin{equation*}
	\begin{split}
	(f(x)-f(0,0,x_3))\int_{\bbR^2} \left[ K_2(x'_h-y_h) - K_2(x_h-y_h) \right] \mathrm{sgn}(y_1y_2) dy_h  
	\end{split}
	\end{equation*} The fact that this is bounded by $C\nrm{f}_{C^\alpha_*(\tilde{U})}|x_h-x'_h|^\alpha$ follows exactly the proof of Lemma \ref{lem:2D-2}, when $K$ is the kernel for either $R_{11}, R_{12}$, or $R_{22}$. Under the same assumptions for $K$, it is not difficult to show directly that (analogous to the ``half-moon'' computations) \begin{equation}\label{eq:integral}
	\begin{split}
	\int_{\{|x-y|\le 10|x-x'|\}} K(x'-y) \tilde{\mathbf{1}}(y) dy  
	\end{split}
	\end{equation} is uniformly bounded in $x,x'$. (Alternatively, one can replace $dy$ with $\chi(y_h)dy$, the integration domain $\{|x-y|\le 10|x-x'|\}$ to $\{|x-y|_h\le 10|x-x'|\}$ and reduce to a 2D computation as well.) It remains to treat the cases of $R_{13}, R_{23}$ (since $\mathrm{Id}=-R_{11}-R_{22}-R_{33}$) but in this case, the expressions \begin{equation*}
	\begin{split}
	\int_{\bbR^3} K(x-y) \tilde{\mathbf{1}}(y) dy,\quad \int_{\bbR^3} K(x'-y) \tilde{\mathbf{1}}(y) dy
	\end{split}
	\end{equation*} vanish in the first place since the kernels are odd in $y_3$ (after a shift by $x_3$) and $\tilde{\mathbf{1}}(y)$ is independent of $y_3$. It is easy to show that \eqref{eq:integral} is bounded in this case as well; we omit the details.
	
	We now consider variations in $x_3$; take two points $x,x'$ with $x_h=x'_h$ and $x_3 \ne x_3'$, and rewrite the difference as (using the $y_3$-invariance of $\tilde{\mathbf{1}}(y)$) \begin{equation*}
	\begin{split}
	\int_{\bbR^3} K(x-y)(\tilde{f}(y) - \tilde{\bf 1}(y)f(x))dy - \int_{\bbR^3} K(x-y) (\tilde{f}(y) - \tilde{\bf 1}(y)f(x)) dy. 
	\end{split}
	\end{equation*} Then, we proceed similarly as in the above: divide the integrals into regions $\{ |x-y| \le 10|x_3-x_3'| \}$ and its complement. Inspecting the terms, it is not difficult to see that it suffices to obtain the bound \begin{equation*}
	\begin{split}
	\left|(f(x)-f(x')) \int_{|x-y|\le 10 |x_3-x_3'| } K(x-y) \tilde{\bf 1}(y)dy \right| \le C \nrm{f}_{C^\alpha(\tilde{U})}|x-x'|^\alpha .
	\end{split}
	\end{equation*} The proof is similar to that of showing \eqref{eq:integral} is bounded. We omit the details. 
\end{proof}

\begin{proof}[Proof in the cases $k=1, 2$]
	We mainly emphasize the modifications from the proof above; note that now we do not have any vanishing condition. 
	
	\medskip
	
	\noindent (i) $L^\infty$ bound 
	
	\medskip
	
	\noindent We take $\tilde{f} := \tilde{f}^k$ ($k=1,2$), $\tilde{\mathbf{1}} := \tilde{\mathbf{1}}^k$ and follow the proof above; take $x\in\tilde{U}$ and decompose \begin{equation*}
	\begin{split}
	\int_{\bbR^3} K(x-y)\tilde{f}(y) dy = \left[ \int_{ \{|x-y|  <\frac{|x|}{2} \}} + \int_{ \{ |x-y|\ge \frac{|x|}{2}, |y|\le 10|x|    \} } + \int_{ \{|y|>10|x|\} } \right] K(x-y)\tilde{f}(y) dy = I + II + III.
	\end{split}
	\end{equation*} The expressions $|II|$ and $|III|$ can be bounded exactly the same way as before. On the other hand, in the local region, we note that \begin{equation*}
	\begin{split}
	\int_{ \{|x-y|  <\frac{|x|}{2} \} } K(x-y) \tilde{\mathbf{1}}(y) dy 
	\end{split}
	\end{equation*} is now uniformly bounded in $x \in \mathrm{supp}(f)$, where $K$ is the kernel for any $R_{ij}$ with $1\le i,j\le3$. This gives \begin{equation*}
	\begin{split}
	\left|f(x)\int_{ \{|x-y|  <\frac{|x|}{2} \} } K(x-y) \tilde{\mathbf{1}}(y) dy \right| \le C\nrm{f}_{L^\infty}.	\end{split}
	\end{equation*}
	Next, \begin{equation*}
	\begin{split}
	\left|\int_{ \{|x-y|  <\frac{|x|}{2} \} } K(x-y)(\tilde{f}(y)- f(x) \tilde{\mathbf{1}}(y) )dy  \right| \le C \nrm{f}_{\mathring{C}^\alpha(\tilde{U})},
	\end{split}
	\end{equation*} simply using that \begin{equation*}
	\begin{split}
	|\tilde{f}(y)- f(x) \tilde{\mathbf{1}}(y)| \le C|x|^{-\alpha} \nrm{f}_{\mathring{C}^\alpha(\tilde{U})}
	\end{split}
	\end{equation*} for $y \in \{|x-y|  <\frac{|x|}{2} \}$. The proof is complete. 
	
	\medskip
	
	\noindent (ii) $C^\alpha_*$ bound 
	
	\medskip
	
	\noindent As in the case of $k=3$ above, for the purpose of estimating $C^\alpha_*$ and $\mathring{C}^\alpha_*$, we appeal to Lemma \ref{lem:symmetry-general} and consider $\tilde{f}:=\tilde{f}^{k,m}$ and write $\tilde{\mathbf{1}}(y) := \tilde{\textbf{1}}^{k,m}(y)$ which can be identified with $\mathrm{sgn}(y_k)$. 
	
	Again, following the proof above, we start by rewriting \begin{equation*}
	\begin{split}
	&\int_{\bbR^3} K(x-y)\tilde{f}(y) dy - \int_{\bbR^3} K(x'-y)\tilde{f}(y) dy  \\
	& \quad  = \int_{\{ |x-y|\le 10|x-x'| \}} K(x-y)(\tilde{f}(y) - \tilde{\mathbf{1}}(y)f(x) )dy - \int_{\{ |x-y|\le 10|x-x'| \}} K(x'-y)(\tilde{f}(y)- \tilde{\mathbf{1}}(y)f(x')) dy \\
	&\qquad +  \int_{\{ |x-y|> 10|x-x'| \}} \left( K(x-y)-K(x'-y) \right)( \tilde{f}(y)- \tilde{\mathbf{1}}(y)f(x) ) dy \\
	&\qquad + f(x)\int_{\bbR^3} \left[ K(x'-y) - K(x-y) \right] \tilde{\bf 1}(y)dy - (f(x')-f(x))\int_{\{|x-y|\le 10|x-x'|\}} K(x'-y) \tilde{\bf 1}(y) dy. 
	\end{split}
	\end{equation*} As usual, it is straightforward to treat the first three terms, and the last term can be handled with a uniform bound on the integral \begin{equation*}
	\begin{split}
	\left|\int_{\{|x-y|\le 10|x-x'|\}} K(x'-y) \tilde{\bf 1}(y) dy\right| \le C. 
	\end{split}
	\end{equation*} Lastly, \begin{equation*}
	\begin{split}
	\int_{\bbR^3} \left[ K(x'-y) - K(x-y) \right] \tilde{\bf 1}(y)dy =0 
	\end{split}
	\end{equation*} simply because \begin{equation*}
	\begin{split}
	\int_{\bbR^3} K(x'-y) \tilde{\bf 1}(y)dy =\int_{\bbR^3}  K(x-y)   \tilde{\bf 1}(y)dy 
	\end{split}
	\end{equation*} recalling that $\tilde{\mathbf{1}}(y) = \mathrm{sgn}(y_k)$. Indeed, when $K$ is the kernel for $R_{13}$ and $R_{23}$, this is obvious from the $x_3$-invariance, and when $K$ is the kernel for $R_{11}, R_{12},$ and $R_{22}$, we can reduce to the corresponding equality from the 2D case (see Example \ref{example:Riesz} and observe that one can write $\tilde{\textbf{1}}$ as a linear combination of $\tilde{\textbf{1}}_{\Omega_4}$ defined there and its 2D rotations). We omit the proof of the $\mathring{C}^\alpha_*$-estimate, which is a straightforward adaptation of this argument. \end{proof}

\subsubsection{Estimates near $\vec{\frka}_3$ and $\vec{\frka}_4$}\label{subsubsec:est3} 

In this section, we consider the remaining cases of $f$ supported near $\vec{\frka}_3$ and $\vec{\frka}_4$ (and away from other half-lines). Most part of the arguments are parallel to the case of $\vec{\frka}_2$, and somewhat simpler. In the $\vec{\frka}_4$ case, we redefine the coordinate system by orthogonal basis \begin{equation*}
\begin{split}
\{ e_1 = (0,1,0), e_2 = (0,0,1), e_3 = (1,0,0) \}. 
\end{split}
\end{equation*} Note that as in the case of $\vec{\frka}_2$, the  radial direction is defined to be the new $x_3$-axis. Moreover, $\vec{\frka}_4$ is adjacent to 8 fundamental domains for $\tilde{\mathcal{O}}$ (including $\tilde{U}$ itself), which gives rise to the adjacent extension $\tilde{f}^m$. Again, for the convenience of the reader, we explicitly write them out in components: using the notation $f = (f^1,f^2,f^3)$ (in the new coordinates system), we first have \begin{equation*}
\begin{split}
\tilde{f}^{3,m}(x) &= g^3(x) + g^3(x_h^\perp,x_3) + g^3(-x_h,x_3) + g^3(-x_h^\perp,x_3),
\end{split}
\end{equation*} \begin{equation}\label{eq:g3}
\begin{split}
g^3(x) = f^3(x) \mathbf{1}_{ \{ x_1>x_2>0 \} } - f^3(x_2,x_1,x_3)\mathbf{1}_{ \{ x_2>x_1>0\} } .
\end{split}
\end{equation} Next, \begin{equation*}
\begin{split}
\tilde{f}^{1,m}(x) = (g^1(x) - g^1(-x_h,x_3))e_1 + (g^1(x_h^\perp,x_3) - g^1(-x_h^\perp,x_3))e_2,
\end{split}
\end{equation*}\begin{equation}\label{eq:g1}
\begin{split}
g^1(x) = f^1(x) \mathbf{1}_{ \{ x_1>x_2>0 \} } - f^1(x_1,-x_2,x_3) \mathbf{1}_{ \{ x_1>-x_2>0 \} } 
\end{split}
\end{equation} and \begin{equation*}
\begin{split}
\tilde{f}^{2,m}(x) = (g^2(x) - g^2(-x_h,x_3))e_2 + (g^2(x_h^\perp,x_3) - g^2(-x_h^\perp,x_3))e_1,
\end{split}
\end{equation*} \begin{equation}\label{eq:g2}
\begin{split}
g^2(x) = f^2(x) \mathbf{1}_{ \{ x_1>x_2>0 \} } + f^2(x_1,-x_2,x_3) \mathbf{1}_{ \{ x_1>-x_2>0 \} }. 
\end{split}
\end{equation} Important observation is that, freezing the $x_3$-coordinate, $\tilde{f}^{3,m}$ is a scalar-valued function which is 4-fold symmetric in $x_h$, and $\tilde{f}^{k,m}$ is odd in $x_h$ for $k = 1,2$. This allows one to essentially reduce the H\"older estimates to 2D computations, Lemma \ref{lem:2D-1} and Lemma \ref{lem:2D-3}, respectively. 

We mention briefly the case of $\vec{\frka}_3$. In this case the coordinate system is defined by \begin{equation*}
\begin{split}
\{  e_1 = (-\frac{1}{\sqrt{6}}, -\frac{1}{\sqrt{6}}, \frac{2}{\sqrt{6}})  , e_2 = (-\frac{1}{\sqrt{2}},\frac{1}{\sqrt{2}},0),    e_3 = (\frac{1}{\sqrt{3}},\frac{1}{\sqrt{3}},\frac{1}{\sqrt{3}}) \}.
\end{split}
\end{equation*} We suppress from writing out the formulas for $\tilde{f}^{k,m}$ near $\vec{\frka}_3$. However, the only essential feature that will be used in the proof is that, upon fixing $x_3$, $\tilde{f}^{3,m}$ is 3-fold rotationally symmetric in $x_h$, and $\tilde{f}^{k,m}$ is odd in $x_h$ for $k=1,2$.

We now state the main result of this section. 

\begin{lemma}\label{lem:est-43}
	Assume that $f\in C^\alpha\cap\mathring{C}^\alpha(\tilde{U})$ is supported near $\vec{\frka}_4$. Then, for any $1\le i,j,k \le 3$, we have that \begin{equation*}
	\begin{split}
	\nrm{R_{ij}\tilde{f}^k}_{ C^\alpha\cap\mathring{C}^\alpha(\tilde{U}) } \le \sum_{1\le\ell\le 3} C\nrm{f^{\ell}}_{C^\alpha\cap\mathring{C}^\alpha(\tilde{U})}. 
	\end{split}
	\end{equation*} The same estimate holds for $f$ supported near $\vec{\frka}_3$. 
\end{lemma}

We shall only consider the case of $\vec{\frka}_4$, the $\vec{\frka}_3$ case being strictly analogous. 

\begin{proof}
	\medskip
	
	\noindent (i) $L^\infty$ bound 

	\medskip
	
	\noindent 
	We first consider $\tilde{f}^3$ and follow the steps of the proof of Lemma \ref{lem:est-2} with some $x\in\tilde{U}$. To treat the region $\{ |x-y|<\frac{|x|}{2} \}$, we need to define $\tilde{\textbf{1}}$ appropriately. We simply take \begin{equation*}
	\begin{split}
	\tilde{\textbf{1}}^3(y) := \tilde{\mathbf{1}}_{R_4}(y_h)
	\end{split}
	\end{equation*} where $\tilde{\mathbf{1}}_{R_4}$ is defined in Figure \ref{fig:functions}. Again, the point is that we have \begin{equation*}
	\begin{split}
	|\tilde{f}^3(y) - \tilde{\textbf{1}}^3(y)f^3(x)| \le C\nrm{f^3}_{\mathring{C}^\alpha_*(\tilde{U})}|x|^{-\alpha}|x-y|^\alpha
	\end{split}
	\end{equation*} whenever $y \in \{ |x-y|<\frac{|x|}{2} \}$. This establishes the bound \begin{equation*}
	\begin{split}
	\nrm{R_{ij}\tilde{f}^3}_{L^\infty(\tilde{U})} \le C \nrm{f^3}_{\mathring{C}^\alpha(\tilde{U})}. 
	\end{split}
	\end{equation*} To treat the cases $k = 1,2$, we again note that it only remains to treat the integral \begin{equation*}
	\begin{split}
	\int_{\{ |x-y|<\frac{|x|}{2} \}} K(x-y)\tilde{f}^k(y)dy. 
	\end{split}
	\end{equation*} In this region, we have $\tilde{f}^k(y)=\tilde{f}^{k,m}(y)$ (replacing $\frac{|x|}{2}$ by $\frac{|x|}{10}$ if necessary) and we may treat separately $\tilde{g}^k:=g^k(x)-g^k(-x_h,x_3)$ and $\tilde{f}^{k,m} - \tilde{g}^k$, where $g^k$ is defined in \eqref{eq:g1},\eqref{eq:g2}. We just show how to treat $\tilde{g}^k$: in this case, we have that \begin{equation*}
	\begin{split}
	&\int_{\{ |x-y|<\frac{|x|}{2} \}} K(x-y)\tilde{g}^k(y)dy = \int_{\{ |x-y|<\frac{|x|}{2} \} \cap \{ 0<|y_2|<y_1   \}    } K(x-y)\tilde{g}^k(y)dy \\
	&\quad = \int_{\{ |x-y|<\frac{|x|}{2} \} \cap \{ 0<|y_2|<y_1   \}    } K(x-y)(\tilde{g}^k(y) - f^k(x)\tilde{\mathbf{1}}^k(y) )dy  + \int_{\{ |x-y|<\frac{|x|}{2} \} \cap \{ 0<|y_2|<y_1   \}    } K(x-y)\tilde{\mathbf{1}}^k(y)dy 
	\end{split}
	\end{equation*} Here, $\tilde{\textbf{1}}^1(y) := \tilde{\mathbf{1}}_{R_4}(y_h)$ and $\tilde{\textbf{1}}^2(y)= \mathbf{1}_{\mathbb{R}^3}(y)$. Similarly as in the above, these definitions guarantee that \begin{equation*}
	\begin{split}
		|\tilde{f}^k(y) - \tilde{\textbf{1}}^k(y)f^k(x)| \le C\nrm{f^k}_{\mathring{C}^\alpha_*(\tilde{U})}|x|^{-\alpha}|x-y|^\alpha
	\end{split}
	\end{equation*} as long as $y\in \{ |x-y|<\frac{|x|}{2} \} \cap \{ 0<|y_2|<y_1   \} $. This establishes the bound \begin{equation*}
	\begin{split}
	\sum_{k=1,2}\nrm{R_{ij}\tilde{f}^k}_{L^\infty(\tilde{U})} \le \sum_{k=1,2} C \nrm{f^k}_{\mathring{C}^\alpha(\tilde{U})}. 
	\end{split}
	\end{equation*} 
	
	\medskip
	
	\noindent (ii) $C^\alpha_*$ bound 
	
	\medskip
	
	\noindent To obtain the $C^\alpha_*$ bound, it suffices to consider $\tilde{f}^{k,m}$ rather than $\tilde{f}^k$, with an application of the symmetry reduction Lemma \ref{lem:symmetry-general}. We shall even omit the superscript $m$. The proof of $C^\alpha_*$ estimate is again parallel to the case of $\vec{\frka}_2$ treated in the previous section; variations in the $x_3$ direction is handled using $x_3$-invariance of $\tilde{\mathbf{1}}^k$ (defined in (i) above), and variations in $x_h$ can be reduced to obtaining 2D $C^\alpha_*$ estimates which correspond exactly to Lemma \ref{lem:2D-1} ($k=3$) and Lemma \ref{lem:2D-3} ($k=1,2$). 

The proof of $\mathring{C}^\alpha_*$-bound is parallel to that of $C^\alpha_*$ bound. We again omit the details. 
\end{proof}

\section{Local well-posedness}\label{sec:lwp}

We complete the proof of Theorem \ref{thm:lwp-corner} using the H\"older estimates we have established in the previous sections. 

\subsection{A priori estimates}\label{subsec:apriori}

Let $\omega \in \mathring{C}^\alpha\cap C^\alpha (\tilde{U})$ be a divergence-free vector satisfying $\omega^1+\omega^2 = 0$ along $\{ (z,z,0):z>0 \}$. We have that the same holds for the difference \begin{equation*}
\begin{split}
\omega - \omega(\mathbf{0}), 
\end{split}
\end{equation*} and since this function is vanishing at the origin, the estimates from the previous section gives that the corresponding velocity gradient belongs to $\mathring{C}^\alpha\cap C^\alpha (\tilde{U})$. Moreover, $\omega(\mathbf{0})$ is of the form \eqref{eq:explicit} for some choice of $\lambda$ and $\mu$, and the explicit computations in Section \ref{subsec:cutoff} show that the corresponding velocity gradient is again some constant function in $\tilde{U}$, depending on $\lambda$ and $\mu$. Therefore, we conclude that \begin{equation*}
\begin{split}
\nrm{\nabla u}_{ \mathring{C}^\alpha\cap C^\alpha (\tilde{U}) } \lesssim \nrm{\omg}_{\mathring{C}^\alpha\cap C^\alpha (\tilde{U})}
\end{split}
\end{equation*}  where $u$ is the velocity corresponding to $\omg$. Now assume that there is a solution $\omg$ in some time interval to \begin{equation*}
\begin{split}
\rd_t\omg + u\cdot\nb\omg = \nabla u \omg, 
\end{split}
\end{equation*} where $\omega^1+\omega^2 = 0$ along $\{ (z,z,0):z>0 \}$ for all $t$. Using the above bound for $\nabla u$, it is straightforward to derive the estimate  \begin{equation*}
\begin{split}
\frac{d}{dt}\nrm{\omg(t)}_{\mathring{C}^\alpha\cap C^\alpha (\tilde{U})} \lesssim \nrm{\omg(t)}_{\mathring{C}^\alpha\cap C^\alpha (\tilde{U})}^2.
\end{split}
\end{equation*} For instance, one can follow the proof of a priori estimates given in \cite{EJB,EJEuler}. In particular, there exists some $T_0>0$ depending only on $\nrm{\omg_0}_{\mathring{C}^\alpha\cap C^\alpha (\tilde{U})}$ such that \begin{equation*}
\begin{split}
\nrm{\omg(t)}_{\mathring{C}^\alpha\cap C^\alpha (\tilde{U})}\le 2\nrm{\omg_0}_{\mathring{C}^\alpha\cap C^\alpha (\tilde{U})},\quad t \in [0,T_0].
\end{split}
\end{equation*}

\subsection{Existence and uniqueness}\label{subsec:eandu} 

Under the assumptions of Theorem \ref{thm:lwp-corner}, we need to prove existence and uniqueness of a solution $\omega \in C([0,T); C^\alpha\cap\mathring{C}^\alpha(\tilde{U}))$ for some $T>0$. We divide the proof into uniqueness and existence. 

\begin{proof}[Proof of uniqueness]
	Let $\omega_0$ satisfy the  assumptions of Theorem \ref{thm:lwp-corner}, and let us assume that there exist some $T>0$ and two solutions  $\omega, \tilde{\omega} \in C([0,T); C^\alpha\cap\mathring{C}^\alpha(\tilde{U}))$ both satisfying $\omega(t=0) = \tilde{\omega}(t=0) = \omega_0$. Moreover, $\omega^1(t) + \omega^2(t)$ and $\tilde\omega^1(t) + \tilde\omega^2(t)$  vanish on $\{ (x_1,x_2,x_3): x_1 = x_2 \ge 0, x_3 = 0 \}$ for any $0\le t <T$. Denoting the corresponding velocities by $u,\tilde{u}$, we have from the a priori estimate that \begin{equation*}
	\begin{split}
	\nrm{\nabla u(t)}_{C^\alpha\cap\mathring{C}^\alpha(\tilde{U})} \le \nrm{\omega(t)}_{C^\alpha\cap\mathring{C}^\alpha(\tilde{U})},\quad \nrm{\nabla \tilde{u}(t)}_{C^\alpha\cap\mathring{C}^\alpha(\tilde{U})} \le \nrm{\tilde{\omega}(t)}_{C^\alpha\cap\mathring{C}^\alpha(\tilde{U})} . 
	\end{split}
	\end{equation*} In particular, we have for $0<t<T_0$ (where $T_0$ is given in \ref{subsec:apriori}) \begin{equation*}
	\begin{split}
	\nrm{\nb u(t)}_{L^\infty(\bbR^3)} +\nrm{\nb \tilde{u}(t)}_{L^\infty(\bbR^3)}\le 2\nrm{\omg_0}_{C^\alpha\cap \mathring{C}^\alpha(\tilde{U})}.
	\end{split}
	\end{equation*} From now on, we shall view the solutions as defined on $\bbR^3$. 
	Let us repeat the argument appeared in our previous work for the Boussinesq system \cite[Theorem 1]{EJB}. Denoting the difference by $v := u - \tilde{u}$ on $[0,T_0]$ and returning to the velocity formulation of 3D Euler, we write \begin{equation}\label{eq:v}
	\begin{split}
	\rd_t v + u\cdot\nb v + v\cdot\nb \tilde{u} + \nb \pi = 0,
	\end{split}
	\end{equation} where $\pi := p-\tilde{p}$. Here, $p$ and $\tilde{p}$ are the pressure corresponding to $u$ and $\tilde{u}$, respectively. We shall use the following key estimate from \cite{EJB}: \begin{equation}\label{eq:pressure_bound}
	\begin{split}
	\nrm{|x|^{-1}\nb\pi(x)}_{L^\infty} &\le C \left[ \left( \nrm{\nb u}_{L^\infty}+\nrm{\nb\tilde{u}}_{L^\infty} \right)  \nrm{ |x|^{-1}v(x)}_{L^\infty} \left( 1 + \ln\left( \frac{ \nrm{\nb\tilde{v}}_{L^\infty} }{ \nrm{ |x|^{-1}v(x)}_{L^\infty }} \right) \right)  \right]
	\end{split}
	\end{equation} Assuming \eqref{eq:pressure_bound}, we can finish the proof of uniqueness as follows. Dividing both sides of \eqref{eq:v} by $|x|$ and composing with the flow generated by $u$, we have \begin{equation*}
	\begin{split}
	\frac{d}{dt} \nrm{ |x|^{-1}v(x)}_{L^\infty} &\le C\left( \nrm{\nb u}_{L^\infty}  + \nrm{\nb \tilde{u}}_{L^\infty} \right)\nrm{ |x|^{-1}v(x)}_{L^\infty} + \nrm{ |x|^{-1}\nb p(x)}_{L^\infty} \\
	&\le C \nrm{ |x|^{-1}v(x)}_{L^\infty}(1 + \ln( \frac{C}{\nrm{ |x|^{-1}v(x)}_{L^\infty}} )),
	\end{split}
	\end{equation*} with $C>0$ now depending on $\sup_{t\in[0,T_0]}(\nrm{\nb u}_{L^\infty}  + \nrm{\nb \tilde{u}}_{L^\infty})$, which is bounded in terms of the initial data. The previous inequality is sufficient to guarantee that $\nrm{|x|^{-1}v(x)}_{L^\infty}=0$ on $[0,T_0]$ as $\nrm{|x|^{-1}v_0(x)}_{L^\infty}=0$. Repeating the same argument starting at $t = T_0$, one can show that $u=\tilde{u}$ all the way up to time $T>0$. 
	
	Let us now comment on the proof of \eqref{eq:pressure_bound}. We have \begin{equation*}
	\begin{split}
	-\lap\pi = \sum_{i,j} ( \rd_i u_j \rd_ju_i - \rd_i\tilde{u}_j\rd_j\tilde{u}_i ) = \sum_{i,j} ( \rd_i v_j \rd_ju_i - \rd_i\tilde{u}_j\rd_j v_i ) 
	\end{split}
	\end{equation*} Since $u$ and $\tilde{u}$ are divergence-free, we can also rewrite the above as \begin{equation*}
	\begin{split}
	\nb \pi  = \nb (-\lap)^{-1} \sum_i( \rd_i  \sum_j( v_j\rd_ju_i - \tilde{u}_j\rd_jv_i)),
	\end{split}
	\end{equation*} and we observe that the vector $W = (W_i)_{1\le i\le 3}$ defined by \begin{equation*}
	\begin{split}
	W_i = \sum_j( v_j\rd_ju_i - \tilde{u}_j\rd_jv_i)
	\end{split}
	\end{equation*} is symmetric (as a vector field) with respect to $\mathcal{O}$. Then we have \begin{equation*}
	\begin{split}
	\frac{\nb\pi}{|x|} = \frac{1}{|x|} \nb(-\lap)^{-1}\nb \cdot W 
	\end{split}
	\end{equation*} and since the singular integral operator $\nb(-\lap)^{-1}\nb \cdot$ has a kernel of the form $p(x)|x|^{-5}$ where $p$ is a (vector valued) homogeneous polynomial of order 2, there is a gain of decay when integrated against $W$. From this observation, it is straightforward to obtain the estimate \eqref{eq:pressure_bound}, following the argument of \cite{EJB}. 
\end{proof}

\begin{proof}[Proof of existence]
	Let $\omega_0$ satisfy the  assumptions of Theorem \ref{thm:lwp-corner}, with $(\omega_0^1+\omega_0^2)(z,z,0)=0$. The a priori estimate shows that the corresponding velocity gradient $\nabla u_0$ belongs to $C^\alpha\cap\mathring{C}^\alpha(\tilde{U})$. We define $\omega^{(0)}(t), u^{(0)}(t)$ to be $\omega_0$ and $u_0$ for $0\le t<T$ where $T>0$ is to be determined below. Given $(\omega^{(n)},u^{(n)})$ satisfying $(\omega^{(n),1}+\omega^{(n),2})(z,z,0)=0$, we define inductively $\omega^{(n+1)}$ as follows: \begin{equation*}
	\begin{split}
	&\rd_t \omega^{(n+1),1} + u^{(n)}\cdot\nb \omega^{(n+1),1}  = \rd_1u^{(n),1} \omega^{(n+1),1} + \rd_2u^{(n),1} \omega^{(n+1),2} + \rd_3 u^{(n+1),1}\omega^{(n),3} \\
	&\rd_t \omega^{(n+1),2} + u^{(n)}\cdot\nb \omega^{(n+1),2}  = \rd_1u^{(n),2} \omega^{(n+1),1} + \rd_2u^{(n),2} \omega^{(n+1),2} + \rd_3 u^{(n+1),2}\omega^{(n),3}\\
	&\rd_t \omega^{(n+1),3} + u^{(n)}\cdot\nb \omega^{(n+1),3}  = \rd_1u^{(n),3} \omega^{(n+1),1} + \rd_2u^{(n),3} \omega^{(n+1),2} + \rd_3 u^{(n+1),3}\omega^{(n),3},
	\end{split}
	\end{equation*}  with initial data $\omega_0^{(n+1)} = \omega_0$. We verify that, on $(z,z,0)$, we have: \begin{equation*}
	\begin{split}
	&[\rd_t + u^{(n)}\cdot\nb] (  \omega^{(n+1),1}  + \omega^{(n+1),2}  ) = ( \rd_1 u^{(n),1} + \rd_1 u^{(n),2} )(  \omega^{(n+1),1}  + \omega^{(n+1),2}  )\\
	&\qquad  + (-\omega^{(n),3} +  \rd_2 u^{(n),2} - \rd_1 u^{(n),1} )\omega^{(n+1),2} + \omega^{(n),3}(\omega^{(n+1),2} - \omega^{(n+1),1} + \rd_1u^{(n),3} + \rd_2u^{(n),3}) \\
	&\qquad \quad = ( \rd_1 u^{(n),1} + \rd_1 u^{(n),2} )(  \omega^{(n+1),1}  + \omega^{(n+1),2}  ) - \omega^{(n),3}(  \omega^{(n+1),1}  + \omega^{(n+1),2}  ), 
	\end{split}
	\end{equation*} which implies that along the line $\{(z,z,0)\}$, $(  \omega^{(n+1),1}  + \omega^{(n+1),2}  ) = 0$ if it holds at $t = 0$. At the last step we have used the identities \begin{equation*}
	\begin{split}
	\rd_1u^{(k),3} =\rd_2u^{(k),3}=0, 
	\end{split}
	\end{equation*} \begin{equation*}
	\begin{split}
	\rd_1u^{(k),1}+\rd_2u^{(k),1} = \rd_1u^{(k),2}+\rd_2u^{(k),2} 
	\end{split}
	\end{equation*} for $k=n,n+1$, which are consequences of the slip boundary conditions\begin{equation*}
	\begin{split}
	u^{(k),3}(x_1,x_2,0)=0, \quad u^{(k),1}(z,z,x_3)= u^{(k),2}(z,z,x_3). 
	\end{split}
	\end{equation*} 
	Using the a priori estimates, one can show that for the same $T_0>0$ from \ref{subsec:apriori}, we have \begin{equation*}
	\begin{split}
	\sup_{t \in [0,T_0]}\left(\nrm{\omega^{(n)}(t)}_{C^\alpha\cap\mathring{C}^\alpha} + \nrm{\nabla u^{(n)}(t)}_{C^\alpha\cap\mathring{C}^\alpha} \right)\lesssim \nrm{\omg_0}_{C^\alpha\cap\mathring{C}^\alpha}
	\end{split}
	\end{equation*} uniformly in $n$. Passing to a sub-sequential limit, one obtains a pair $(\omega(t),\nb u(t))$ bounded in $L^\infty([0,T_0];C^\alpha\cap\mathring{C}^\alpha)$. We have that $\omega^{(n)} \rightarrow \omega$ and $\nb u^{(n)} \rightarrow \nb u$ in $L^\infty_t C^0_x$. From this it is easy to see that the pair $(\omega,\nb u)$ is a solution to the Euler equations with initial data $\omega_0$. 
\end{proof}

\subsection{Propagation of higher regularity}\label{sec:higher}

Given the local well-posedness in $C^\alpha$ of the vorticity, it is not difficult to propagate higher H\"older regularity inside the domain $\tilde{U}$. Of course, it is \textit{necessary} to impose suitable vanishing conditions on the derivatives for the initial vorticity. In this section, we sketch the propagation of $C^{1,\alpha}$ regularity for the vorticity for any $0<\alpha<1$. Formally we state it as follows: \begin{proposition}\label{prop:higher} 
	In addition to the assumptions of Theorem \ref{thm:lwp-corner}, suppose that $\nb\omg_0 \in C^{\alpha}\cap\mathring{C}^\alpha(\tilde{U})$ and \begin{equation*}
	\begin{split}
	\nabla(\omega^1_0+\omega^2_0) = 0
	\end{split}
	\end{equation*} on $\vec{\frka}_2$ and either \begin{itemize}
		\item $\omega^1_0-\omega^2_0 = 0$ and $(\rd_1+\rd_2)(\omega^1_0-\omega^2_0 ) = 0$ or
		\item $\omega^3_0 = 0$ and $(\rd_1+\rd_2)\omega^3_0 = 0$
	\end{itemize} holds on $\vec{\frka}_2$. Then, the unique $C^\alpha$ solution defined on $[0,T^*)$ with initial data $\omg_0$ remains in $C^{1,\alpha}$ for all $t<T^*$.  
\end{proposition}
\begin{remark}\label{rem:initial}
	One can consider initial data  of the form \begin{equation}  \label{eq:initial}
	\left\{
	\begin{aligned} 
	\omg_0^1(x_1,x_2,x_3) & = -\lmb_0 \chi(x_1+x_2), \\
	\omg_0^2(x_1,x_2,x_3) & = \lmb_0 \chi(x_1+x_2), \\
	\omg_0^3(x_1,x_2,x_3) & = \mu_0 \chi(x_1+x_2)
	\end{aligned}
	\right.
	\end{equation} for some constants $\lmb_0,\mu_0\in\bbR$ and a cut-off function $\chi$ which satisfies $\chi(z) = 1$ for $|z|\le 1$ and $\chi(z)=0$ for $z\ge2$. Note that this $\omega_0$ is divergence-free and compactly supported in $\tilde{U}$. Further taking either $\lmb_0 = 0$ or $\mu_0  = 0$, this initial data satisfies all the assumptions of Proposition \ref{prop:higher}. Moreover, note from \eqref{eq:ODE_octant} that finite-time singularity formation is still possible, by taking either $\lmb_0 = 0,\mu_0<0$ or $\lmb_0<0,\mu_0=0$. 
\end{remark}

The condition $\nabla(\omega^1 +\omega^2 ) = 0$ propagates itself, which is the content of the following lemma: 
\begin{lemma}\label{lem:vanishing}
	Assume that $\nabla(\omega^1_0+\omega^2_0) = 0$ on $\vec{\frka}_2$. Then, $\nabla(\omega^1(t)+\omega^2(t)) = 0$ on  $\vec{\frka}_2$ as long as $\omega$ remains a $C^{1,\alpha}$-solution in $\tilde{U}$. 
\end{lemma}

In the remainder of this section,  it will be convenient to rotate the coordinates system: define \begin{equation*}
\begin{split}
y_1 := \frac{1}{\sqrt{2}} (x_1+x_2),\quad y_2 := \frac{1}{\sqrt{2}}(-x_1+x_2), \quad y_3:=x_3. 
\end{split}
\end{equation*} We shall write the components of $u$ and $\omg$ with respect to this coordinate system as well. Then, the boundary conditions on $\vec{\frka}_2$ take the form \begin{equation}\label{eq:BC-frka2}
\begin{split}
\rd_2^k\rd_1^ju^3 = 0, \quad \rd_1^k\rd_3^ju^2 = 0
\end{split}
\end{equation} for arbitrary integers $k, j \ge 0$. Using these, on $\vec{\frka}_2$, we note that \begin{equation}\label{eq:BCC-frka2}
\begin{split}
\rd_1^k\rd_2u^1 = \rd_1^{k+1}u^2-\rd_1^k\omg^3 = -\rd_1^k\omg^3,\quad \rd_1^j\rd_3u^1 = \rd_1^j\omg^2+\rd_1^{j+1}u^3 = \rd_1^j\omg^2
\end{split}
\end{equation} for any integers $k,j\ge0$. We shall assume that in some time interval, the solution $(\omg,u)$ belongs to $C^{1,\alpha}\times C^{2,\alpha}$, which can be justified after showing propagation of the vanishing conditions in Proposition \ref{prop:higher}. 

\begin{proof}[Proof of Lemma \ref{lem:vanishing}]
We compute \begin{equation*}
\begin{split}
\rd_t \rd_1\omg^1 + u\cdot\nb \rd_1\omg^1 & =  \rd_{11}u^1\omg^1+\rd_1u^1\rd_1\omg^1 + \rd_{12}u^1\omg^2+\rd_2u^1\rd_1\omg^2 +\rd_{13}u^1\omg^3+\rd_3u^1\rd_1\omg^3 \\
 &\qquad - \rd_1u^1\rd_1\omg^1 - \rd_1u^2\rd_2\omg^1 - \rd_1u^3\rd_3\omg^1 
\end{split}
\end{equation*} Restricting on $\vec{\frka}_2$, we have $\omg^1 = 0$ and applying \eqref{eq:BC-frka2}, \eqref{eq:BCC-frka2} gives that \begin{equation*}
\begin{split}
\rd_t \rd_1\omg^1 + u\cdot\nb \rd_1\omg^1 & = -\rd_1\omg^3\omg^2 - \omg^3\rd_1\omg^2 + \rd_1\omg^2\omg^3 + \omg^2\rd_1\omg^3 = 0. 
\end{split}
\end{equation*} This shows that the condition $\rd_1\omg^1 = 0$ propagates in time. Next, we compute \begin{equation*}
\begin{split}
\rd_t \rd_3\omg^1 + u\cdot\nb \rd_3\omg^1 & =  \rd_{31}u^1\omg^1+\rd_1u^1\rd_3\omg^1 + \rd_{32}u^1\omg^2+\rd_2u^1\rd_3\omg^2 +\rd_{33}u^1\omg^3+\rd_3u^1\rd_3\omg^3 \\
&\qquad - \rd_3u^1\rd_1\omg^1 - \rd_3u^2\rd_2\omg^1 - \rd_3u^3\rd_3\omg^1 .
\end{split}
\end{equation*} We simplify the right hand side, restricting the equation on $\vec{\frka}_2$. First, recalling $\omg^1 = \rd_1\omg^1 = 0$, we cancel a few terms, and rewrite the right hand side as follows: \begin{equation*}
\begin{split}
(\rd_1u^1-\rd_3u^3)\rd_3\omg^1  + \rd_3\omg^2\rd_2u^1 + \rd_3\omg^3\rd_3u^1 + \omg^2\rd_{23}u^1 + \omg^3\rd_{33}u^1 - \rd_3u^2\rd_2\omg^1.
\end{split}
\end{equation*} The last term vanishes from \eqref{eq:BC-frka2}. Rewriting the remaining four terms using \eqref{eq:BCC-frka2}, \begin{equation*}
\begin{split}
\rd_3\omg^2\rd_2u^1 + \rd_3\omg^3\rd_3u^1 + \omg^2\rd_{23}u^1 + \omg^3\rd_{33}u^1 = \omg^2(\rd_3\omg^3+\rd_2\omg^2) = -\omg^2\rd_1\omg^1 = 0. 
\end{split}
\end{equation*} Hence, \begin{equation*}
\begin{split}
\rd_t \rd_3\omg^1 + u\cdot\nb \rd_3\omg^1 & = (\rd_1u^1-\rd_3u^3)\rd_3\omg^1 ,
\end{split}
\end{equation*} which shows that $\rd_3\omg^1 = 0$ propagates as well. Propagation of $\rd_2\omg^1 = 0$ can be proved in the same way, noting that the conditions \eqref{eq:BC-frka2} and \eqref{eq:BCC-frka2} are symmetric in $y_2$ and $y_3$. \end{proof}

\begin{proof}[Proof of Proposition \ref{prop:higher}]
	All that there is to show is the $C^\alpha$-bound for the Riesz transforms $R_{ij}\nb\omega$. It suffices to consider components of $\nb\omega$ which extends as an odd function of both $y_3$ and $y_2$, around $\vec{\frka}_2$. Note that $\omega_1$, $\omega_2$, and $\omega_3$ are respectively odd-odd, even-odd, and odd-even in $y_2$ and $y_3$. Hence, for the $C^\alpha$-estimate of $R_{ij}\nb\omega$, it suffices to propagate the vanishing conditions for $\rd_1\omega^1$, $\rd_2\omega^2$, and $\rd_3\omega^3$, along $\vec{\frka}_2$. From the divergence-free condition, it suffices to consider the first two. 
	
	Lemma \ref{lem:vanishing} takes care of $\rd_1\omega^1$, so let us
now turn to $\rd_2\omg^2$: \begin{equation*}
	 \begin{split}
	 \rd_t \rd_2\omg^2 + u\cdot\nb \rd_2\omg^2 & =  \rd_2\omg^1\rd_1u^2 + \omg^1\rd_{12}u^2 +\rd_2\omg^2\rd_2u^2+\omg^2\rd_{22}u^2 + \rd_2\omg^3\rd_3u^2 + \omg^3\rd_{23}u^2 \\
	 &\qquad - \rd_2u^1\rd_1\omg^2 - \rd_2u^2\rd_2\omg^2 - \rd_2u^3\rd_3\omg^2 \\
	 & =   \rd_2\omg^1\rd_1u^2 + \omg^1\rd_{12}u^2 + \omg^2\rd_{22}u^2 + \rd_2\omg^3\rd_3u^2 + \omg^3\rd_{23}u^2 - \rd_2u^1\rd_1\omg^2  - \rd_2u^3\rd_3\omg^2 \\
	 & =    \omg^2\rd_{22}u^2 + \omg^3\rd_{23}u^2  - \rd_2u^1\rd_1\omg^2  
	 \end{split}
	 \end{equation*} where we have used \eqref{eq:BC-frka2} and $\omg^1=0$ to get the last equality. We now note that \begin{equation*}
	 \begin{split}
	   \omg^3\rd_{23}u^2  - \rd_2u^1\rd_1\omg^2  = \omg^3\rd_2(\rd_2u^3 - \omg^1) + \omg^3\rd_1\omg^2 = \omg^3\rd_1\omg^2 
	 \end{split}
	 \end{equation*} and \begin{equation*}
	 \begin{split}
	 \omg^2\rd_{22}u^2 = -\omg^2\rd_{21}u^1-\omg^2\rd_{23}u^3 = \omg^2\rd_1\omg^3 - \omg^2\rd_3\omg^1 = \omg^2\rd_1\omg^3 . 
	 \end{split}
	 \end{equation*} Hence,  for the propagation of $\rd_2\omg^2=0$, it suffices to have \begin{equation*}
	 \begin{split}
	 \omega^2\rd_1\omega^3 + \omg^3\rd_1\omg^2 = 0. 
	 \end{split}
	 \end{equation*} This is precisely the reason why we need the additional assumption that either \begin{equation*}
	 \begin{split}
	 \omg^2 =0, \quad\rd_1\omg^2 = 0 
	 \end{split}
	 \end{equation*} or \begin{equation*}
	 \begin{split}
	 \omega^3 = 0,\quad\rd_1\omega^3 = 0 
	 \end{split}
	 \end{equation*} along $\vec{\frka}_2$. 
	 
	 It still remains to show that the additional assumptions themselves propagate in time. We only consider the case of $\omg^2 =0, \rd_1\omg^2 = 0$. From \begin{equation*}
	 \begin{split}
	 \rd_t\omg^2 + u\cdot\nb\omg^2 = \omg\cdot\nb u^2
	 \end{split}
	 \end{equation*} and \eqref{eq:BC-frka2}, it is immediate to see the propagation of $\omg^2 =0$. Next, \begin{equation*}
	 \begin{split}
	 \rd_t\rd_1\omg^2 + u\cdot\nb\rd_1\omg^2 = \rd_1\omg\cdot\nb u^2 + \omg\cdot\nb\rd_1 u^2 - \rd_1u\cdot\nb\omg^2. 
	 \end{split}
	 \end{equation*} From \eqref{eq:BC-frka2}, $\rd_1\omg\cdot\nb u^2 = \rd_1\omg^2 \rd_2u^2$ on $\vec{\frka}_2$, and $\omg\cdot\nb\rd_1 u^2  = \omg^2 \rd_{12}u^2 = 0$. Lastly, $\rd_1u\cdot\nb\omg^2 = \rd_1u^1\rd_1\omg^2$. Hence \begin{equation*}
	 \begin{split}
	 \rd_t\rd_1\omg^2 + u\cdot\nb\rd_1\omg^2 = (\rd_1u^1+\rd_2u^2)\rd_1\omg^2
	 \end{split}
	 \end{equation*} and the propagation of $\rd_1\omg^2 =0$ is proved. 
	 
	 From the propagation of vanishing conditions for $\rd_i\omg^i$, one can establish that $\nrm{R_{ij}\nb\omg}_{C^\alpha(\tilde{U})} \lesssim \nrm{\nb\omg}_{C^\alpha(\tilde{U})}$. With this it is not difficult to show that the unique $C^{\alpha}$ solution $\omg$ in $[0,T^*)$ actually belongs to $C^{1,\alpha}$. 	 
\end{proof}

\section{Finite time singularity formation}\label{sec:blowup}

In this section, we conclude the proof of finite-time singularity formation. The proof consists of finding explicit homogeneous solutions which blows up in finite time and cutting those at spatial infinity. 

\subsection{Explicit blow-up solutions}\label{subsec:explicit}

We consider constant vorticity defined in $\tilde{U}$ of the form \begin{equation}\label{eq:vort_octant}
\begin{split}
\omega_0 = \lambda_0 \begin{pmatrix} 
-1 \\
1 \\
0
\end{pmatrix} + \mu_0 \begin{pmatrix}
0 \\
0 \\
1
\end{pmatrix},
\end{split}
\end{equation} for some $\lambda_0, \mu_0 \in \mathbb{R}$. The corresponding velocity with non-penetration boundary conditions on $\partial\Omega$ is explicitly given by \begin{equation}\label{eq:vel_octant}
\begin{split}
u_0 (x_1,x_2,x_3) = \frac{1}{3}\begin{pmatrix}
\left( -\lambda_0  + 2\mu_0\right) x_1 - 3\mu_0 x_2 + 3\lambda_0 x_3  \\
\left( - \lambda_0 - \mu_0 \right)x_2 + 3\lambda_0 x_3 \\
\left( 2\lambda_0 - \mu_0 \right)x_3 
\end{pmatrix}. 
\end{split}
\end{equation} Taking the gradient, we obtain \begin{equation}\label{eq:velgrad_octant}
\begin{split}
\nabla u_0 = \frac{1}{3}\begin{pmatrix}
-\lambda_0 + 2\mu_0 & -3\mu_0 & 3\lambda_0 \\
0 & -\lambda_0 - \mu_0 & 3\lambda_0 \\
0 & 0 & 2\lambda_0 - \mu_0 
\end{pmatrix}.
\end{split}
\end{equation} Next, one may compute the following product: \begin{equation*}
\begin{split}
\omega_0\cdot\nabla u_0 = [\nabla u_0]\omega_0 = \frac{1}{3}\begin{pmatrix}
\lambda_0^2 - 2\lambda_0\mu_0 \\
-\lambda_0^2 + 2\lambda_0\mu_0 \\
2\lambda_0\mu_0 - \mu_0^2
\end{pmatrix}.
\end{split}
\end{equation*} From the above, one sees that the unique solution to the $3D$ Euler equations with initial vorticity of the form \eqref{eq:vort_octant} has the same from for all times -- writing the solution as \begin{equation*}
\begin{split}
\omega(t) = \lambda(t) \begin{pmatrix} 
-1 \\
1 \\
0
\end{pmatrix} + \mu(t) \begin{pmatrix}
0 \\
0 \\
1
\end{pmatrix},
\end{split}
\end{equation*} one obtains the following ODE system of two variables: \begin{equation}\label{eq:ODE_octant}
\left\{
\begin{aligned} 
\dot{\lambda} &= \frac{2}{3}\lambda\mu - \frac{1}{3}\lambda^2, \\
\dot{\mu} &= \frac{2}{3}\lambda\mu - \frac{1}{3} \mu^2 .
\end{aligned}
\right.
\end{equation} We prove blow-up for the above ODE system assuming either one of the following conditions: \begin{enumerate}
	\item $\lmb \ne \mu$ and $\lmb\mu\ne0$, or 
	\item $\lmb=\mu>0.$ 
\end{enumerate} These are exactly the sufficient conditions stated in Theorem \ref{thm:blowup-corner}. First, assuming $\lmb_0=\mu_0$, it is obvious that the ODE system blows up in finite time if and only if the initial data is strictly positive. Now, since the system is symmetric in $\lmb$ and $\mu$, we may assume that $\lmb_0>\mu_0$ if $\lmb_0\ne\mu_0$. We then write down the equation for $\lmb-\mu$: \begin{equation*}
\begin{split}
\frac{d}{dt}(\lmb-\mu) = -\frac{1}{3}(\lmb-\mu)(\lmb+\mu).
\end{split}
\end{equation*} In particular, $\lmb-\mu$ remains strictly positive. Now, \begin{equation*}
\begin{split}
\frac{d}{dt} \mu = \frac{\mu}{3}( 2(\lmb-\mu)+\mu) .
\end{split}
\end{equation*} If in addition we had $\mu_0>0$, then $\mu>0$ and the above ODE clearly blows up as $2(\lmb-\mu)+\mu>\mu$. Now we may assume $\mu_0<0$ and there are two cases: (i) $\lmb_0>0$ and (ii) $\lmb_0<0$. In the former, we have that $\lmb$ is monotonically decreasing in time since \begin{equation*}
\begin{split}
\frac{d}{dt}\lmb = \frac{\lmb}{3}(2\mu-\lmb) 
\end{split}
\end{equation*} and $2\mu-\lmb<0$. Then, returning to the equation for $\mu$, it is clear that $\mu\to-\infty$ in finite time. Finally in the latter case, we have $0>\lmb_0>\mu_0$ and in this case it is not hard to see that $\lmb-\mu \to+\infty$ in finite time since both $\lmb$ and $\mu$ remain negative.


\subsection{A cut-off argument}\label{subsec:cutoff}

We shall prove that for initial vorticity which is of the form \eqref{eq:vort_octant} only near the origin, finite time blow-up occurs, as long as the solution to the ODE system \eqref{eq:ODE_octant} blows up. To this end, take an explicit solution \begin{equation}\label{eq:explicit}
\begin{split}
\omega^{SI}(t) = \lambda(t) \begin{pmatrix} 
-1 \\
1 \\
0
\end{pmatrix} + \mu(t) \begin{pmatrix}
0 \\
0 \\
1
\end{pmatrix}
\end{split}
\end{equation} for which we have finite time blow-up at time $T^*>0$: $\lim_{t \rightarrow T^*}(|\mu(t)| + |\lmb(t)|) = +\infty$. We denote the corresponding velocity as \begin{equation}
u^{SI}(t) =  \frac{1}{3}\begin{pmatrix}
-\lambda(t) + 2\mu(t) & -3\mu(t) & 3\lambda(t) \\
0 & -\lambda(t) - \mu(t) & 3\lambda(t) \\
0 & 0 & 2\lambda(t) - \mu(t) 
\end{pmatrix} \begin{pmatrix}
x_1 \\ x_2 \\x_3 
\end{pmatrix}
\end{equation} on $0\le t<T^*$. We now take initial data of the form \begin{equation*}
\begin{split}
\omega_0 = \omega^{SI}_0 + \tilde{\omega}_0
\end{split}
\end{equation*} on $\tilde{U}$, where $\omega_0$ satisfies the assumptions of Theorem \ref{thm:lwp-corner}. This in particular implies that $\tilde{\omega}_0 \in C^\alpha \cap \mathring{C}^\alpha$ and that $\tilde{\omega}_0^1+\tilde{\omega}_0^2$  vanishes on $\{ (x_1,x_2,x_3): x_1=x_2\ge 0, x_3=0\}$. We may further assume that $\tilde{\omega}_0(\textbf{0}) =0$. One may even take $\tilde{\omega}_0 \in C^\infty$ identically zero on $|x| \le 1$ and equal to $-\omega^{SI}_0$ on $|x|>1$, so that $\omega_0$ belongs to $C^\infty_c$ with support contained in $\tilde{U} \cap \{ |x| \le 1 \}$. Such $\tilde{\omega}_0$ certainly exists, as one can simply use $\omega_0$ as in \eqref{eq:initial}.

Indeed, there is a very robust way to obtain compactly supported and divergence-free vector field with desired profile in the region $\tilde{U}\cap\{ |x|\le 1 \}$, using the well-known Bogovski\v{i} operator introduced in \cite{Bog}. Detailed information regarding the properties of this operator can be found in several textbooks; see for instance \cite{AcDu,Gal}. To be precise, we take \begin{equation*}
\begin{split}
f(x):= h(|x|) \omega^{SI}_0(x)
\end{split}
\end{equation*} where $h$ is a $C^\infty$-function satisfying $h(z)=1$ for $|z|\le 1$ and $h(z)=0$ for $|z|\ge2$. Now we recall the theorem of Bogovski\v{i}:
\begin{theorem}[{{\cite[Lemma III.3.1 and Remark III.3.2]{Gal}}}]
	Let $\Omega\subset\bbR^n$ be a bounded open set with Lipschitz boundary. There exists a linear operator defined on the set of mean zero functions  $\mathrm{div}^{-1}_\Omega : L^2(\Omega) \rightarrow (H^1_0(\Omega))^3 $  satisfying the following: \begin{equation*}
	\begin{split}
	\nabla\cdot (\mathrm{div}^{-1}_\Omega g) = g, 
	\end{split}
	\end{equation*} \begin{equation*}
	\begin{split}
	\nrm{ \mathrm{div}^{-1}_\Omega g}_{H^m_0(\Omega)} \lesssim_{\Omega,n,m} \nrm{g}_{H^{m-1}(\Omega)},\quad m\ge 1.  
	\end{split}
	\end{equation*}
\end{theorem}
We take the operator $\mathrm{div}^{-1}_\Omega$ with $\Omega = \tilde{U}\cap\{ |x| <3 \}$ and $g = \nabla\cdot f$. From the explicit kernel representation of the Bogovski\v{i} operator, \cite[(III.3.8)]{Gal}, it can be easily seen that $ \mathrm{div}^{-1}_\Omega g$ is supported away from $\{|x|=3\}$ since $g$ is. In particular, we have that $g \in C^\infty_c(\tilde{U})$. Then, we may set \begin{equation*}
\begin{split}
\omega_0 := f-  \mathrm{div}^{-1}_{\tilde{U}\cap\{ |x| <3 \}} g
\end{split}
\end{equation*} The function $\mathrm{div}^{-1}_{\tilde{U}\cap\{ |x| <3 \}} g$ vanishes on $\partial\tilde{U}$, and hence $\omega_0$ satisfies the assumptions of Theorem \ref{thm:lwp-corner}.

Now assume towards a contradiction that the unique local solution $(\omega(t),u(t))$ (given by Theorem \ref{thm:lwp-corner}) with initial data $\omega_0$ is global in time. Then, we may set \begin{equation*}
\begin{split}
M := \sup_{0\le t\le T^*} (\nrm{\omg(t)}_{C^\alpha} + \nrm{\nb u(t)}_{C^\alpha}).
\end{split}
\end{equation*} We now \textit{define} $\tilde{\omega}(t):=\omega(t) - \omega^{SI}(t)$ and $\tilde{u}(t):=u(t)-u^{SI}(t)$. These functions are well-defined for any $0<t<T^*$, and we see from \begin{equation*}
\begin{split}
\rd_t \omega + u\cdot\nabla \omega = [\nb u] \omega 
\end{split}
\end{equation*} and \begin{equation*}
\begin{split}
\rd_t \omega^{SI} + u^{SI}\cdot\nabla \omega^{SI} = [\nb u^{SI} ]\omega^{SI}
\end{split}
\end{equation*} that $\tilde{\omega}$ must satisfy the following equation: \begin{equation*}
\begin{split}
\partial_t \tilde{\omega} + u\cdot\nabla  \tilde{\omega}=[\nabla u]\tilde{\omega} + [\nabla \tilde{u}]\omega^{SI} - \tilde{u}\cdot\nabla \omega^{SI}=[\nabla u]\tilde{\omega} + [\nabla \tilde{u}]\omega^{SI}.  
\end{split}
\end{equation*} The latter identity follows since $\nabla\omega^{SI} \equiv 0$. Let us propagate that $\tilde{\omega} \in C^\alpha \cap \mathring{C}^\alpha$ and $\tilde{\omega}(t,\textbf{0}) = 0$, on $0\le t<T^*$. We rewrite the equation along the flow $\Phi$ generated by $u$: $\rd_t\Phi(t,x)=u(t,\Phi(t,x))$, $\Phi(0,x)=x$. This gives \begin{equation*}
\begin{split}
\frac{d}{dt}( \tilde{\omega}(t) \circ \Phi(t)) =  [\nabla u\circ \Phi ]\tilde{\omega} \circ\Phi + [\nabla \tilde{u}\circ\Phi] \omega^{SI}\circ\Phi.  
\end{split}
\end{equation*} {We claim that $\nb \tilde{u}(t,\textbf{0})=0$}, which follows from writing down the first-order Taylor expansion of $\tilde{u}(t)$ in $x$ and imposing the boundary condition as well as the condition that the curl vanishes at the origin. To see this, we write \begin{equation*}
\begin{split}
\begin{pmatrix}
\tilde{u}_1 \\ \tilde{u}_2 \\ \tilde{u}_3
\end{pmatrix} = \begin{pmatrix}
ax_1+bx_2+cx_3 \\
dx_1+ex_2+fx_3 \\
gx_1+hx_2+(-a-e)x_3
\end{pmatrix} + o(|x|)
\end{split}
\end{equation*} where $\rd_{x_3}\tilde{u}_3(0)=-a-e$ follows from the divergence-free condition. From the slip boundary conditions at $\{x_3=0\}$, $\{x_1=x_2\}$, and $\{x_2=x_3\}$, we respectively obtain that $g=h=0$ and $a+b=d+e$, $c=f$ and $d=0, c+f=-a-e$. Then the matrix simplifies into \begin{equation*}
\begin{split}
\begin{pmatrix}
ax_1 + bx_2 - (3a+2b)x_3 \\
(a+b)x_2 - (3a+2b)x_3\\
-(2a+b)x_3 
\end{pmatrix}.
\end{split}
\end{equation*} Finally, the corresponding vorticity vector is given by \begin{equation*}
\begin{split}
\begin{pmatrix}
3a+2b\\
-3a-2b\\
-b
\end{pmatrix}
\end{split}
\end{equation*} so that the vanishing condition for the curl imposes $a=b=0$, finishing the proof that $|\tilde{u}(x)|=o(|x|)$. Since we know that $\tilde{u}\in C^{1,\alpha}$, this implies that $|\nabla\tilde{u}(x)| = O(|x|^\alpha)$. 

Now observing that $\Phi(t, \textbf{0} ) = { \bf 0}$ and evaluating the equation at $x = \textbf{0}$,  we have that $\frac{d}{dt} \tilde{\omega}(t,\textbf{0}) = 0$, which gives $\tilde{\omega}(t,\textbf{0}) = 0$ on $0\le t<T^*$. Next, \begin{equation*}
\begin{split}
\frac{d}{dt} \nrm{ \tilde{\omega} \circ \Phi }_{C^\alpha} & \le \nrm{\nb u \circ \Phi}_{C^\alpha} \nrm{\tilde{\omega} \circ \Phi}_{L^\infty} + \nrm{\nb u \circ \Phi}_{L^\infty} \nrm{\tilde{\omega} \circ \Phi}_{C^\alpha} \\
& \qquad + \nrm{\nb \tilde{u} \circ \Phi}_{C^\alpha} \nrm{{\omega^{SI}} \circ \Phi}_{L^\infty} + \nrm{\nb \tilde{u} \circ \Phi}_{L^\infty} \nrm{{\omega^{SI}} \circ \Phi}_{C^\alpha}. 
\end{split}
\end{equation*} Using regularity of the flow $\sup_{0\le t\le T^*} \nrm{\nabla\Phi}_{C^\alpha} \le CM$ and $\nrm{\nabla\tilde{u}}_{C^\alpha} \le C \nrm{\tilde{\omega}}_{C^\alpha}$, we deduce that \begin{equation*}
\begin{split}
\frac{d}{dt} \nrm{\tilde{\omega}}_{C^\alpha} \le (M + \nrm{\omega^{SI}}_{L^\infty([0,T];C^\alpha)})\nrm{\tilde{\omega}}_{C^\alpha} 
\end{split}
\end{equation*} for any $T<T^*$. Hence, it follows that \begin{equation*}
\begin{split}
 \nrm{\tilde{\omega}(T)}_{C^\alpha}  \le C = C(M,T,\nrm{\tilde{\omega}_0}_{C^\alpha}). 
\end{split}
\end{equation*} Using the above with $\tilde{\omega}(T,\textbf{0}) = 0$, $|\tilde{\omega}(T,x) | \le C|x|^\alpha$ and in particular by taking $|x|$ sufficiently small, we obtain that  $$\Vert \omega(T) \Vert_{L^\infty} \ge  \Vert \omega^{SI}(T)\Vert_{L^\infty}$$ for any $T<T^*$. Taking $T \rightarrow T^*$, we have from the choice of $\omega^{SI}$ that $\nrm{\omega^{SI}(T)}_{L^\infty} \rightarrow +\infty$, which is a contradiction. The proof of Theorem \ref{thm:blowup-corner} is now complete. \qedsymbol

\appendix

\section*{Appendix A. Singular integral formulas for the velocity gradient}
In this section, we demonstrate explicitly that the velocity gradient in 3D can be expressed in terms of a linear combination of the vorticity and its double Riesz transforms. 
We have \begin{equation*}
\begin{split}
\nabla u(x) = \nabla K * \omega (x) + \frac{1}{3}\begin{pmatrix}
0 & - \omega_3(x) & \omega_2(x) \\
\omega_3(x) & 0 & -\omega_1(x) \\
-\omega_2(x) & \omega_1(x) & 0 
\end{pmatrix}. 
\end{split}
\end{equation*} Explicitly writing it out, 
\begin{equation}\label{eq:pr1_u}
\begin{split}
&\rd_1 u_1(x) = \frac{1}{4\pi}\int_{\mathbb{R}^3} \frac{-3y_1y_3}{|y|^5}\omega_2(x-y) + \frac{3y_1y_2}{|y|^5}  \omega_3(x-y) dy \\
&\rd_1 u_2(x) = \frac{1}{3}\omega_3(x) + \frac{1}{4\pi}\int_{\mathbb{R}^3} \frac{(y_2^2-y_1^2) + (y_3^2-y_1^2)}{|y|^5}\omega_3(x-y) + \frac{3y_1y_3}{|y|^5}  \omega_1(x-y) dy \\
&\rd_1 u_3(x) = -\frac{1}{3}\omega_2(x) + \frac{1}{4\pi}\int_{\mathbb{R}^3} \frac{-3y_1y_2}{|y|^5}\omega_1(x-y) - \frac{(y_2^2-y_1^2) + (y_3^2-y_1^2)}{|y|^5}  \omega_2(x-y) dy,
\end{split}
\end{equation}\begin{equation}\label{eq:pr2_u}
\begin{split}
&\rd_2u_1(x) = -\frac{1}{3}\omega_3(x) + \frac{1}{4\pi}\int_{\mathbb{R}^3} - \frac{3y_2y_3}{|y|^5}\omega_2(x-y) - \frac{(y_1^2 - y_2^2) + (y_3^2 - y_2^2)}{|y|^5}\omega_3(x-y) dy \\
&\rd_2u_2(x) = \frac{1}{4\pi}\int_{\mathbb{R}^3} -\frac{3y_1y_2}{|y|^5}\omega_3(x-y) + \frac{3y_2y_3}{|y|^5}\omega_1(x-y) dy \\
&\rd_2u_3(x) =\frac{1}{3}\omega_1(x) + \frac{1}{4\pi}\int_{\mathbb{R}^3}\frac{(y_1^2 - y_2^2) + (y_3^2 - y_2^2)}{|y|^5}\omega_1(x-y) + \frac{3y_1y_2}{|y|^5}\omega_2(x-y) dy,
\end{split}
\end{equation}\begin{equation}\label{eq:pr3_u}
\begin{split}
&\rd_3u_1(x) = \frac{1}{3}\omega_2(x) +  \frac{1}{4\pi}\int_{\mathbb{R}^3} \frac{(y_1^2-y_3^2) + (y_2^2 - y_3^2)}{|y|^5}\omega_2(x-y) + \frac{3y_2y_3}{|y|^5}\omega_3(x-y) dy \\
&\rd_3u_2(x) = -\frac{1}{3}\omega_1(x) +\frac{1}{4\pi}\int_{\mathbb{R}^3} \frac{3y_1y_3}{|y|^5}\omega_3(x-y) - \frac{(y_1^2-y_3^2) + (y_2^2 - y_3^2)}{|y|^5}\omega_1(x-y) dy \\
&\rd_3u_3(x) = \frac{1}{4\pi}\int_{\mathbb{R}^3} \frac{-3y_2y_3}{|y|^5}\omega_1(x-y) + \frac{3y_1y_3}{|y|^5}\omega_2(x-y) dy. 
\end{split}
\end{equation}

\bibliographystyle{amsplain}
\bibliography{octant}

\end{document}